\documentclass[final,3p,times]{elsarticle}
\usepackage{amsmath}
\usepackage{amsfonts}
\usepackage{amsthm}
\usepackage[english]{babel}
\usepackage{xcolor}
\usepackage{graphicx}
\usepackage{caption}

\usepackage{tikz}
\usepackage{pgfplots}
\usepackage{subfig}
\usepackage{enumerate}
\usepackage{lmodern}
\usepackage{mathrsfs}
\usepackage{microtype}
\usepackage[title]{appendix}
\usepackage{mathscinet}
\usepackage{textcomp}
\usepackage{amssymb}
\usepackage{color}
\usepackage{latexsym}
\input{amssym.def}
 \numberwithin{equation}{section}
	\newtheorem{theorem}{Theorem}[section]
	\newtheorem{lemma}{Lemma}[section]
	\newtheorem{proposition}{Proposition}[section]
	\newtheorem{remark}{Remark}[section]
	\newtheorem{definition}{Definition}[section]
	\newtheorem{corollary}{Corollary}[section]

	 \theoremstyle{plain}

\journal{}
\allowdisplaybreaks

\begin{document}

\title{Decay characterization of solutions to semi-linear structurally damped $\sigma$-evolution equations with time-dependent damping}
\date{}

\author[1]{Cung The Anh}
\ead{anhctmath@hnue.edu.vn} 
\author[2]{Phan Duc An}
\ead{anan26042001@gmail.com}
\author[3]{Pham Trieu Duong}
\ead{duongptmath@hnue.edu.vn}

\address[1]{Cung T. A., Department of Mathematics\\  Hanoi National University of Education\\ 136 Xuan Thuy, Hanoi, Vietnam}

\address[2]{Phan D. A., Department of Mathematics\\  Hanoi National University of Educatino\\ 136 Xuan Thuy, Hanoi, Vietnam}
\address[3]{ Pham T. D., Department of Mathematics\\  Hanoi National University of Education\\ 136 Xuan Thuy, Hanoi, Vietnam}

\begin{abstract}
In this paper, we study the Cauchy problem to the linear damped $\sigma$-evolution equation with time-dependent damping in the
effective cases
\begin{equation*}
u_{t t}+(-\Delta)^\sigma u+b(t)(-\Delta)^\delta u_t=0,
\end{equation*}
and investigate the decay rates of the solution and its derivatives that are expressed in terms of the decay character of the initial data $u_0(x)=u(0, x)$ and $u_1(x)=u_t(0, x)$. We are interested also in the existence and decay rate of the global in time solution with small data for the corresponding semi-linear problem with the nonlinear term of power type $||D|^\gamma u|^p$. The blow-up results for solutions to the semi-linear problem in the case  $\gamma=0$ are presented to show the sharpness of the exponent $p$.

\end{abstract}
\maketitle



\section{Introduction}
\subsection{Motivation}
Starting with the results by Matsumura \cite{Matsumura1976}, the problem of determining the critical exponent
for the damped wave equations
\begin{equation} \label{eq:Matsumura1976}
\begin{cases}u_{t t}-\Delta u+u_t=|u|^p, & (t, x) \in[0, \infty) \times \mathbb{R}^n,\quad p>1, \\ u(0, x)=u_0(x), \quad u_t(0, x)=u_1(x), & x \in \mathbb{R}^n,\end{cases}
\end{equation}
has been investigated successfully by Todorova, Yordanov in \cite{Yordanov2001} and Zhang in \cite{Zhang1999}.

Another issue to explore is related to the linear wave equation with time-dependent damping term
\begin{equation} \label{eq:effective}
\begin{cases}u_{t t}+\left(-\Delta\right)^\sigma u+b(t) u_t=0, & (t, x) \in[0, \infty) \times \mathbb{R}^n, \quad \sigma \geq 1, \\ u(0, x)=u_0(x), u_t(0, x)=u_1(x), & x \in \mathbb{R}^n .\end{cases}
\end{equation}
Asymptotic behavior of solutions and their wave energy essentially depends on the positive coefficient $b=b(t)$ in the damping term. According to the categorization  by J. Wirth in \cite{Wirth2006,  Wirth2007}, the following types of time-dependent dissipation terms are of special interest: scattering to the free wave equation, non-effective dissipation, effective dissipation and over-damping generating. Briefly speaking, if the solution exhibits an asymptotic behavior similar to that of the wave equation, then as $t \rightarrow \infty$, it scatters towards the solution of the free wave equation. This situation is referred to as a {\it scattering-generating} case. If the $L^p-L^q$ estimates, with $1 \leq p \leq q \leq \infty$, for the solution to equation \eqref{eq:effective} are closely linked to those of the solution to the free wave equation, then the damping term is referred to as {\it non-effective}. If the solution to equation \eqref{eq:effective} displays the same decay characteristics as the corresponding parabolic Cauchy problem
$$
\begin{cases}v_t=-\frac{1}{b(t)} \left(-\Delta \right)^\sigma v, & (t, x) \in[0, \infty) \times \mathbb{R}^n, \quad \sigma \geq 1, \\ v(0, x)=v_0(x), & x \in \mathbb{R}^n,\end{cases}
$$
where a suitable data $v_0=v_0(x)$ depends on $u_0, u_1$, then the damping term is classified as {\it effective}. Meanwhile, if there are no established decay estimates for the energy of solutions, then the damping term is called {\it over-damping generating}. Several authors, such as D'Abbicco, Lucente, and Reissig in \cite{DAbbicco_Reissig2013} have succeeded in deriving the estimates for the solution of linear Cauchy problems \eqref{eq:effective} for $\sigma=1$, while Aslan and Dao \cite{Dao2023} have derived  estimates for $\sigma>1$.

In the next step, we  consider the following Cauchy problem for semi-linear damped $\sigma$-evolution equation:
\begin{equation} \label{eq:semilinear_effective}
\begin{cases}u_{t t}+\left(-\Delta \right)^\sigma u+b(t) u_t=|u|^p, & (t, x) \in[0, \infty) \times \mathbb{R}^n, \\ u(0, x)=u_0(x), \quad u_t(0, x)=u_1(x), & x \in \mathbb{R}^n ,\end{cases}
\end{equation}
with $\sigma \geq 1, p>1$. Readers can refer to \cite{DAbbicco_Reissig2013, Dao2023} (in the case of $b(t) u_t$ is effective) and Ikeda and Wakasugi \cite{Ikeda2020} (in the case of $b(t) u_t$ is over damping) for other semi-linear problems for damped wave and $\sigma$-evolution equations with time-dependent damping models. A substantial amount of important results on blow-up mechanics for nonlinear problems of damped wave models has been obtained recently by Lin, Nishihara and Zhai \cite{Nishihara2012}, D’Abbicco and Lucente \cite{DAbbicco2013}, Ikeda and Sobajima \cite{Ikeda2018}, Lai and Takamura \cite{Takamura2018}, Ikeda and Inui \cite{Ikeda2019}, Ikeda, Sobajima and Wakasugi \cite{Sobajima2019}. 

In order to study the uniform decay rates of many dissipative evolution equations with only $L^2$ data, a new idea has been introduced by Bjorland and Schonbek in \cite{Bjorland2009} to associate to a function $u_0 \in L^2\left(\mathbb{R}^n\right)$ a decay character $r^*=r^*\left(u_0\right)$ that assigns an improved decay rate to the solution of the heat equation with the datum $u_0$. The use of the decay character may lead to finer decay estimates of the solutions to dissipative evolution Cauchy problems with only $L^2$ initial data.

Next, we consider the semi-linear structurally damped $\sigma$-evolution
equation
\begin{equation} \label{ineq:semi_linear_b(t)=1}
\begin{cases}
u_{t t}+(-\Delta)^\sigma u+ (-\Delta)^\delta u_t=|u|^p, & (t, x) \in [0, \infty) \times \mathbb{R}^n, \\ u(0, x)=u_0(x), \quad u_t(0, x)=u_1(x), & x \in \mathbb{R}^n,
\end{cases}
\end{equation}
where $\sigma \geq 1, \delta \in [0, \sigma], p>1$ and $n \geq 1$. The corresponding linear Cauchy problem for \eqref{ineq:semi_linear_b(t)=1}
is
\begin{equation} \label{ineq:linear_b(t)=1}
\begin{cases}
u_{t t}+(-\Delta)^\sigma u+ (-\Delta)^\delta u_t=0, & (t, x) \in [0, \infty) \times \mathbb{R}^n, \\ u(0, x)=u_0(x), \quad u_t(0, x)=u_1(x), & x \in \mathbb{R}^n.
\end{cases}
\end{equation}

Using decay character to study the solutions of linear and semilinear equations to obtain desired decay  estimates is a relatively new approach. Cárdenas, Armando and César \cite{Armando2022} (in the case of $\sigma=1, \delta=0$) along with Anh, Duong, Loc \cite{ADL2024} (in the case of $\sigma \geq 1, \delta \in [0, \sigma]$) have successfully established solutions for both linear  equation \eqref{ineq:linear_b(t)=1} and semilinear equations \eqref{ineq:semi_linear_b(t)=1}  based on the decay character. For equation \eqref{eq:1.0}, there have been no results related to the decay character for estimating. Additionally, there have been no blow-up results related to this equation when $\gamma=0$.

In this article, we will apply the Fourier Splitting Method developed by Schonbek in \cite{Schonbek1980, Schonbek1985,  Schonbek1986} based on  the notion of decay characters to study the decay rates of solutions for linear and semi-linear Cauchy problems with time-dependent damping.  More precisely,  we will establish global existence and decay rates of global solutions with small data and derive as well the
blow-up results for the following semi-linear problem with parameters $\sigma \geq 1, \delta \in [0,\frac{\sigma}{2}], \gamma \in [0,\sigma), p>1$ and $n \geq 1$:
\begin{equation}
\begin{cases} \label{eq:1.0}
u_{t t}(t, x)+(-\Delta)^\sigma u(t, x)+b(t) (-\Delta)^\delta u_t(t, x)=||D|^{\gamma} u(t,x)|^p, & (t, x) \in [0, \infty) \times \mathbb{R}^n, \\ u(0, x)=u_0(x), \quad u_t(0, x)=u_1(x), & x \in \mathbb{R}^n.
\end{cases}
\end{equation}
The corresponding linear Cauchy problem for \eqref{eq:1.0} with $\sigma \geq 1, \delta \in [0,\frac{\sigma}{2}] $ and $n \geq 1$ is the following
\begin{equation}
\begin{cases} \label{eq:1.1}
u_{t t}(t, x)+(-\Delta)^\sigma u(t, x)+b(t) (-\Delta)^\delta u_t(t, x)=0, & (t, x) \in [0, \infty) \times \mathbb{R}^n, \\ u(0, x)=u_0(x), \quad u_t(0, x)=u_1(x), & x \in \mathbb{R}^n.
\end{cases}
\end{equation}

\subsection{Notations}

Let $P_r\left(u_0\right)$ be the decay indicator of the function $u_0 \in L^2\left(\mathbb{R}^n\right)$ (see details in Definition \ref{definition_2.1}). Roughly speaking, the indicator $P_r\left(u_0\right)$ compares the Fourier transform $\widehat{u_0}$ to the power $|\xi|^r$ near $\xi=0$. Let $r^*\left(u_0\right)$ be the decay character of the function $u_0$ (see Definition \ref{definition_2.2}). We denote by $\|h\|_{P, \eta}$ and $\|h\|_{P}$ the following norm of $h$ in $H^\eta$ and $L^2$ (see the details in Section \ref{section4.6}).
$$
\|h\|_{P, \eta}=
\|h\|_{_{\dot{H}^\eta}}+\left(P_{r^*\left(h\right)}\left(h\right)\right)^{1 / 2},
$$
$$
\|h\|_{P}=
\|h\|_{L^2}+\left(P_{r^*\left(h\right)}\left(h\right)\right)^{1 / 2}.
$$

Throughout the paper, for nonnegative functions $f(t), g(t)$ the notation $f(t) \lesssim g(t)$ is used to denote the inequalities $f(t) \leq C g(t)$ that are satisfied uniformly for all $t>0$, with a positive constant $C$.

 The spaces $H^a$ and $\dot{H}^a$ with $a \geq 0$ are fractional Sobolev spaces  (Bessel and Riesz potential spaces) based on the $L^2$ spaces. We also denote by $\langle D\rangle^a$ and $|D|^a$ the pseudo-differential operators with symbol $\langle\xi\rangle^a$ and $|\xi|^a$, respectively.

For the function $u=u(t, x)$ defined on $\left\{(t, x): 0 \leq t \in \mathbb{R}, \quad x \in \mathbb{R}^n\right\}$, for brevity we  write $\|u\|_{L^p}$ instead of $\|u\|_{L^p(\mathbb{R}^n)}$. For $s \in[0, t]$, we denote 
\begin{align*}
\mathcal{B}(s, t):=\int_s^t \frac{1}{b(\tau)} d \tau, ~~~
\widehat{\mathcal{B}}(s, t):=\int_s^t b(\tau) d \tau.
\end{align*}
\subsection{ Our main results}

\begin{theorem} \label{theorem:1.1}
Let $\sigma \geq 1, \delta \in\left[0, \frac{\sigma}{2}\right]$ and $\left(u_0, u_1\right) \in D^\sigma:=H^\sigma \times L^2$. If $-\frac{n}{2}<r^*\left(u_0\right), r^*\left(u_1\right) - 2 \delta < \infty$ and $b(t)$ satisfies assumptions from Definition \ref{definition_2.0}, then there exists a solution $u \in \mathcal{C}\left([0, \infty), H^\sigma \right) \cap \mathcal{C}^1\left([0, \infty), L^2\right)$ to \eqref{eq:1.1} 
with the following decay rates for all $\alpha\in [0, \sigma]$
\begin{enumerate}[i)]
\item for $\delta \in (0, \frac{\sigma}{2}]:$
\begin{align} \label{ineq:linear_1}
\|u(t,\cdot)\|_{\dot{H}^\alpha} \lesssim & \frac{1}{b(t)} \left(1+\mathcal{B}(0,t)\right)^{-\frac{r^{*}(u_1) + \frac{n}{2} + \alpha - 2 \delta}{2\sigma - 2\delta}} \|u_1\|_{P}+  \left(1+\mathcal{B}(0,t)\right)^{-\frac{r^{*}(u_0) + \frac{n}{2} + \alpha}{2\sigma - 2\delta}} \|u_0\|_{P, \alpha}\notag \\
 +& \frac{1}{b(t)} \left(1+t\right)^{-\frac{r^{*}(u_1) + \frac{n}{2} + \alpha - 2 \delta}{2\sigma - 2\delta}} \|u_1\|_{P} +  \left(1+t\right)^{-\frac{r^{*}(u_0) + \frac{n}{2} + \alpha}{2\sigma - 2\delta}} \|u_0\|_{P, \alpha},
\end{align}
\begin{align} \label{ineq:linear_2}
\|u_t(t,\cdot)\|_{L^2} \lesssim & \|u_1\|_{P} \left(\frac{1}{b(0)}+\frac{1}{b(t)} \right) \frac{1}{b(t)}  \left(1+\mathcal{B}(0,t)\right)^{-\frac{r^{*}(u_1) + \frac{n}{2} + 2 \sigma - 4 \delta}{2\sigma - 2\delta}} \notag\\
+& \|u_0\|_{P, \sigma} \frac{1}{b(t)} \left(1+\mathcal{B}(0,t)\right)^{-\frac{r^{*}(u_0) + \frac{n}{2} + 2 \sigma - 2 \delta}{2\sigma - 2\delta}} \notag\\
 +& \|u_1\|_{P} \left(1+\widehat{\mathcal{B}}(0,t)\right)^{-\frac{r^{*}(u_1) + \frac{n}{2} + 2 \sigma - 4 \delta}{2\sigma - 2\delta}} + \|u_0\|_{P, \sigma} \frac{1}{b(t)} \left(1+t\right)^{-\frac{r^{*}(u_0) + \frac{n}{2} + 2 \sigma - 2 \delta}{2\sigma - 2\delta}}.
\end{align}
\item for $\delta=0:$
\begin{equation} \label{ineq:linear_1.000}
\|u(t,\cdot)\|_{\dot{H}^\alpha} \lesssim \|u_1\|_{P} \left(1+\mathcal{B}(0,t)\right)^{-\frac{r^{*}(u_1) + \frac{n}{2} + \alpha}{2\sigma}} + \|u_0\|_{P, \alpha} \left(1+\mathcal{B}(0,t)\right)^{-\frac{r^{*}(u_0) + \frac{n}{2} + \alpha}{2\sigma }},
\end{equation}
\begin{equation}  \label{ineq:linear_2.000}
\|u_t(t,\cdot)\|_{L^2} \lesssim \|u_1\|_{P} \frac{1}{b(t)} \left(1+\mathcal{B}(0,t)\right)^{-\frac{r^{*}(u_1) + \frac{n}{2} + 2 \sigma}{2\sigma}} + \|u_0\|_{P, \sigma} \frac{1}{b(t)} \left(1+\mathcal{B}(0,t)\right)^{-\frac{r^{*}(u_0) + \frac{n}{2} + 2 \sigma}{2\sigma}}.
\end{equation}
\end{enumerate}
\end{theorem}
We obtain the next result on the global (in time) existence of small data solutions.
\begin{theorem} \label{theorem:1.3}Let $\sigma \geq 1, \delta \in\left[0, \frac{\sigma}{2}\right],  0 \leq \gamma < \sigma$ and $n \in \mathbb{N}^*$. Assume that $b(t)$ satisfies Definition \ref{definition_2.0} and $\left(u_0, u_1\right) \in H^\sigma \times L^2$ satisfy
\begin{align*}
-\frac{n}{2}&<\min \left\{r^*\left(u_0\right), r^*\left(u_1\right)-2 \delta\right\} \leq-2 \delta & \text{ if } b^{\prime}(t) \geq 0 \text{ and } \delta\in (0, \sigma], \\
-\frac{n}{2}&<\min \left\{r^*\left(u_0\right), r^*\left(u_1\right)-2\sigma-2 \delta\right\} \leq \min \left\{r^*\left(u_0\right), r^*\left(u_1\right)-2 \delta\right\}\le  -2\delta & \text{ if } b^{\prime}(t) < 0 \text{ and } \delta\in (0, \sigma],\\
-\frac{n}{2}&<\min \left\{r^*\left(u_0\right), r^*\left(u_1\right)\right\} \leq 0 & \text{ if } \delta=0.
\end{align*}
Let $\omega:=\frac{n}{n+\min \left\{ r^{*}(u_0), r^{*}(u_1) - 2 \delta \right\} + 2 \delta}$, and \begin{equation} \label{p_*}
p^*:=\tfrac{n + 2\omega(\sigma-\delta)}{\omega n + \omega \min \left\{ r^{*}(u_0), r^{*}(u_1) - 2\delta \right\}+\omega \gamma}.
\end{equation}
We assume that $p>p^*$  and satisfies the following condition:
$$
\begin{cases} p \in\left[\frac{2}{\omega}, \infty\right) & \text { if } n \leq 2 \sigma - 2 \gamma , \\ p \in\left[\frac{2}{\omega}, \frac{n}{n-2 \sigma+2\gamma}\right] & \text { if } 2 \sigma - 2 \gamma <n \leq \frac{4 \sigma - 4 \gamma}{2-\omega}.\end{cases}
$$
Then there exists a constant $\varepsilon>0$ such that for the initial data $\left(u_0, u_1\right) \in D^\sigma:=H^\sigma \times L^2$ with the norm $\left\|\left(u_0, u_1\right)\right\|_{D^\sigma}:=\left\|u_0\right\|_{P, \sigma}+\left\|u_1\right\|_{P}<\varepsilon$,  problem \eqref{eq:1.0} admits a unique global in time solution $u \in \mathcal{C}\left([0, \infty), H^\sigma\right) \cap \mathcal{C}^1 \left([0, \infty), L^2\right)$ with the following decay rates:\\

\noindent i) for $\delta \in (0, \frac{\sigma}{2}]$ and $b^{\prime}(t) \geq 0:$ 
\begin{align*}
\|u(t,\cdot)\|_{L^2} & \lesssim \left\|\left(u_0, u_1\right)\right\|_{D^\sigma} \left(1+\mathcal{B}(0,t)\right)^{-\frac{\frac{n}{2} + \min \left\{ r^{*}(u_0), r^{*}(u_1) - 2 \delta \right\} }{2 \sigma - 2 \delta}}, \\
\|u(t,\cdot)\|_{\dot{H}^\sigma}& \lesssim \left\|\left(u_0, u_1\right)\right\|_{D^\sigma} \left(1+\mathcal{B}(0,t)\right)^{-\frac{\frac{n}{2} + \sigma + \min \left\{ r^{*}(u_0), r^{*}(u_1) - 2 \delta \right\} }{2 \sigma - 2 \delta}}, \\
\|u_t(t,\cdot)\|_{L^2} & \lesssim \left\|\left(u_0, u_1\right)\right\|_{D^\sigma} \frac{1}{b(t)}\left(1+\mathcal{B}(0,t)\right)^{-\frac{\frac{n}{2} + 2 \sigma - 2 \delta + \min \left\{ r^{*}(u_0), r^{*}(u_1)- 2 \delta \right\} }{2 \sigma - 2 \delta}}.
\end{align*}
ii) for $\delta \in (0, \frac{\sigma}{2}]$ and $b^{\prime}(t) < 0:$
\begin{align*}
\|u(t,\cdot)\|_{L^2} & \lesssim \left\|\left(u_0, u_1\right)\right\|_{D^\sigma} \left(1+\widehat{\mathcal{B}}(0,t)\right)^{-\frac{\frac{n}{2} + \min \left\{ r^{*}(u_0), r^{*}(u_1) - 2 \delta \right\} }{2 \sigma - 2 \delta}}, \\
\|u(t,\cdot)\|_{\dot{H}^\sigma}& \lesssim \left\|\left(u_0, u_1\right)\right\|_{D^\sigma} \left(1+\widehat{\mathcal{B}}(0,t)\right)^{-\frac{\frac{n}{2} + \sigma + \min \left\{ r^{*}(u_0), r^{*}(u_1) - 2 \delta \right\} }{2 \sigma - 2 \delta}}, \\
\|u_t(t,\cdot)\|_{L^2} & \lesssim  \left\|\left(u_0, u_1\right)\right\|_{D^\sigma} b(t) \left(1+\widehat{\mathcal{B}}(0,t)\right)^{-\frac{\frac{n}{2} + 2 \sigma - 2 \delta + \min \left\{ r^{*}(u_0), r^{*}(u_1)- 2 \delta \right\} }{2 \sigma - 2 \delta}}.
\end{align*}
iii) for $\delta=0:$
\begin{align*}
\|u(t,\cdot)\|_{L^2} & \lesssim \left\|\left(u_0, u_1\right)\right\|_{D^\sigma} \left(1+\mathcal{B}(0,t)\right)^{-\frac{\frac{n}{2} + \min \left\{ r^{*}(u_0), r^{*}(u_1)\right\} }{2 \sigma}}, \\
\|u(t,\cdot)\|_{\dot{H}^\sigma}& \lesssim \left\|\left(u_0, u_1\right)\right\|_{D^\sigma} \left(1+\mathcal{B}(0,t)\right)^{-\frac{\frac{n}{2} + \sigma + \min \left\{ r^{*}(u_0), r^{*}(u_1) \right\} }{2 \sigma}}, \\
\|u_t(t,\cdot)\|_{L^2} & \lesssim  \left\|\left(u_0, u_1\right)\right\|_{D^\sigma} \frac{1}{b(t)} \left(1+\mathcal{B}(0,t)\right)^{-\frac{\frac{n}{2} + 2 \sigma+ \min \left\{ r^{*}(u_0), r^{*}(u_1) \right\} }{2 \sigma - 2 \delta}}.
\end{align*}
\end{theorem}
\begin{remark}{\rm If $b(t)\equiv 1$ (and $\gamma=0$ in \eqref{eq:1.0}), from Theorems  \ref{theorem:1.1} and  \ref{theorem:1.3}, we recover of course the corresponding results in \cite{ADL2024} for linear/semi-linear structurally damped $\sigma$-evolution equations in the case of constant coefficients.  As pointed out in \cite{ADL2024},  these results recover/improve many existing results in literature as particular cases and contain even some new ones.

For example,  if $\left(u_0, u_1\right) \in \left(H^\alpha \cap \dot{H}^{-\eta} \right) \times \left(L^2 \cap \dot{H}^{-\eta} \right)$ with $\eta \in \left( 0, \frac{n}{2} \right)$. Then,  since $|\xi|^{-\eta}\widehat{u}_0, |\xi|^{-\eta}\widehat{u}_1 \in L^2\left(\mathbb{R}^n\right)$, we have 
$$
\begin{cases}
\rho^{-2 r^*(u_0) - n} \int_{B(\rho)}\left|\widehat{u}_0(\xi)\right|^2 d \xi \leq \rho^{2 \eta - 2 r^*(u_0) - n} \int_{B(\rho)} |\xi|^{-2 \eta} \left|\widehat{u}_0(\xi)\right|^2 d \xi \leq C \rho^{2 \eta - 2 r^*(u_0) - n} , \\
\rho^{- 2 r^*(u_1) - n} \int_{B(\rho)}\left|\widehat{u}_1(\xi)\right|^2 d \xi \leq \rho^{2 \eta - 2 r^*(u_1) - n} \int_{B(\rho)} |\xi|^{-2 \eta} \left|\widehat{u}_1(\xi)\right|^2 d \xi \leq C \rho^{2 \eta - 2 r^*(u_1) - n} .
\end{cases}
$$
From the definition of decay character, we obtain that $r^*(u_0), r^*(u_1) \geq \eta - \frac{n}{2}$.  Therefore,  in Theorem 1.2,  if the data $\left(u_0, u_1\right) \in \left(H^\alpha \cap \dot{H}^{-\eta} \right) \times \left(L^2 \cap \dot{H}^{-\eta} \right)$ with $\eta \in \left( 0, \frac{n}{2} \right)$ and $\sigma=1, \delta=0, \gamma=0,  b(t)\equiv 1$, we recover the result obtained very recently by Chen and Reissig in \cite[Theorem 1]{Chen2023} for damped wave equations with data from Sobolev spaces of negative order.}
\end{remark}

The next result shows the sharpness of the exponent $p$ to \eqref{eq:1.0}.
\begin{theorem} \label{theorem:1.5}
(Blow-up). Let $\sigma \geq 1, \delta \in [0,\sigma/2], \gamma=0$, $n > 2 \delta$. Assume that $b(t)$ satisfies Definition \ref{definition_2.0} and condition \eqref{B-L} below,  and  the initial data  $u_0 \in W^{2 \delta, 1}(\mathbb{R}^n)$ and $u_1 \in L^1(\mathbb{R}^n)$ are chosen such that
\begin{equation} \label{condition_u0u1}
\begin{cases}
\displaystyle\int_{\mathbb{R}^n} \left(- \mathbb{A}_0 u_0(x) + \mathbb{B}_0 u_1(x) + (-\Delta)^\delta \left[ \mathbb{A}_0 u_0(x) \right] + (-\Delta)^\delta u_0(x)\right) dx > 0 & \text { if }~~ b^{\prime}(t) \ge 0, \\
\\
\displaystyle\int_{\mathbb{R}^n} \left(u_1(x) + b(0) (-\Delta)^\delta u_0(x)\right) dx > 0 & \text { if } ~~b^{\prime}(t) < 0,
\end{cases}
\end{equation}
where $\mathbb{B}_0:=\displaystyle\int_0^{\infty} \exp \left(-\int_0^t b(\tau) d \tau\right) d t$ and $\mathbb{A}_0:=b(0) \mathbb{B}_0 - 1$. Assume that the following condition holds
\begin{equation} \label{p_critical}
1 < p \leq 1+\frac{2 \sigma}{n-2\delta}.
\end{equation}
Then, there is no global (in time) Sobolev solution $u \in \mathcal{C}\left([0, \infty), L^2 (\mathbb{R}^n)\right)$ to problem \eqref{eq:1.0}.    
\end{theorem}
This article is organized as follows. In Section \ref{section_2} we recall the definition of the effective dissipation. In Section \ref{section_3} we study the corresponding linear equation. In Section \ref{section_4} we present the detailed proof of Theorem \ref{theorem:1.1} to obtain the linear decay estimates for problem \eqref{eq:1.1}. In Section \ref{section_5} the existence of global small data solutions for semi-linear problem \eqref{eq:1.0} stated in Theorem \ref{theorem:1.3} is proved. Finally, Section \ref{section_6} presents the proof of the blow-up results for \eqref{eq:1.0}. In Appendices we present the definition and some basic properties of the decay characters of functions, as well as some useful inequalities.
\section{Preliminaries}\label{section_2}
\subsection{Effective dissipation} \label{section_2.1}
\begin{definition}[Effective dissipation] \label{definition_2.0}
If the strictly positive function $b=b(t)$ satisfies

$\mathbf{(B1)}$ $b \in \mathcal{C}^3([0, \infty))$;

$\mathbf{(B2)}  \left|b^{\prime}(t)\right|=\begin{cases}
o\left( ( 1+t )^{-\frac{2\delta}{\sigma}} b^2(t)\right), &\text{i.e.} ~~ t^{\frac{\sigma-2\delta}{\sigma}} b(t) \rightarrow \infty ~\text{as} ~t \rightarrow \infty~~\text{ if}  ~~\delta \in (0, \frac{\sigma}{2}),
\\
o\left(b^2(t)\right),& \text{i.e}~ t b(t) \rightarrow \infty ~\text{as} ~t \rightarrow \infty~\text{ if}~\delta=0 ~\text{or}~\delta=\frac{\sigma}{2};
\end{cases}$

$\mathbf{(B3)}$ $\frac{\left|b^{(k)}(t)\right|}{b(t)} \lesssim \frac{1}{(1+t)^k}$ for $k=1,2$;

$\mathbf{(B4)}$ $\frac{1}{b(t)} \notin L^1([0, \infty))$;

$\mathbf{(B5)}$ $\left((1+t)^2 b(t)\right)^{-1} \in L^1([0, \infty))$;

$\mathbf{(B6)}$ $b^{\prime}(t)$ does not change its sign,\\
then the damping term $b(t) \left(-\Delta \right)^{\delta}u_t$ is called effective.
\end{definition}
From assumption $\mathbf{(B2)}$, it follows that $t b(t) \rightarrow \infty$ as $t \rightarrow \infty$, hence $b(t) \notin L^1([0, \infty))$.  Conditions $\mathbf{(B1)}$-$\mathbf{(B6)}$ will be useful later  in Sections \ref{section_3} - \ref{section_5}. To prove the blow-up results  in Section \ref{section_6} we introduce another important condition on $b=b(t)$:
\begin{equation} \label{B-L}
\mathbf{(B-L)} \quad \mathbb{B}_{\infty}:=\limsup _{t \rightarrow \infty}\left|\frac{b^{\prime}(t)}{b^2(t)}\right|<1.
\end{equation}
\begin{lemma}[\cite{DAbbicco_Reissig2013}] \label{lemmaB} Conditions from Definition \ref{definition_2.0}
imply that $\mathcal{B}(0, t)$ is positive, strictly increasing and $\mathcal{B}(0, t) \rightarrow$ $\infty$ as $t \rightarrow \infty$. In addition, the function $\mathcal{B}(s, t)$ satisfies the followings:
$$
\begin{aligned}
& \mathcal{B}(s, t) \approx \frac{t}{b(t)}-\frac{s}{b(s)} \text { for all } s \in[0, t], \\
& \mathcal{B}(s, t) \approx \mathcal{B}(0, t) \text { for all } s \in[0, t / 2], \\
& \mathcal{B}(0, s) \approx \mathcal{B}(0, t) \text { for all } s \in[t / 2, t] .
\end{aligned}
$$
\end{lemma}
\begin{lemma} \label{lemmaB^} 
The function $\widehat{\mathcal{B}}(0, t)$ is positive, strictly increasing and $\widehat{\mathcal{B}}(0, t) \rightarrow$ $\infty$ as $t \rightarrow \infty$. Moreover,
$$
\begin{aligned}
& \widehat{\mathcal{B}}(s, t) \approx tb(t)-sb(s) \text { for all } s \in[0, t], \\
& \widehat{\mathcal{B}}(s, t) \approx \widehat{\mathcal{B}}(0, t) \text { for all } s \in[0, t / 2], \\
& \widehat{\mathcal{B}}(0, s) \approx \widehat{\mathcal{B}}(0, t) \text { for all } s \in[t / 2, t] .
\end{aligned}
$$
\end{lemma}
\begin{proof} This lemma is proved similarly as  Lemma \ref{lemmaB} (see \cite{DAbbicco_Reissig2013}).
\end{proof}
\begin{remark}{\rm 
From Lemmas \ref{lemmaB} and \ref{lemmaB^} we have
\begin{align*}
& 1+t \approx b(t) \left( 1+\mathcal{B}(0, t) \right) \approx b(t)^{-1} \left(1+\widehat{\mathcal{B}}(0, t)\right),  \forall t\geq 0,\\
& 1+\mathcal{B}(0, t) \approx \frac{1}{b^2(t)} \left(1+\widehat{\mathcal{B}}(0, t)\right), \forall t\geq 0.
\end{align*}}
\end{remark}

\begin{lemma}[\cite{Sobajima2019}, \cite{Nishihara2012}] \label{lemma:b(t)g(t)}
Let us consider the initial value problem for the ODE
\begin{align} \label{eq:5.0}
\left\{\begin{array}{l}
-g^{\prime}(t)+b(t) g(t)=1, \quad t>0, \\
g(0)=\mathbb{B}_0,
\end{array}\right.
\end{align}
where the constant $\mathbb{B}_0$ is defined in Theorem \ref{theorem:1.5}. Then, the solution to \eqref{eq:5.0} satisfies the following properties:
\begin{enumerate}[i)]
\item There exist positive constants $T_0, \mathbb{B}_1$ and $\mathbb{B}_2$, such that  
$\mathbb{B}_1 \leq b(t) g(t) \leq \mathbb{B}_2$ for any $t \geq T_0$;
\item There exists a positive constant $T_1$, such that 
$
\left|g^{\prime}(t)\right| \leq \frac{1+\mathbb{B}_{\infty}}{1-\mathbb{B}_{\infty}}$ for any $t \geq T_1$.
\end{enumerate}
\end{lemma}
Next, let us recall some well-known facts on the radial function $\psi=\psi(x):=\langle x\rangle^{-r}$ for any $r>0$, with $\langle x\rangle:=\left(1+|x|^2\right)^{\frac{1}{2}}$ for all $x \in \mathbb{R}^n$.
\begin{lemma}[Lemma 9.2.1 in \cite{Dao2020}] \label{lemma:2.10.000}
Let $q>0$. Then $
\left|\partial_x^\alpha\langle x\rangle^{-q}\right| \lesssim\langle x\rangle^{-q-|\alpha|} \quad \text { for all } x \in \mathbb{R}^n 
$, for any multi-index $\alpha$ with $|\alpha| \geq 1$.
\end{lemma}
\begin{lemma}[Lemma 9.2.2 in \cite{Dao2020}]  \label{lemma:2.10.00}

Let $m \in \mathbb{Z}, s \in(0,1)$ and $\gamma:=m+s$. If $v \in H^{2 \gamma}\left(\mathbb{R}^n\right)$, then it holds
$$
(-\Delta)^\gamma v(x)=(-\Delta)^m\left((-\Delta)^s v(x)\right)=(-\Delta)^s\left((-\Delta)^m v(x)\right).
$$
\end{lemma}
\begin{lemma}[Lemma 9.2.3 in \cite{Dao2020}]  \label{lemma:2.10.0}

Let $q>0, m \in \mathbb{Z}, s \in(0,1)$ and $\gamma:=m+s$. Then for all $x \in \mathbb{R}^n$:
$$
\left|(-\Delta)^\gamma\langle x\rangle^{-q}\right| \lesssim \begin{cases}\langle x\rangle^{-q-2 \gamma} & \text { if } \quad 0<q+2 m<n, \\ \langle x\rangle^{-n-2 s} \log (e+|x|) & \text { if } q+2 m=n, \\ \langle x\rangle^{-n-2 s} & \text { if } q+2 m>n .\end{cases}
$$
\end{lemma}
\begin{lemma}[Lemma 9.2.4 in \cite{Dao2020}]  \label{lemma:2.11.0} Let $s \in(0,1)$. Let $\psi$ be a smooth function satisfying $\partial_x^2 \psi \in L^{\infty}$. For any $R>0$, let $\psi_R$ be a function defined by
$
\psi_R(x):=\psi\left(R^{-1} x\right)
$
for $x \in \mathbb{R}^n$. Then
$$
(-\Delta)^s\left(\psi_R\right)(x)=R^{-2 s}\left((-\Delta)^s \psi\right)\left(R^{-1} x\right), \;\;\forall x\in \mathbb{R}^n.
$$
\end{lemma}
\begin{lemma}[Lemma 7 in \cite{Dao2021}]  \label{lemma:2.12}
Let $s \in \mathbb{R}$. Assume that $\mu_1=\mu_1(x) \in H^s$ and $\mu_2=\mu_2(x) \in H^{-s}$. Then
$$
\int_{\mathbb{R}^n} \mu_1(x) \mu_2(x) d x=\int_{\mathbb{R}^n} \widehat{\mu}_1(\xi) \hat{\mu}_2(\xi) d \xi.
$$   
\end{lemma}

\section{The linear equation}
\label{section_3}

First we study the corresponding linear Cauchy problem for \eqref{eq:1.0} with $\sigma \geq 1, \delta \in [0, \frac{\sigma}{2}] $ and $n \geq 1$\begin{equation}
\begin{cases} \label{eq:2.1}
u_{t t}(t, x)+(-\Delta)^\sigma u(t, x)+b(t) (-\Delta)^\delta u_t(t, x)=0, & (t, x) \in [0, \infty) \times \mathbb{R}^n, \\ u(0, x)=u_0(x), \quad u_t(0, x)=u_1(x), & x \in \mathbb{R}^n.
\end{cases}
\end{equation}
The Fourier transformation w.r.t. $x$ to \eqref{eq:2.1}, with $\widehat{u}=\widehat{u}(t, \xi)=\mathcal{F}_{x \rightarrow \xi}(u(t, x))(t, \xi)$, yields
\begin{equation}
\begin{cases} \label{eq:2.2}
\widehat{u}_{t t}+|\xi|^{2 \sigma} \widehat{u}+b(t) |\xi|^{2 \delta} \widehat{u}_t=0, & (t, \xi) \in[0,\infty) \times \mathbb{R}^n, \\ \widehat{u}(0, \xi)= \widehat{u}_0(\xi), \quad \widehat{u}_t(0, \xi)=\widehat{u}_1(\xi), & \xi \in \mathbb{R}^n .\end{cases}
\end{equation}
Putting
$
\widehat{u}(t, \xi)=\exp \left(-\frac{|\xi|^{2 \delta}}{2} \int\limits_0^t b(\tau) d \tau\right) \widehat{v}(t, \xi),
$
we can rewrite Cauchy problem \eqref{eq:2.2}  as
\begin{equation} \label{eq:2.3}
\begin{cases}\widehat{v}_{t t}+m(t, \xi) \widehat{v}=0, & (t, \xi) \in[0, \infty) \times \mathbb{R}^n, \\ \widehat{v}(0, \xi)= \widehat{u}_0( \xi), & \xi \in \mathbb{R}^n, \\ \widehat{v}_t(0, \xi)=\frac{b(0)}{2}|\xi|^{2 \delta} \widehat{u}_0( \xi)+\widehat{u}_1( \xi), & \xi \in \mathbb{R}^n,
\end{cases}
\end{equation}
where the coefficient $m=m(t, \xi)$ is defined by
\begin{equation} \label{m(t,xi)}
m(t, \xi):=|\xi|^{2 \sigma}-\frac{|\xi|^{4 \delta}}{4} b^2(t)-\frac{|\xi|^{2 \delta}}{2} b^{\prime}(t) .
\end{equation}
We introduce the curve $\Gamma:=\left\{(t, \xi) \in[0, \infty) \times \mathbb{R}^n:|\xi|^{\sigma-2 \delta}=\frac{1}{2} b(t)\right\}$ which divides the extended phase space   into two following regions
\begin{equation*}
\begin{cases}
\Pi_{\text {hyp }}=\left\{(t, \xi) \in[0, \infty) \times \mathbb{R}^n:|\xi|^{\sigma-2 \delta} > \frac{1}{2} b(t)\right\}, \\ 
\Pi_{\mathrm{ell}}=\left\{(t, \xi) \in[0, \infty) \times \mathbb{R}^n:|\xi|^{\sigma-2 \delta} < \frac{1}{2} b(t)\right\}.
\end{cases}
\end{equation*} 
The following auxiliary weight function has a crucial role in our analysis
\begin{equation} \label{xi_b(t)}
\langle\xi\rangle_{b(t)}:=\sqrt{\left| |\xi|^{2 \sigma}-\frac{|\xi|^{4 \delta} b^2(t)}{4} \right|}.
\end{equation}
We will focus on the case $\delta \in (0, \frac{\sigma}{2}]$ (the case $\delta=0$ has been studied previously in \cite{Dao2023}). We will perform a further division of both regions of the extended phase space into following smaller zones:
$$
\begin{aligned}
Z_{\mathrm{hyp}}(N) & =\left\{(t, \xi) \in[0, \infty) \times \mathbb{R}^n:\langle\xi\rangle_{b(t)} \geq N \frac{|\xi|^{2\delta}b(t)}{2}\right\} \cap \Pi_{\mathrm{hyp}}, \\
Z_{\mathrm{pd}}(N, \varepsilon) & =\left\{(t, \xi) \in[0, \infty) \times \mathbb{R}^n: \varepsilon \frac{|\xi|^{2\delta}b(t)}{2} \leq \langle\xi\rangle_{b(t)} \leq N \frac{|\xi|^{2\delta}b(t)}{2}\right\} \cap \Pi_{\mathrm{hyp}}, \\
Z_{\mathrm{red}}(\varepsilon) & =\left\{(t, \xi) \in[0, \infty) \times \mathbb{R}^n:\langle\xi\rangle_{b(t)} \leq \varepsilon \frac{|\xi|^{2\delta}b(t)}{2}\right\}, \\
Z_{\mathrm{ell}}\left(\varepsilon, d_0\right) & =\left\{(t, \xi) \in[0, \infty) \times \mathbb{R}^n:\langle\xi\rangle_{b(t)} \geq \varepsilon \frac{|\xi|^{2\delta}b(t)}{2}\right\} \cap \left\{|\xi|^{2\delta} \geq \frac{d_0}{(1+t)b(t)}\right\} \cap \Pi_{\mathrm{ell}}, \\
Z_{\mathrm{diss}}\left(d_0\right) &=\left\{(t, \xi) \in[0, \infty) \times \mathbb{R}^n:|\xi|^{2\delta} \leq \frac{d_0}{(1+t)b(t)}\right\} \cap \Pi_{\text {ell }}.
\end{aligned}
$$

The assumption $\mathbf{(B2)}$ ensures that the elliptic zone is not empty if $\delta \in (0,\frac{\sigma}{2})$.  If $\delta = \frac{\sigma}{2}$, the elliptic zone is naturally determined without the need for assumption $\mathbf{(B2)}$.

Generally, $N$ is a large positive constant and $\varepsilon$ is a small positive constant, which will be chosen later. The separating curves between these zones are defined as follows:

- Curve $t_{\text {diss }}=t_{\text {diss }}(|\xi|)$ separates the zones $Z_{\mathrm{diss}}\left(d_0\right)$ and $Z_{\mathrm{ell}}(\varepsilon, d_0)$;

- Curve $t_{\text {ell }}=t_{\text {ell }}(|\xi|)$ separates the zones $Z_{\mathrm{ell}}\left(\varepsilon, d_0\right)$ and $Z_{\mathrm{red}}(\varepsilon)$;

- Curve $t_{\mathrm{red}}=t_{\mathrm{red}}(|\xi|)$  separates the zones $Z_{\mathrm{red}}(\varepsilon)$ and $Z_{\mathrm{pd}}(N, \varepsilon)$;

- Curve $t_{\mathrm{pd}}=t_{\mathrm{pd}}(|\xi|)$ separates the zones $Z_{\mathrm{pd}}(N, \varepsilon)$ and $Z_{\mathrm{hyp}}(N)$.

The estimates for the corresponding
fundamental solutions in zones $Z_{\mathrm{hyp}}(N), Z_{\mathrm{pd}}(N, \varepsilon), Z_{\mathrm{red}}(\varepsilon), Z_{\mathrm{ell}}\left(\varepsilon, d_0\right)$ can be processed almost identically as those for the case $\delta=0$ obtained in \cite{DAbbicco2013,  Wirth2004, Wirth2007} and especially recently in \cite{Dao2023}  (by considering instead of $b(t)$ the "new coefficient" $\tilde{b}(t):=|\xi|^{2\delta} b(t)$). Thus we will omit the details and sketch only the main steps of the estimates in these zones and  focus only on the important estimates in $Z_{\mathrm{diss}}\left(d_0\right)$, since only in that case with small frequencies $|\xi|$ near $0$, the term $|\xi|^{2 \sigma}-\frac{|\xi|^{4 \delta}}{4} b^2(t)$ cannot be considered anymore as the principal part of the term $m(t, \xi)$.
\subsubsection{Estimates in the hyperbolic zone}
We see that $\langle\xi\rangle_{b(t)} \sim|\xi|^\sigma$ in  $Z_{\text {hyp }}(N)$.
By considering the micro-energy $V=\left(\langle\xi\rangle_{b(t)} \hat{v}, D_t \hat{v}\right)^{\mathrm{T}}$, the asymptotic behavior of the fundamental solution $E_{\text {hyp }}^V=E_{\text {hyp }}^V(t, s, \xi)$ is given by the following statement.

\begin{lemma} \label{lemma2.1}
The following estimate holds for the fundamental solution $E_{\text {hyp }}^V=E_{\text {hyp }}^V(t, s, \xi)$ with $(s, \xi),(t, \xi) \in Z_{\mathrm{hyp}}$ and $t \geq s$:
$$
\left(\left|E_{\text {hyp }}^V(t, s, \xi)\right|\right) \lesssim\left(\begin{array}{ll}
1 & 1 \\
1 & 1
\end{array}\right) .
$$
\end{lemma}

\subsubsection{Estimates in the elliptic zone}
In $Z_{\text {ell }}$ we use the micro-energy $V=\left(\langle\xi\rangle_{b(t)} \hat{v}, D_t \widehat{v}\right)^{\mathrm{T}}$ for all $t \geq s$ and $(t, \xi),(s, \xi) \in Z_{\text {ell }}\left(\varepsilon, d_0\right)$. The corresponding first-order system of Cauchy problem \eqref{eq:2.3} is the following
\begin{equation*}
\mathrm{D}_t V=\left(\begin{array}{cc} 0
& \langle\xi\rangle_{b(t)} \\
- \langle\xi\rangle_{b(t)} & 0
\end{array}\right) V +\left(\begin{array}{cc}
\frac{\mathrm{D}_t\langle\xi\rangle_{b(t)}}{\langle\xi\rangle_{b(t)}} & 0 \\
- \frac{ |\xi|^{2 \delta} b^{\prime}(t)}{2\langle\xi\rangle_{b(t)}} & 0
\end{array}\right) V.
\end{equation*}

{\it Step 1. Diagonalization procedure:} In order to prove estimates and structural properties for the fundamental solution $E_{\text {ell }}^V=$ $E_{\mathrm{ell}}^V(t, s, \xi)$ corresponding to the micro-energy $V$, we applying the diagonalization procedure, to obtain after the second step of the diagonalization that the entries of the remainder matrix are uniformly integrable over the elliptic zone. Arguing as in \cite[Section 4.2.2]{Wirth2004} we can prove the following lemma.
\begin{lemma} \label{lemma 2.2}
The fundamental solution $E_{\mathrm{ell}}^V=E_{\mathrm{ell}}^V(t, s, \xi)$ can be estimated by
$$
\left(\left|E_{\text {ell }}^V(t, s, \xi)\right|\right) \lesssim \frac{\langle\xi\rangle_{b(t)}}{\langle\xi\rangle_{b(s)}} \exp \left(\int\limits_s^t\langle\xi\rangle_{b(\tau)} d \tau\right)\left(\begin{array}{ll}
1 & 1 \\
1 & 1
\end{array}\right),
$$
with $(t, \xi),(s, \xi) \in Z_{\mathrm{ell}}\left(\varepsilon, d_0\right)$ and $0 \leq s \leq t$.
\end{lemma}
{\it Step 2. Transforming back to the original Cauchy problem:} After obtaining estimates for $E_{\text {ell }}^V=E_{\text {ell }}^V(t, s, \xi)$ it is sufficient to apply the backward transformation to the original Cauchy problem. This means that we transform back $E_{\text {ell }}^V=E_{\text {ell }}^V(t, s, \xi)$ to estimate the fundamental solution $E_{\text {ell }}=E_{\text {ell }}(t, s, \xi)$ which is related to a first-order system for the micro-energy $\left(|\xi|^\sigma \widehat{u}, D_t \hat{u}\right)^{\mathrm{T}}$ and gives the representation
\begin{equation} \label{ell}
E_{\mathrm{ell}}(t, s, \xi)=T(t, \xi) E_{\mathrm{ell}}^V(t, s, \xi) T^{-1}(s, \xi),
\end{equation}
where the matrix $T(t, \xi)$ and its inverse matrix $T^{-1}(t, \xi)$ are given in the following way:
\begin{equation*}
\left(\begin{array}{c}
|\xi|^\sigma \hat{u} \\
D_t \hat{u}
\end{array}\right)=\underbrace{\left(\begin{array}{cc}
\frac{|\xi|^\sigma}{\lambda(t) h(t, \xi)} & 0 \\
i \frac{|\xi|^{2 \delta}b(t)}{2 \lambda(t) h(t, \xi)} & \frac{1}{\lambda(t)}
\end{array}\right)}_{T(t, \xi)}\left(\begin{array}{c}
h(t, \xi) \hat{v} \\
D_t \hat{v}
\end{array}\right), \
T^{-1}(t, \xi)=\left(\begin{array}{cc}
\frac{\lambda(t) h(t, \xi)}{|\xi|^\sigma} & 0 \\
-i \frac{|\xi|^{2 \delta} b(t) \lambda(t)}{2|\xi|^\sigma} & \lambda(t)
\end{array}\right),
\end{equation*}
where the auxiliary function $\lambda=\lambda(t)$ is given by
\begin{equation} \label{lamda}
\lambda(t):=\exp \left(\frac{|\xi|^{2 \delta}}{2} \int\limits_0^t b(\tau) d \tau\right).
\end{equation}
\begin{lemma} \label{lemma2.3} In the elliptic zone it holds that
\begin{equation*}
\langle\xi\rangle_{b(t)}-\frac{|\xi|^{2 \delta} b(t)}{2} \leq-\frac{|\xi|^{2 \sigma}}{|\xi|^{2 \delta} b(t)},
\end{equation*}
\begin{equation*}
\frac{\lambda(s)}{\lambda(t)} \exp \left(\int\limits_s^t\langle\xi\rangle_{b(\tau)} d \tau\right) \leq \exp \left(-|\xi|^{2 \sigma - 2 \delta} \int\limits_s^t \frac{1}{b(\tau)} d \tau\right),
\end{equation*}
where $\lambda=\lambda(t)$ is defined in \eqref{lamda}.
\end{lemma}
{\it Step 3. A refined estimate for the fundamental solution in the elliptic zone:} From Lemma \ref{lemma 2.2} we get for $(t, \xi),(s, \xi) \in Z_{\text {ell }}\left(\varepsilon, d_0\right)$, the estimate
$$
\left(\left|E_{\mathrm{ell}}^V(t, s, \xi)\right|\right) \lesssim \frac{b(t)}{b(s)} \exp \left(\int\limits_s^t\langle\xi\rangle_{b(\tau)} d \tau\right)\left(\begin{array}{ll}
1 & 1 \\
1 & 1
\end{array}\right) .
$$
This yields in combination with \eqref{ell} and Lemma \ref{lemma2.3} the estimate
\begin{align} \label{ell_1}
\left(\left|E_{\mathrm{ell}}(t, s, \xi)\right|\right) & \lesssim\left(\begin{array}{cc}
\frac{|\xi|^\sigma}{|\xi|^{2 \delta}} & 0 \\
 b(t) & b(t)
\end{array}\right) \exp \left(\int\limits_s^t\left(\langle\xi\rangle_{b(\tau)}-\frac{|\xi|^{2 \delta} b(\tau)}{2}\right) d \tau\right)\left(\begin{array}{ll}
1 & 1 \\
1 & 1
\end{array}\right)\left(\begin{array}{cc}
\frac{|\xi|^{2 \delta}}{|\xi|^\sigma} & 0 \\
\frac{|\xi|^{2 \delta}}{|\xi|^\sigma} & \frac{1}{b(s)}
\end{array}\right) \notag \\
& \lesssim \exp \left(-|\xi|^{2 \sigma - 2 \delta} \int\limits_s^t \frac{1}{b(\tau)} d \tau\right)\left(\begin{array}{cc}
1 & \frac{|\xi|^\sigma}{|\xi|^{2 \delta} b(s)} \\
\frac{|\xi|^{2 \delta} b(t)}{|\xi|^\sigma} & \frac{b(t)}{b(s)}
\end{array}\right).
\end{align}
It should be noted that among the two matrices $T(t, \xi)$ and $T^{-1}(t, \xi)$, we choose $h(t, \xi)=|\xi|^{2 \delta} b(t)$.
\begin{lemma} \label{lemma2.4}
The fundamental solution $E_{\mathrm{ell}}=E_{\mathrm{ell}}(t, s, \xi)$ satisfies the following estimate:
\begin{equation*}
\left(\left|E_{\mathrm{ell}}(t, s, \xi)\right|\right) \lesssim \exp \left(-|\xi|^{2 \sigma - 2 \delta} \int\limits_s^t \frac{1}{b(\tau)} d \tau\right)\left(\begin{array}{cc}
1 & \frac{|\xi|^{\sigma - 2 \delta}}{b(s)} \\
\frac{|\xi|^{\sigma - 2 \delta}}{b(t)} & \frac{|\xi|^{2 \sigma - 4 \delta}}{b(s) b(t)}
\end{array}\right)+\frac{\lambda^2(s)}{\lambda^2(t)}\left(\begin{array}{ll}
0 & 0 \\
0 & 1
\end{array}\right)
\end{equation*}
for all $t \geq s$ and $(t, \xi),(s, \xi) \in Z_{\mathrm{ell}}\left(\varepsilon, d_0\right)$.
\end{lemma}
\begin{proof}
If $\Phi_k(t, s, \xi), k=1,2$, solves $\Phi_{t t}+b(t)|\xi|^{2 \delta} \Phi_t+|\xi|^{2 \sigma} \Phi=0$, with  $\Phi_k(s, s, \xi)=\delta_{1 k},\;\partial_t \Phi_k(s, s, \xi)=\delta_{2 k}$, then
$$
\left(\begin{array}{c}
|\xi|^\sigma \hat{u}(t, \xi) \\
D_t \hat{u}(t, \xi)
\end{array}\right)=\left(\begin{array}{cc}
\Phi_1(t, s, \xi) & i|\xi|^\sigma \Phi_2(t, s, \xi) \\
\frac{D_t \Phi_1(t, s, \xi)}{|\xi| \sigma} & i D_t \Phi_2(t, s, \xi)
\end{array}\right)\left(\begin{array}{c}
|\xi|^\sigma \hat{u}(s, \xi) \\
D_t \hat{u}(s, \xi)
\end{array}\right) .
$$
Hence, it follows from \eqref{ell_1}
$$
\begin{aligned}
\left|\Phi_1(t, s, \xi)\right| & \lesssim \exp \left(- \int_s^t \frac{|\xi|^{2 \sigma-2 \delta}}{b(\tau)} d \tau\right), \;\;
\left|\Phi_2(t, s, \xi)\right|  \lesssim \frac{1}{b(s)|\xi|^{2 \delta}} \exp \left(- \int_s^t \frac{|\xi|^{2 \sigma-2 \delta}}{b(\tau)} d \tau\right), \\
\left|\partial_t \Phi_1(t, s, \xi)\right| & \lesssim b(t)|\xi|^{2 \delta} \exp \left(- \int_s^t \frac{|\xi|^{2 \sigma-2 \delta}}{b(\tau)} d \tau\right), \;\;
\left|\partial_t \Phi_2(t, s, \xi)\right|  \lesssim \frac{b(t)}{b(s)} \exp \left(- \int_s^t \frac{|\xi|^{2 \sigma-2 \delta}}{b(\tau)} d \tau\right).
\end{aligned}
$$
Let $\Psi_k(t, s, \xi)=\partial_t \Phi_k(t, s, \xi),  k=1,2$. Then $\partial_t \Psi_k+b(t)|\xi|^{2 \delta} \Psi_k=-|\xi|^{2 \sigma} \Phi_k, \quad \Psi_k(s, s, \xi)=$ $\delta_{2 k}$. Standard calculations lead to
$$
\begin{aligned}
 \Psi_1(t, s, \xi)=-|\xi|^{2 \sigma} \int_s^t \frac{\lambda^2(\tau)}{\lambda^2(t)} \Phi_1(\tau, s, \xi) d \tau, \;\;
\Psi_2(t, s, \xi)=\frac{\lambda^2(s)}{\lambda^2(t)}-|\xi|^{2 \sigma} \int_s^t \frac{\lambda^2(\tau)}{\lambda^2(t)} \Phi_2(\tau, s, \xi) d \tau.
\end{aligned}
$$
If we are able to derive the desired estimate for $\left|\Psi_1(t, s, \xi)\right|$, then we immediately  conclude the desired estimate for $\left|\Psi_2(t, s, \xi)\right|$. Using the estimates for $\left|\Phi_1(t, s, \xi)\right|$ and applying integration by parts we get
$$
\begin{aligned}
\left|\Psi_1(t, s, \xi)\right| \lesssim & \underbrace{\frac{|\xi|^{2\sigma}}{\lambda^2(t)} \int_s^t \lambda^2(\tau) \exp \left(-|\xi|^{2\sigma-2\delta} \int_s^\tau \frac{1}{b(\theta)} d \theta\right) d \tau}_{I} \\
= & \frac{|\xi|^{2\sigma-2\delta}}{\lambda^2(t)} \int_s^t \frac{1}{b(\tau)} \exp \left(-|\xi|^{2\sigma-2\delta} \int_s^\tau \frac{1}{b(\theta)} d \theta\right) \partial_\tau \lambda^2(\tau) d \tau \\
= & \frac{|\xi|^{2\sigma-2\delta}}{\lambda^2(t)}\left( \frac{\lambda^2(\tau)}{b(\tau)} \exp \left(-|\xi|^{2\sigma-2\delta} \int_s^\tau \frac{1}{b(\theta)} d \theta\right)\right)_s^t \\
&+\frac{|\xi|^{2\sigma}}{\lambda^2(t)} \int_s^t \lambda^2(\tau)\left(\frac{b(\tau)^{\prime}}{|\xi|^{2\delta} b^2(\tau)}+\frac{|\xi|^{2\sigma-4\delta}}{b^2(\tau)}\right) \exp \left(-|\xi|^{2\sigma-2\delta} \int_s^\tau \frac{1}{b(\theta)} d \theta\right) d \tau \\
\leq & \frac{|\xi|^{2\sigma-2\delta}}{b(t)} \exp \left(-|\xi|^{2\sigma-2\delta} \int_s^t \frac{1}{b(\tau)} d \tau\right)-\frac{|\xi|^{2\sigma-2\delta}}{b(s)} \frac{\lambda^2(s)}{\lambda^2(t)}\\
&+\underbrace{\left(\frac{C_1}{d_0}+\frac{1}{4}\right)}_{C_\tau < 1} \underbrace{\frac{|\xi|^{2\sigma}}{\lambda^2(t)} \int_s^t \lambda^2(\tau) \exp \left(-|\xi|^{2\sigma-2\delta} \int_s^\tau \frac{1}{b(\theta)} d \theta\right) d \tau}_{I}\\
\lesssim & \frac{|\xi|^{2\sigma-2\delta}}{b(t)} \exp \left(-|\xi|^{2\sigma-2\delta} \int_s^t \frac{1}{b(\tau)} d \tau\right)-\frac{|\xi|^{2\sigma-2\delta}}{b(s)} \frac{\lambda^2(s)}{\lambda^2(t)}.
\end{aligned}
$$
Here we use $|\xi|^{2\delta} \geq \frac{d_0}{(1+\tau)b(\tau)}, |\xi|^{2\sigma-4\delta} \leq \frac{b^2(\tau)}{4}\left(1-\varepsilon^2\right)$ from the definition of the elliptic zone and $\frac{b(\tau)^{\prime}}{b(\tau)} \leq \frac{|b(\tau)^{\prime}|}{b(\tau)} \leq \frac{C_1}{1+\tau}$ from assumption $\textbf{(B6)}$, if  $d_0>\frac{4C_1}{3}$. The second term  is dominated the first one, since $$
\frac{b(s)}{b(t)} \exp \left(-|\xi|^{2\sigma-2\delta} \int_s^t \frac{1}{b(\tau)} d \tau\right) \frac{\lambda^2(t)}{\lambda^2(s)}=\exp \left(\int_s^t(\underbrace{|\xi|^{2\delta} b(\tau)-\frac{|\xi|^{2\sigma-2\delta}}{b(\tau)}-\frac{b^{\prime}(\tau)}{b(\tau)}}_{>0, \text { if } \tau \geq t_0}) d \tau\right)
$$
for $t_0 \leq s \leq t$ with $t_0$ sufficiently large. Thus, we get
$
\left|\Phi_t^1(t, s, \xi)\right| \lesssim \frac{|\xi|^{2\sigma-2\delta}}{b(t)} \exp \left(-|\xi|^{2\sigma-2\delta} \int_s^t \frac{1}{b(\tau)} d \tau\right)$.
Similarly, we can represent $\Psi_2$ in the following way:
$$
\begin{aligned}
\Psi_2(t, s, \xi) & =\frac{\lambda^2(s)}{\lambda^2(t)}-|\xi|^{2\sigma} \int_s^t \frac{\lambda^2(\tau)}{\lambda^2(t)} \Phi_2(\tau, s, \xi) d \tau, \\
\left|\Psi_2(t, s, \xi)\right| & \lesssim \frac{\lambda^2(s)}{\lambda^2(t)}+\frac{|\xi|^{2\sigma-2\delta}}{b(s) \lambda^2(t)} \int_s^t \lambda^2(\tau) \exp \left(-|\xi|^{2\sigma-2\delta} \int_s^\tau \frac{1}{b(\theta)} d \theta\right) d \tau \\
& \lesssim \frac{\lambda^2(s)}{\lambda^2(t)}+\frac{|\xi|^{2\sigma-4\delta}}{b(t) b(s)} \exp \left(-|\xi|^{2\sigma-2\delta} \int_s^t \frac{1}{b(\tau)} d \tau\right).
\end{aligned}
$$
Thus, all desired estimates have been verified.
\end{proof}

\subsubsection{Estimates in the dissipative zone}
In the dissipative zone we define the micro-energy $U=U(t, \xi):=\left(\gamma(t) \hat{u}, D_t \hat{u}\right)^{\mathrm{T}}, \; \gamma(t):=\frac{1}{\left((1+t)b(t)\right)^\frac{\sigma}{2\delta}}$.
Then, the Fourier transformed Cauchy problem \eqref{eq:2.2} leads to the system of first order
\begin{align} \label{eq:2.51}
D_t U=\underbrace{\left(\begin{array}{cc}
\frac{D_t \gamma(t)}{\gamma(t)} & \gamma(t) \\
\frac{|\xi|^{2\sigma}}{\gamma(t)} & i |\xi|^{2\delta} b(t)
\end{array}\right)}_{A(t, \xi)} U .
\end{align} 
We are interested in the fundamental solution
$
E_{\text {diss }}=E_{\text {diss }}(t, s, \xi)=\left(\begin{array}{ll}
E_{\text {diss }}^{(11)} & E_{\text {diss }}^{(12)} \\
E_{\text {diss }}^{(21)} & E_{\text {diss }}^{(22)}
\end{array}\right)
$
to system \eqref{eq:2.51}, that is, the solution of
$
D_t E_{\text {diss }}(t, s, \xi)=A(t, \xi) E_{\text {diss }}(t, s, \xi), \quad E_{\text {diss }}(s, s, \xi)=I,
$
for all $0 \leq s \leq t$ and $(t, \xi),(s, \xi) \in Z_{\text {diss }}\left(d_0\right)$. Thus, the solution $U=U(t, \xi)$ is represented as
$
U(t, \xi)=E_{\text {diss }}(t, s, \xi) U(s, \xi) .
$
The entries $E_{\text {diss }}^{(k \ell)}(t, s, \xi), k, \ell=1,2$, of the fundamental solution $E_{\text {diss }}(t, s, \xi)$ satisfy the following system of Volterra integral equations for $k=1,2$:
$$
\begin{aligned}
& D_t E_{\text {diss }}^{(1 \ell)}(t, s, \xi)=\frac{D_t \gamma(t)}{\gamma(t)} E_{\text {diss }}^{(1 \ell)}(t, s, \xi)+\gamma(t) E_{\text {diss }}^{(2 \ell)}(t, s, \xi), \\
& D_t E_{\text {diss }}^{(2 \ell)}(t, s, \xi)=\frac{|\xi|^{2\sigma}}{\gamma(t)} E_{\text {diss }}^{(1 \ell)}(t, s, \xi)+i |\xi|^{2\delta} b(t) E_{\text {diss }}^{(2 \ell)}(t, s, \xi)
\end{aligned}
$$
together with their initial conditions
$
\left(\begin{array}{ll}
E_{\text {diss }}^{(11)}(s, s, \xi) & E_{\text {diss }}^{(12)}(s, s, \xi) \\
E_{\text {diss }}^{(21)}(s, s, \xi) & E_{\text {diss }}^{(22)}(s, s, \xi)
\end{array}\right)=\left(\begin{array}{ll}
1 & 0 \\
0 & 1
\end{array}\right)$. 
By direct calculations, with  $\lambda=\lambda(t)$ is defined in \eqref{lamda},
we get
$$
\begin{aligned}
& E_{\text {diss }}^{(11)}(t, s, \xi)=\frac{\gamma(t)}{\gamma(s)}+i \gamma(t) \int_s^t E_{\text {diss }}^{(21)}(\tau, s, \xi) d \tau, \;\; E_{\text {diss }}^{(21)}(t, s, \xi)=\frac{i|\xi|^{2\sigma}}{\lambda^2(t)} \int_s^t \frac{\lambda^2(\tau)}{\gamma(\tau)} E_{\text {diss }}^{(11)}(\tau, s, \xi) d \tau, \\
& E_{\text {diss }}^{(12)}(t, s, \xi)=i \gamma(t) \int_s^t E_{\text {diss }}^{(22)}(\tau, s, \xi) d \tau, \;\; E_{\text {diss }}^{(22)}(t, s, \xi)=\frac{\lambda^2(s)}{\lambda^2(t)}+\frac{i|\xi|^{2\sigma}}{\lambda^2(t)} \int_s^t \frac{\lambda^2(\tau)}{\gamma(\tau)} E_{\text {diss }}^{(12)}(\tau, s, \xi) d \tau.
\end{aligned}
$$

The next lemma estimates for the entries $E_{\text {diss }}^{(k \ell)}(t, s, \xi)$, $k, \ell=1,2$.
\begin{lemma} \label{lemma_diss_1}
Assume effective conditions for $b(t)$. Then in the dissipative zone:
$$
\left(\left|E_{\text {diss }}(t, s, \xi)\right|\right) \lesssim \left(\begin{array}{cc}
\frac{\gamma(t)}{\gamma(s)} \exp \left(-|\xi|^{\sigma} (t-s) \right) &  \frac{\gamma(t)}{|\xi|^{2\delta}b(t)} \exp \left(-|\xi|^{\sigma} (t-s) \right) \\
\frac{|\xi|^{2\sigma-2\delta}}{\gamma(s) b(t)} \exp \left(-|\xi|^{\sigma} (t-s) \right) &H
\end{array}\right)
$$
with $(s, \xi),(t, \xi) \in Z_{\text {diss }}\left(d_0\right), \; 0 \leq s \leq t$, $\gamma(t):=\frac{1}{\left((1+t)b(t)\right)^\frac{\sigma}{2\delta}}$ and $H:=\exp \left( -\max\{|\xi|^{\sigma} (t-s), |\xi|^{2\delta}\widehat{B} (s,t)) \} \right)$.
\end{lemma}
\begin{proof}
First let us consider the first column, that is,
$$
\begin{aligned}
& E_{\text {diss }}^{(11)}(t, s, \xi)=\frac{\gamma(t)}{\gamma(s)}+i \gamma(t) \int_s^t E_{\text {diss }}^{(21)}(\tau, s, \xi) d \tau, \;\; E_{\text {diss }}^{(21)}(t, s, \xi)=\frac{i|\xi|^{2\sigma}}{\lambda^2(t)} \int_s^t \frac{\lambda^2(\tau)}{\gamma(\tau)} E_{\text {diss }}^{(11)}(\tau, s, \xi) d \tau.
\end{aligned}
$$
Plugging the representation for $E_{\text {diss }}^{(21)}(t, s, \xi)$ into the integral equation for $E_{\text {diss }}^{(11)}(t, s, \xi)$ gives
$$
\begin{aligned}
E_{\mathrm{diss}}^{(11)}(t, s, \xi) & =\frac{\gamma(t)}{\gamma(s)}+i \gamma(t) \int_s^t\left(\frac{i|\xi|^{2\sigma}}{\lambda^2(\tau)} \int_s^\tau \frac{\lambda^2(\theta)}{\gamma(\theta)} E_{\mathrm{diss}}^{(11)}(\theta, s, \xi) d \theta\right) d \tau \\
& =\frac{\gamma(t)}{\gamma(s)}-|\xi|^{2\sigma} \gamma(t) \int_s^t \int_s^\tau \frac{\lambda^2(\theta)}{\lambda^2(\tau)} \frac{1}{\gamma(\theta)} E_{\mathrm{diss}}^{(11)}(\theta, s, \xi) d \theta d \tau.
\end{aligned}
$$
It follows that
$$
\frac{\gamma(s)}{\gamma(t)} E_{\text {diss }}^{(11)}(t, s, \xi)=1-|\xi|^{2\sigma} \int_s^t \int_s^\tau \frac{\lambda^2(\theta)}{\lambda^2(\tau)} \frac{\gamma(s)}{\gamma(\theta)} E_{\text {diss }}^{(11)}(\theta, s, \xi) d \theta d \tau.
$$
By setting $y(t, s, \xi):=\frac{\gamma(s)}{\gamma(t)} E_{\text {diss }}^{(11)}(t, s, \xi)$ we obtain
$$
\begin{aligned}
y(t, s, \xi) =1-|\xi|^{2\sigma} \int_s^t \int_s^\tau \frac{\lambda^2(\theta)}{\lambda^2(\tau)} y(\theta, s, \xi) d \theta d \tau .
\end{aligned}
$$
Thus, we have
$
|y(t, s, \xi)| \lesssim 1+ \int_s^t\int_s^\tau |\xi|^{2\sigma} |y(\theta, s, \xi)| d \theta d \tau .
$
Applying Gronwall's inequality (cf.  Lemma \ref{Gronwall}) we get the estimate
\begin{align} \label{ineq:E12}
|y(t, s, \xi)| &\lesssim \exp \left(\int_s^t\int_s^\tau |\xi|^{2\sigma} d \theta d \tau \right) \lesssim \exp \left(|\xi|^{2\sigma} (1+t)^2 \right) \lesssim 1 \notag \\
&=ee^{-1} \lesssim e \exp \left(-|\xi|^{\sigma} (t-s) \right)\lesssim \exp \left(-|\xi|^{\sigma} (t-s) \right).
\end{align}
In the above we have used the inequality $ |\xi|^{\sigma} \lesssim \frac{1}{1+t}$, which is implied from $|\xi|^{\sigma-2\delta} \lesssim b(t)$ (from the definition of the elliptic region), and $|\xi|^{2\delta} \lesssim \frac{1}{(1+t)b(t)}$ (from the definition of the dissipative zone), respectively. Hence, we can conclude that
$
\left|E_{\text {diss }}^{(11)}(t, s, \xi)\right| \lesssim \frac{\gamma(t)}{\gamma(s)} \exp \left(-|\xi|^{\sigma} (t-s) \right).
$

Now we consider $E_{\text {diss }}^{(21)}(t, s, \xi)$. By using the estimate for $\left|E_{\text {diss }}^{(11)}(t, s, \xi)\right|$ we obtain
$$
\begin{aligned}
\left|E_{\text {diss }}^{(21)}(t, s, \xi)\right| & \lesssim|\xi|^{2\sigma} \int_s^t \frac{1}{\gamma(\tau)} \frac{\lambda^2(\tau)}{\lambda^2(t)}\left|E_{\text {diss }}^{(11)}(\tau, s, \xi)\right| d \tau \\
& \lesssim \frac{1}{\gamma(s)}|\xi|^{2\sigma} \int_s^t  \frac{\lambda^2(\tau)}{\lambda^2(t)} d \tau \lesssim \frac{1}{\gamma(s)} |\xi|^{2\sigma} (1+t) \lesssim \frac{1}{\gamma(s)} \frac{|\xi|^{2\sigma-2\delta}}{b(t)}.
\end{aligned}
$$
Here, $|\xi|^{2\delta} \lesssim \frac{1}{(1+t)b(t)}$ thanks to the definition of the dissipative zone. Similarly  as in \eqref{ineq:E12} we obtain
\begin{align*}
\left|E_{\text {diss }}^{(21)}(t, s, \xi)\right| & \lesssim \frac{1}{\gamma(s)} \frac{|\xi|^{2\sigma-2\delta}}{b(t)} \exp \left(-|\xi|^{2\sigma-2\delta} (t-s) \right).
\end{align*}
Next, we consider the entries of the second column. We have
$$
\begin{aligned}
 E_{\text {diss }}^{(12)}(t, s, \xi)=i \gamma(t) \int_s^t E_{\text {diss }}^{(22)}(\tau, s, \xi) d \tau, \;\; E_{\text {diss }}^{(22)}(t, s, \xi)=\frac{\lambda^2(s)}{\lambda^2(t)}+\frac{i|\xi|^{2\sigma}}{\lambda^2(t)} \int_s^t \frac{\lambda^2(\tau)}{\gamma(\tau)} E_{\text {diss }}^{(12)}(\tau, s, \xi) d \tau .
\end{aligned}
$$
Plugging the representation for $E_{\text {diss }}^{(22)}(t, s, \xi)$ into the integral equation for $E_{\text {diss }}^{(12)}(t, s, \xi)$ gives
$$
\begin{aligned}
E_{\mathrm{diss}}^{(12)}(t, s, \xi) & =i \gamma(t) \int_s^t\left(\frac{\lambda^2(s)}{\lambda^2(\tau)}+\frac{i|\xi|^{2\sigma}}{\lambda^2(\tau)} \int_s^\tau \frac{\lambda^2(\theta)}{\gamma(\theta)} E_{\mathrm{diss}}^{(12)}(\theta, s, \xi) d \theta\right) d \tau \\
& =i \gamma(t) \lambda^2(s) \int_s^t \frac{1}{\lambda^2(\tau)} d \tau-|\xi|^{2\sigma} \gamma(t) \int_s^t \int_s^\tau \frac{\lambda^2(\theta)}{\lambda^2(\tau)} \frac{1}{\gamma(\theta)} E_{\mathrm{diss}}^{(12)}(\theta, s, \xi) d \theta d \tau.
\end{aligned}
$$
By setting $y(t, s, \xi):=\frac{1}{\gamma(t)} E_{\text {diss }}^{(12)}(t, s, \xi)$, it follows that
$$
y(t, s, \xi)=i \lambda^2(s) \int_s^t \frac{1}{\lambda^2(\tau)} d \tau-|\xi|^{2\sigma} \int_s^t \int_s^\tau \frac{\lambda^2(\theta)}{\lambda^2(\tau)} y(\theta, s, \xi) d \theta d \tau.
$$
In the same way as it was done for $E_{\text {diss }}^{(11)}(t, s, \xi)$, we obtain
\begin{align*}
|y(t, s, \xi)| & \lesssim \lambda^2(s) \int_s^t \frac{1}{\lambda^2(\tau)} d \tau+\int_s^t\int_s^\tau |\xi|^{2\sigma} |y(\theta, s, \xi)| d \theta d \tau  \lesssim (1+t) +\int_s^t\int_s^\tau |\xi|^{2\sigma} |y(\theta, s, \xi)| d \theta d \tau \\ & \lesssim \frac{1}{|\xi|^{2\delta}b(t)} +\int_s^t\int_s^\tau |\xi|^{2\sigma} |y(\theta, s, \xi)| d \theta d \tau.
\end{align*}
Here we use $|\xi|^{2\delta} \lesssim \frac{1}{(1+t)b(t)}$ in the dissipative zone. Now by Gronwall's inequality we get
\begin{align*}
|y(t, s, \xi)|  \lesssim \frac{1}{|\xi|^{2\delta}b(t)} \exp \left(\int_s^t\int_s^\tau |\xi|^{2\sigma} |y(\theta, s, \xi)| d \theta d \tau \right) \lesssim \frac{1}{|\xi|^{2\delta}b(t)} \exp \left(|\xi|^{2\sigma} (1+t)^2 \right) \lesssim \frac{1}{|\xi|^{2\delta}b(t)}.
\end{align*}
In the last estimate we have used  $|\xi|^{\sigma-2\delta} \lesssim b(t)$ from the definition of the elliptic region, $|\xi|^{2\delta} \lesssim \frac{1}{(1+t)b(t)}$ from the definition of the dissipative zone, respectively.
Thus, we obtain $\left|E_{\text {diss }}^{(12)}(t, s, \xi)\right| \lesssim \frac{\gamma(t)}{|\xi|^{2\delta}b(t)}$. Similarly as in \eqref{ineq:E12}, we get
$
\left|E_{\text {diss }}^{(12)}(t, s, \xi)\right| \lesssim \frac{\gamma(t)}{|\xi|^{2\delta}b(t)} \exp \left(-|\xi|^{\sigma} (t-s) \right).
$

Finally, let us consider $E_{\text {diss }}^{(22)}(t, s, \xi)$ by using the representations for $\left|E_{\text {diss }}^{(12)}(t, s, \xi)\right|$. It holds that
$$
\begin{aligned}
\left|E_{\mathrm{diss}}^{(22)}(t, s, \xi)\right| & \lesssim \frac{\lambda^2(s)}{\lambda^2(t)}+|\xi|^{2\sigma} \int_s^t \frac{\lambda^2(\tau)}{\lambda^2(t)} \frac{1}{\gamma(\tau)}\left|E_{\mathrm{diss}}^{(12)}(\tau, s, \xi)\right| d \tau \\
& \lesssim \frac{\lambda^2(s)}{\lambda^2(t)}+|\xi|^{2\sigma} \int_s^t \int_s^\tau E_{\text {diss }}^{(22)}(\theta, s, \xi) d \theta d \tau.
\end{aligned}
$$
We apply Gronwall's inequality to obtain
\begin{align*}
\left|E_{\text {diss }}^{(22)}(t, s, \xi)\right| \lesssim \frac{\lambda^2(s)}{\lambda^2(t)} \exp \left(\int_s^t\int_s^\tau |\xi|^{2\sigma} d \theta d \tau \right) \lesssim \frac{\lambda^2(s)}{\lambda^2(t)} \exp \left(|\xi|^{2\sigma} (1+t)^2 \right) \lesssim \frac{\lambda^2(s)}{\lambda^2(t)}.
\end{align*}
Therefore, $ \left|E_{\text {diss }}^{(22)}(t, s, \xi)\right| \lesssim \exp \left(     -|\xi|^{2\delta}\hat{B}(s,t) \right)$. 
In the last estimate we have used $|\xi|^{\sigma-2\delta} \lesssim b(t)$ from the definition of the elliptic region and $|\xi|^{2\delta} \lesssim \frac{1}{(1+t)b(t)}$ from the definition of the dissipative zone, respectively. 

Similarly as in deriving \eqref{ineq:E12} we can prove that $ \left|E_{\text {diss }}^{(22)}(t, s, \xi)\right| \lesssim \exp \left(-|\xi|^{\sigma}(t-s)\right)$. This completes the proof.
\end{proof}
\subsubsection{Estimates in the reduced zone and pseudo-differential zone}
In  $Z_{\mathrm{red}}(\varepsilon)$ we introduce the micro-energy $V=\left(\varepsilon \frac{b(t)}{2} |\xi|^{2 \delta} \widehat{v}, D_t \widehat{v}\right)^{\mathrm{T}}$ which, thanks to  \eqref{eq:2.3}, satisfies
\begin{equation} \label{eq:2.9}
D_t V=\underbrace{\left(\begin{array}{cc}
\frac{D_t b(t)}{b(t)} & \varepsilon \frac{b(t)}{2} |\xi|^{2 \delta} \\
\frac{|\xi|^{2 \sigma}-\frac{|\xi|^{4 \delta}}{4} b^2(t)-\frac{|\xi|^{2 \delta}}{2} b^{\prime}(t)}{\varepsilon \frac{b(t)}{2} |\xi|^{2 \delta}} & 0
\end{array}\right)}_{A^V(t, \xi)} V.
\end{equation}
Let us estimate the fundamental solution $E_{\mathrm{red}}^V=E_{\mathrm{red}}^V(t, s, \xi)$ to \eqref{eq:2.9} that solves 
$$
\left\{\begin{array}{l}
D_t E_{\mathrm{red}}^V(t, s, \xi)=A^V(t, \xi) E_{\mathrm{red}}^V(t, s, \xi), \\
E_{\mathrm{red}}^V(s, s, \xi)=I .
\end{array}\right.
$$
\begin{lemma} \label{lemma2.5}
The fundamental solution $E_{\mathrm{red}}^V=E_{\mathrm{red}}^V(t, s, \xi)$ to \eqref{eq:2.9} satisfies the following estimate:
\begin{equation*} 
\left(\left|E_{\mathrm{red}}^V(t, s, \xi)\right|\right) \lesssim \exp \left(\varepsilon |\xi|^{2 \delta} \int\limits_s^t b(\tau) d \tau\right)\left(\begin{array}{ll}
1 & 1 \\
1 & 1
\end{array}\right)
\end{equation*}
for all $t \geq s \geq t_0$ with sufficiently large $t_0=t_0(\varepsilon)$ and $(t, \xi),(s, \xi) \in Z_{\mathrm{red}}(\varepsilon)$.
\end{lemma}
For the proof of Lemma \ref{lemma2.5}, see  \cite[Section 2.3]{Wirth2007}.\\
In $Z_{\mathrm{pd}}(N, \varepsilon)$ we introduce the micro-energy $V=\left(\langle\xi\rangle_{b(t)} \widehat{v}, D_t \widehat{v}\right)^{\mathrm{T}}$. By \eqref{eq:2.3} it holds that
\begin{equation} \label{eq:z_pd}
D_t V=\left(\begin{array}{cc}
0 & \langle\xi\rangle_{b(t)} \\
\langle\xi\rangle_{b(t)} & 0
\end{array}\right) V+\left(\begin{array}{cc}
\frac{D_t\langle\xi\rangle_{b(t)}}{\langle\xi\rangle_{b(t)}} & 0 \\
-\frac{|\xi|^{2 \delta}b^{\prime}(t)}{2\langle\xi)_{b(t)}} & 0
\end{array}\right) V \text {. }
\end{equation}
\begin{lemma} \label{lemma2.6}
The fundamental solution $E_{\mathrm{pd}}^V=E_{\mathrm{pd}}^V(t, s, \xi)$ to \eqref{eq:z_pd} satisfies the following estimate:
$$
\left(\left|E_{\mathrm{pd}}^V(t, s, \xi)\right|\right) \lesssim\left(\frac{1+t}{1+s}\right)^{C |\xi|^{2 \delta}}\left(\begin{array}{ll}
1 & 1 \\
1 & 1
\end{array}\right)
$$
for all $t \geq s \geq t_0$ with sufficiently large $t_0=t_0(\varepsilon)$ and $(t, \xi),(s, \xi) \in Z_{\mathrm{pd}}(N, \varepsilon)$.
\end{lemma}
For the proof of Lemma \ref{lemma2.6}, see Section 2.4 in \cite{Wirth2007}.


 \section{Energy estimates of higher order} \label{section_4}

\subsection{A family of parameter-dependent linear Cauchy problems}
Let us consider the following family of parameter-dependent Cauchy problem
\begin{equation}
\begin{cases} \label{eq:3.1}
u_{t t}(t, x)+(-\Delta)^\sigma u(t, x)+b(t) (-\Delta)^\delta u_t(t, x)=0, & (t, x) \in[s, \infty) \times \mathbb{R}^n, \\ u(s, x)=f(s,x), \quad u_t(s, x)=g(s,x), & x \in \mathbb{R}^n,
\end{cases}
\end{equation}
where $\sigma \geq 1, \delta \in [0,\frac{\sigma}{2}] $ and $n \geq 1$.
The partial Fourier transformation with respect to $x$-variables yields
\begin{equation}
\begin{cases} \label{eq:3.2}
\widehat{u}_{t t}+|\xi|^{2 \sigma} \widehat{u}+b(t) |\xi|^{2 \delta} \widehat{u}_t=0, & (t, \xi) \in[s,\infty) \times \mathbb{R}^n, \\ \widehat{u}(s, \xi)= \widehat{f}(s,\xi), \quad \widehat{u}_t(s, \xi)= \widehat{g}(s,\xi), & \xi \in \mathbb{R}^n .\end{cases}
\end{equation}
Applying the change of variables
$\widehat{y}(t, \xi)=\frac{\lambda(t)}{\lambda(s)} \widehat{u}(t, \xi)$ with $\lambda(t)$ is defined in \eqref{lamda}, we obtain the  problem
\begin{equation} \label{eq:3.3}
\begin{cases}\widehat{y}_{t t}+m(t, \xi) \widehat{y}=0, & (t, \xi) \in[s, \infty) \times \mathbb{R}^n, s \geq 0, \\ \widehat{y}(s, \xi)= \exp \left(\frac{|\xi|^{2 \delta}}{2} \int\limits_0^s b(\tau) d \tau\right) \widehat{f}(s,\xi) , & \xi \in \mathbb{R}^n, \\
\widehat{y}_t(s, \xi)=\exp \left(\frac{|\xi|^{2 \delta}}{2} \int\limits_0^s b(\tau) d \tau\right)\left(\frac{b(s)}{2}|\xi|^{2 \delta} \widehat{f}(s,\xi) +\widehat{g}(s,\xi)\right) , & \xi \in \mathbb{R}^n .
\end{cases}
\end{equation}
\subsubsection{For $\delta \in (0, \frac{\sigma}{2}]$} In the same manner as in Section \ref{section_3}, we divide the  phase space $[s, \infty) \times \mathbb{R}^n$ into the following zones:
$$
\begin{aligned}
Z_{\mathrm{hyp}}(N) & =\left\{(t, \xi) \in[s, \infty) \times \mathbb{R}^n:\langle\xi\rangle_{b(t)} \geq N \frac{|\xi|^{2\delta}b(t)}{2}\right\} \cap \Pi_{\mathrm{hyp}}, \\
Z_{\mathrm{pd}}(N, \varepsilon) & =\left\{(t, \xi) \in[s, \infty) \times \mathbb{R}^n: \varepsilon \frac{|\xi|^{2\delta}b(t)}{2} \leq \langle\xi\rangle_{b(t)} \leq N \frac{|\xi|^{2\delta}b(t)}{2}\right\} \cap \Pi_{\mathrm{hyp}}, \\
Z_{\mathrm{red}}(\varepsilon) & =\left\{(t, \xi) \in[s, \infty) \times \mathbb{R}^n:\langle\xi\rangle_{b(t)} \leq \varepsilon \frac{|\xi|^{2\delta}b(t)}{2}\right\}, \\
Z_{\mathrm{ell}}\left(\varepsilon, d_0\right) & =\left\{(t, \xi) \in[s, \infty) \times \mathbb{R}^n:\langle\xi\rangle_{b(t)} \geq \varepsilon \frac{|\xi|^{2\delta}b(t)}{2}\right\} \cap \left\{|\xi|^{2\delta} \geq \frac{d_0}{(1+t)b(t)}\right\} \cap \Pi_{\mathrm{ell}}, \\
Z_{\mathrm{diss}}\left(d_0\right) &=\left\{(t, \xi) \in[s, \infty) \times \mathbb{R}^n:|\xi|^{2\delta} \leq \frac{d_0}{(1+t)b(t)}\right\} \cap \Pi_{\text {ell }}.
\end{aligned}
$$
We introduce the limit $b_{\infty}:=\lim _{t \rightarrow \infty} b(t) \in[0, \infty]$ that exists  thanks to the monotonic behavior of $b=b(t)$.

Let us introduce the function
\begin{equation} \label{h_1}
h_1=h_1(t, \xi)=\chi(|\xi|^{2\delta} (1+t)b(t)) \frac{1}{\left((1+t)b(t)\right)^\frac{\sigma}{2\delta}}+(1-\chi(|\xi|^{2\delta} (1+t)b(t)) |\xi|^{\sigma},
\end{equation}
and
\begin{equation} \label{h}
h_2=h_2(t, \xi)=\chi\left(\frac{\langle\xi\rangle}{\varepsilon \frac{b(t)}{2}}\right) \varepsilon |\xi|^{2 \delta} \frac{b(t)}{2}+\left(1-\chi\left(\frac{\langle\xi\rangle}{\varepsilon \frac{b(t)}{2}}\right)\right) \langle\xi\rangle_{b(t)},
\end{equation}
for our models \eqref{eq:3.2} and \eqref{eq:3.3}. Here $\chi \in C^{\infty}([0, \infty))$ is a localizing function with $\chi(\zeta)=1$ for $0 \leq \zeta \leq \frac{1}{2}$ and $\chi(\zeta)=0$ for $\zeta \geq 1$. We define the energy $
U(t, \xi)=\left(h_1(t, \xi) \hat{u}(t, \xi), D_t \hat{u}(t, \xi)\right)^{T}$.

By \eqref{eq:3.2}, $U=U(t, \xi)$ solves the system
\begin{equation} \label{eq:3.4.00}
D_t U(t, \xi)=\underbrace{\left(\begin{array}{cc}
\frac{D_t h_1(t, \xi)}{h_1(t, \xi)} & h_1(t, \xi) \\
\frac{|\xi|^{2\sigma}}{h_1(t, \xi)} & i |\xi|^{2\delta} b(t)
\end{array}\right)}_{A(t, \xi)} U(t, \xi) .
\end{equation}

For any $t \geq t_1 \geq s$, let $E=E\left(t, t_1, \xi\right)$ be the fundamental solution to \eqref{eq:3.4.00} which solves
$
D_t E\left(t, t_1, \xi\right)=A(t, \xi) E\left(t, t_1, \xi\right), E\left(t_1, t_1, \xi\right)=I
$. For any $t \geq t_2 \geq t_1 \geq s$ we can write $E\left(t, t_1, \xi\right)=E\left(t, t_2, \xi\right) E\left(t_2, t_1, \xi\right)$.

On the other hand, we use the dissipative transformed Cauchy problem \eqref{eq:3.3} and define its energy by
$Y(t, \xi)=\left(i h_2(t, \xi) \widehat{y}(t, \xi), D_t \widehat{y}(t, \xi)\right)^{\mathrm{T}}$. Then, from \eqref{eq:3.3} we have
\begin{equation} \label{eq:3.4}
D_t Y(t, \xi)=\underbrace{\left(\begin{array}{ll}
\frac{D_t h_2(t, \xi)}{h_2(t, \xi)} & i h_2(t, \xi) \\
i \frac{m(t, \xi)}{h_2(t, \xi)} & 0
\end{array}\right)}_{A^Y(t, \xi)} Y(t, \xi) .
\end{equation}
We denote by $E^Y=E^Y\left(t, t_1, \xi\right)$ the fundamental solution to \eqref{eq:3.4} for any $t \geq t_1 \geq s$, i.e., the solution to
$$
\left\{\begin{array}{l}
D_t E^Y\left(t, t_1, \xi\right)=A^Y(t, \xi) E^Y\left(t, t_1, \xi\right) \\
E^Y\left(t_1, t_1, \xi\right)=I.
\end{array}\right.
$$

\subsection{Representation of the solutions}

Let us turn now to Cauchy problem \eqref{eq:3.1}. We introduce $\widehat{K}_1=\widehat{K}_1(t, s, \xi)$ as the solution to \eqref{eq:3.2} with initial conditions $\widehat{u}(s,\xi)=0$ and $\widehat{u}_t(s,\xi)=1$. Following the approach of \cite[Section 7.4]{DAbbicco2013},  we have
\begin{equation} \label{K_1}
\widehat{K}_1(t, s, \xi)  =\frac{1}{h_1(t, \xi)} E_{12}(t, s, \xi)=\frac{\lambda(s)}{\lambda(t)} \frac{E_{12}^Y(t, s, \xi)}{h_2(t, \xi)},
\end{equation}
\begin{equation} \label{DtK_1}
D_t \widehat{K}_1(t, s, \xi)  =E_{22}(t, s, \xi)=\frac{\lambda(s)}{\lambda(t)}\left(E_{22}^Y(t, s, \xi)+ i \frac{|\xi|^{2 \delta}b(t)}{2 h_2(t, \xi)} E_{12}^Y(t, s, \xi)\right) .
\end{equation}
In the same way we consider $\widehat{K}_0=\widehat{K}_0(t, s, \xi)$ as the solution to \eqref{eq:3.2} with $s=0$ and initial conditions $\widehat{u}(0,\xi)=1$ and $\widehat{u}_t(0,\xi)=0$. Then, it holds that
$$
\begin{aligned}
\widehat{K}_0(t, 0, \xi) &= \frac{h_1(0, \xi)}{h_1(t, \xi)}E_{11}(t, 0, \xi)=\frac{1}{\lambda(t)} \frac{h_2(0, \xi)}{h_2(t, \xi)} E_{11}^Y(t, 0, \xi), \\
D_t \widehat{K}_0(t, 0, \xi) & =h_1(0, \xi)E_{21}(t, 0, \xi) =\frac{h_2(0, \xi)}{\lambda(t)}\left(E_{21}^Y(t, 0, \xi)+ i \frac{|\xi|^{2 \delta}b(t)}{2 h_2(t, \xi)} E_{11}^Y(t, 0, \xi)\right),
\end{aligned}
$$
where $h_1=h_1(t, \xi)$ is defined in \eqref{h_1} and $E:=E(t, s, \xi)$ is the fundamental solution to system \eqref{eq:3.4.00}. These above relations allow us to transfer properties of $E=E(t, s, \xi)$ and $E^Y=E^Y(t, s, \xi)$ to $\widehat{K}_1=\widehat{K}_1(t, s, \xi)$ and $E=E(t, 0, \xi)$, $E^Y=E^Y(t, 0, \xi)$ to $\widehat{K}_0=\widehat{K}_0(t, 0, \xi)$, 
where $h_2=h_2(t, \xi)$ is defined in \eqref{h} and $E^Y:=E^Y(t, s, \xi)$ is the fundamental solution to system \eqref{eq:3.4}. 
\subsection{Estimates for the multipliers and their time derivatives}
We need to estimate  multipliers $|\widehat{K}_1(t, s, \xi)|$, $|\widehat{K}_0(t, 0, \xi)|$ and  $|\partial_t \widehat{K}_1(t, s, \xi)|$, $|\partial_t \widehat{K}_0(t, 0, \xi)|$ in each zone of the  phase space. Let us consider the following estimate for $|\widehat{K}_1(t, s, \xi)|$ in $Z_{\text {red }}(\varepsilon)$:
\begin{equation} \label{ineq:3.5}
|\widehat{K}_1^{\mathrm{red}}(t, s, \xi)| \lesssim \frac{1}{|\xi|^{\sigma}}\left(\frac{\lambda(s)}{\lambda(t)}\right)^{1-2 \beta},
\end{equation}
where we choose $\varepsilon>0$ such that $\beta:=C \varepsilon<\frac{1}{2}$. Then, it is easy to see that we can estimate $|\widehat{K}_1(t, s, \xi)|$ in $Z_{\text {hyp }}(N)$ and $Z_{\mathrm{pd}}(N, \varepsilon)$ by \eqref{ineq:3.5}. Thus, we can glue $Z_{\mathrm{red}}(\varepsilon)$ to $Z_{\text {hyp }}(N)$ and  define a new region by
\begin{align*}
\Pi_{\text {hyp }}=Z_{\text {red }}(\varepsilon) \cup Z_{\text {pd }}(N, \varepsilon) \cup Z_{\text {hyp }}(N), \;\;
\Pi_{\mathrm{ell}}=Z_{\mathrm{ell}}\left(\varepsilon, d_0\right) \cup Z_{\mathrm{diss}}\left(d_0\right).
\end{align*}
We denote by $t_{\mathrm{ell}}=t_{\mathrm{ell}}(|\xi|)$ the separating line between $\Pi_{\mathrm{ell}}$ and $\Pi_{\mathrm{hyp}}$. This curve is given by
$$
\frac{b^2\left(t_{\mathrm{ell}}\right)}{4}-|\xi|^{2 \sigma - 4 \delta}=\varepsilon^2 \frac{b^2\left(t_{\mathrm{ell}}\right)}{4}, \quad \text { i.e. } \quad t_{\mathrm{ell}}=b^{-1}\left(\frac{2|\xi|^{\sigma - 2 \delta}}{\sqrt{1-\varepsilon^2}}\right).
$$

\subsection{Estimates for the multiplier $\widehat{K}_1$}
For $s \geq 0$ let us  distinguish between two cases  in the extended phase space.\\

{\bf Small frequencies $\mathbf{|\xi|^{\sigma - 2 \delta} \leq \frac{b(s)}{2} \sqrt{1-\varepsilon^2}}$}. We have the following five sub-cases:
\begin{itemize}
\item Case 1: $0 \leq s \leq t \leq t_{\text {diss }}$
In this case $(t, \xi)$ and $(s, \xi)$ belong to $Z_{\text {diss }}\left(d_0\right)$. It holds $h_1(t, \xi)=\frac{1}{\left((1+t)b(t)\right)^\frac{\sigma}{2\delta}}$. Then, we have the following estimates for all $t \in\left[s, t_{\text {diss }}\right]$:
$$
\begin{aligned}
&\left|\hat{K}_1(t, s, \xi)\right|  \lesssim \frac{1}{|\xi|^{2\delta}b(t)} \exp \left(-|\xi|^{2\sigma-2\delta} (t-s) \right), \;\;\;
\left|\partial_t \hat{K}_1(t, s, \xi)\right|  \lesssim \exp \left(-|\xi|^{2 \delta} \widehat{\mathcal{B}}(s, t)\right),\\
&\left|\partial_t \hat{K}_1(t, s, \xi)\right|  \lesssim \exp \left(-|\xi|^{\sigma}(t-s)\right).
\end{aligned}
$$
\item Case 2: $0 \leq t_{\text {diss }} \leq s \leq t \leq t_{\text {ell }}$. In this case $(t, \xi)$ and $(s, \xi)$ belong to $Z_{\mathrm{ell}}\left(\varepsilon, d_0\right)$. It holds $h_2(t, \xi) \sim |\xi|^{2 \delta} b(t)$. Then, we have the following estimates from Lemma \ref{lemma2.4} for all $t \in\left[s, t_{\mathrm{ell}}\right]$:
$$
\begin{aligned}
|\widehat{K}_1(t, s, \xi)|  \lesssim \frac{1}{b(s) |\xi|^{2 \delta}} \exp \left(-C|\xi|^{2 \sigma - 2 \delta} \mathcal{B}(s, t)\right), \;
|\partial_t \widehat{K}_1(t, s, \xi)|  \lesssim \frac{|\xi|^{2 \sigma - 4 \delta}}{b(s) b(t)} \exp \left(-C|\xi|^{2 \sigma - 2 \delta} \mathcal{B}(s, t)\right) .
\end{aligned}
$$
\item Case 3: $0 \leq s \leq t_{\text {diss }} \leq t \leq t_{\text {ell}}$. In this case $(t, \xi) \in Z_{\text {ell }}\left(\varepsilon, d_0\right)$ and $(s, \xi) \in Z_{\text {diss }}\left(d_0\right)$ we glue the estimates in $Z_{\text {ell }}\left(\varepsilon, d_0\right)$ from Lemma \ref{lemma2.4} and in $Z_{\text {diss }}\left(d_0\right)$ from Lemma
\ref{lemma_diss_1}. It holds $h_2(t, \xi) \sim |\xi|^{2 \delta} b(t)$. That is, we use the representations $E(t, s, \xi)=E_{\text {ell }}\left(t, t_{\text {diss }}, \xi\right) E_{\text {diss }}\left(t_{\text {diss }}, s, \xi\right)$. Hence, we arrive
at the following statement:
$$
\begin{aligned}
|\widehat{K}_1(t, s, \xi)| & \lesssim \frac{1}{b(t) |\xi|^{2 \delta}} \exp \left(-C|\xi|^{2 \sigma - 2 \delta} \mathcal{B}(s, t)\right), \\
|\partial_t \widehat{K}_1(t, s, \xi)| & \lesssim \max\left(\frac{1}{b(s)}, \frac{1}{b(t)} \right) \frac{|\xi|^{2 \sigma - 4 \delta}}{b(t)} \exp \left(-C|\xi|^{2 \sigma - 2 \delta} \mathcal{B}(s, t)\right) \\
&\lesssim \left(\frac{1}{b(s)} + \frac{1}{b(t)} \right) \frac{|\xi|^{2 \sigma - 4 \delta}}{b(t)} \exp \left(-C|\xi|^{2 \sigma - 2 \delta} \mathcal{B}(s, t)\right).
\end{aligned}
$$
\item Case 4: $0 \leq t_{\text {diss }} \leq t_{\text {ell }} \leq s \leq t$. In this case $(t, \xi)$ and $(s, \xi)$ belong to $ Z_{\text {hyp }}\left(N\right)$. It holds $h_2(t, \xi) \sim |\xi|^{\sigma}$. Then, from the estimate in $Z_{\mathrm{red}}(\varepsilon)$ we get the following estimates
\begin{align*}
|\widehat{K}_1(t, s, \xi)|  \lesssim \frac{1}{|\xi|^{\sigma}}\left(\frac{\lambda(s)}{\lambda(t)}\right)^{1-2 \beta},\;\;\;
|\partial_t \widehat{K}_1(t, s, \xi)|  \lesssim\left(\frac{\lambda(s)}{\lambda(t)}\right)^{1-2 \beta}.
\end{align*}
\item Case 5: $0 \leq t_{\text {diss }} \leq s \leq t_{\text {ell }} \leq t$. In this case, thanks to $(t, \xi) \in Z_{\mathrm{red}}(\varepsilon)$ and $(s, \xi) \in Z_{\text {hyp }}\left(N\right)$,  we can glue the estimates in $Z_{\text {red}}\left(\varepsilon\right)$ from Lemma \ref{lemma2.5} and in $Z_{\text {ell }}\left(\varepsilon, d_0\right)$ from Lemma \ref{lemma2.4}. It holds that $h_2(t, \xi) \sim |\xi|^{ \sigma}$.  That is, we use the representations $E(t, s, \xi)=E_{\text {red }}\left(t, t_{\text {ell }}, \xi\right) E_{\text {ell }}\left(t_{\text {ell }}, s, \xi\right)$. We arrive at the following
$$
\begin{aligned}
|\widehat{K}_1(t, s, \xi)|  \lesssim \frac{1}{b(s) |\xi|^{2 \delta}} \exp \left(-C|\xi|^{2 \sigma - 2 \delta} \mathcal{B}(s, t)\right), \;
|\partial_t \widehat{K}_1(t, s, \xi)|  \lesssim \frac{|\xi|^{2 \sigma - 4 \delta}}{b(s) b(t)} \exp \left(-C|\xi|^{2 \sigma - 2 \delta} \mathcal{B}(s, t)\right) .
\end{aligned}
$$
\end{itemize}

{\bf Large frequencies $\mathbf{|\xi|^{\sigma - 2 \delta} \geq \frac{b(s)}{2} \sqrt{1 -\varepsilon^2}}$}. We consider the following two sub-cases:
\begin{itemize}
\item Case 1: $0 \leq s \leq t \leq t_{\text {ell }}$. If $b=b(t)$ is increasing, then $(t, \xi)$ and $(s, \xi)$ belong to $\Pi_{\text {hyp }}(N, \varepsilon)$. Taking $h_2(t, \xi) \sim|\xi|^{\sigma}$ we have
\begin{equation} \label{ineq:3.6}
|\widehat{K}_1(t, s, \xi)| \lesssim \frac{1}{|\xi|^{\sigma}}\left(\frac{\lambda(s)}{\lambda(t)}\right)^{1-2 \beta},
\end{equation}
\begin{equation} \label{ineq:3.7}
|\partial_t \widehat{K}_1(t, s, \xi)| \lesssim\left(\frac{\lambda(s)}{\lambda(t)}\right)^{1-2 \beta},
\end{equation}
which are derived from  \eqref{K_1} and \eqref{DtK_1} and the estimate in $\Pi_{\text {hyp }}(N, \varepsilon)$. Indeed, \eqref{DtK_1} implies that
$$
\begin{aligned}
|\partial_t \widehat{K}_1(t, s, \xi)| & \leq \frac{\lambda(s)}{\lambda(t)}\left(\left|E_{22}^Y(t, s, \xi)\right|+\frac{b(t)|\xi|^{2\delta}}{2 h_2(t, \xi)}\left|E_{12}^Y(t, s, \xi)\right|\right)  \lesssim \frac{\lambda(s)}{\lambda(t)}\left({\frac{\lambda(t)}{\lambda(s)}}\right)^{2 \beta} \lesssim\left(\frac{\lambda(s)}{\lambda(t)}\right)^{1-2 \beta},
\end{aligned}
$$
thanks to the estimates of $E_{12}^Y(t, s, \xi)$ and $E_{22}^Y(t, s, \xi)$ from Lemma \ref{lemma2.6}. We note that   \eqref{ineq:3.6} and \eqref{ineq:3.7} remain true for large frequencies if  $b=b(t)$ is decreasing (then we have only $Z_{\text {hyp }}(N)$ for large frequencies).
\item Case 2: $0 \leq s \leq t_{\text {ell }} \leq t$. We remark that this case comes into play only if $b=b(t)$ is increasing and there is no separating line if $|\xi|^{\sigma - 2 \delta} \geq b_{\infty} \sqrt{1-\varepsilon^2}$. Then,  for all $t_{\text {ell }} \leq t$:
$$
\begin{gathered}
|\widehat{K}_1(t, s, \xi)| \lesssim \frac{1}{b(s)} \exp \left(-C^{\prime}|\xi|^{2 \sigma - \ 2 \delta} \mathcal{B}(s, t)\right), \;\;\;
|\partial_t \widehat{K}_1(t, s, \xi)| \lesssim \frac{|\xi|^{2 \sigma - 4 \delta}}{b(s) b(t)} \exp \left(-C^{\prime}|\xi|^{2 \sigma - 2 \delta} \mathcal{B}(s, t)\right) .
\end{gathered}
$$
\end{itemize}
\subsection{Final estimates}
\subsubsection{For $\delta \in (0, \frac{\sigma}{2})$:}
We define $
\Omega(s, t):=\left(\max \{b(s), b(t)\} \frac{\sqrt{1 -\varepsilon^2}}{2}\right)^{\frac{1}{\sigma - 2 \delta}}$
for any $t \geq s$ with $s \in [0, \infty)$, and $\Lambda(t):=\left( \frac{d_0}{(1+t)b(t)} \right)^\frac{1}{2\delta}$
for $t \geq 0$. Summarizing the  above estimates we arrived at the following  estimates for $|\widehat{K}_1(t, s, \xi)|$ and $|\partial_t \widehat{K}_1(t, s, \xi)|$ with $t \geq s \geq 0$.
\begin{corollary} \label{Corollary_4.1}
If $|\xi| \geq \Omega(s, t)$, then:
\begin{equation} \label{ineq:3.8}
|\widehat{K}_1(t, s, \xi)| \lesssim \frac{1}{|\xi|^{\sigma}}\left(\frac{\lambda(s)}{\lambda(t)}\right)^{1-2 \beta},
\end{equation}
\begin{equation} \label{ineq:3.9}
|\partial_t \widehat{K}_1(t, s, \xi)| \lesssim\left(\frac{\lambda(s)}{\lambda(t)}\right)^{1-2 \beta} .
\end{equation}
If $\Lambda(t) \leq |\xi| \leq \Omega(s, t)$, then:
\begin{equation} \label{ineq:3.10}
|\widehat{K}_1(t, s, \xi)| \lesssim \frac{1}{b(t)|\xi|^{2 \delta}} \exp \left(-C^{\prime}|\xi|^{2 \sigma - 2 \delta} \mathcal{B}(s, t)\right),
\end{equation}
\begin{equation} \label{ineq:3.11}
|\partial_t \widehat{K}_1(t, s, \xi)| \lesssim \left(\frac{1}{b(s)} + \frac{1}{b(t)} \right) \frac{|\xi|^{2 \sigma - 4 \delta}}{b(t)} \exp \left(-C|\xi|^{2 \sigma - 2 \delta} \mathcal{B}(s, t)\right).
\end{equation}
If $|\xi| \leq \Lambda(t)$, then:
\begin{equation} \label{ineq:3.1000}
|\widehat{K}_1(t, s, \xi)| \lesssim \frac{1}{|\xi|^{2 \delta}b(t)} \exp \left(-|\xi|^{2\sigma-2\delta} (t-s) \right),
\end{equation}
\begin{equation} \label{ineq:3.1100}
|\partial_t \widehat{K}_1(t, s, \xi)| \lesssim \exp \left(-|\xi|^{2 \delta} \widehat{\mathcal{B}}(s, t)\right),  \end{equation}
\begin{equation} \label{ineq:3.1101}
 |\partial_t \widehat{K}_1(t, s, \xi)| \lesssim \exp \left(-|\xi|^{\sigma}(t-s)\right).
\end{equation}
\end{corollary}
We have similar estimates for $|\widehat{K}_0(t, 0, \xi)|$ and $|\partial_t \widehat{K}_0(t, 0, \xi)|$.
\begin{corollary}
If $|\xi| \geq \Omega(0, t)$, then
\begin{equation} \label{ineq:3.12}
|\widehat{K}_0(t, 0, \xi)| \lesssim\left(\frac{1}{\lambda(t)}\right)^{1-2 \beta},
\end{equation}
\begin{equation} \label{ineq:3.13}
|\partial_t \widehat{K}_0(t, 0, \xi)| \lesssim|\xi|^{\sigma}\left(\frac{1}{\lambda(t)}\right)^{1-2 \beta}.
\end{equation}
If $\Lambda(t) \leq |\xi| \leq \Omega(0, t)$, then:
\begin{equation} \label{ineq:3.14}
|\widehat{K}_0(t, 0, \xi)| \lesssim \exp \left(-C^{\prime}|\xi|^{2 \sigma - 2 \delta} \mathcal{B}(0, t)\right),
\end{equation}
\begin{equation} \label{ineq:3.15}
|\partial_t \widehat{K}_0(t, 0, \xi)| \lesssim \frac{|\xi|^{2 \sigma - 2 \delta}}{b(t)} \exp \left(-C^{\prime}|\xi|^{2 \sigma - 2 \delta} \mathcal{B}(0, t)\right) .
\end{equation}
If $|\xi| \leq \Lambda(t)$, then:
\begin{equation} \label{ineq:3.1400}
|\widehat{K}_0(t, 0, \xi)| \lesssim \exp \left(-|\xi|^{2\sigma-2\delta} t \right),
\end{equation}
\begin{equation} \label{ineq:3.1500}
|\partial_t \widehat{K}_0(t, 0, \xi)| \lesssim \frac{|\xi|^{2 \sigma - 2 \delta}}{b(t)} \exp \left(-|\xi|^{2\sigma-2\delta} t \right).
\end{equation}
\end{corollary}
\subsubsection{For $\delta = \frac{\sigma}{2}$:} 

\begin{corollary} \label{Corollary_4.3} 
\begin{enumerate}[i)]
\item If $|\xi| \geq M$, then
\begin{equation} \label{ineq:3.8.1}
|\widehat{K}_1(t, s, \xi)| \lesssim \frac{1}{|\xi|^\sigma} \left(\frac{\lambda(s)}{\lambda(t)}\right)^{1-2 \beta},
\end{equation}
\begin{equation} \label{ineq:3.9.1}
|\partial_t \widehat{K}_1(t, s, \xi)| \lesssim\left(\frac{\lambda(s)}{\lambda(t)}\right)^{1-2 \beta} .
\end{equation}
\item If $|\xi| \leq M$, then
\begin{equation} \label{ineq:3.10.1}
|\widehat{K}_1(t, s, \xi)| \lesssim \frac{1}{|\xi|^{\sigma}b(t)} \exp \left(-|\xi|^{\sigma} (t-s) \right) + \frac{1}{b(t)|\xi|^{\sigma}} \exp \left(-C^{\prime}|\xi|^{\sigma} \mathcal{B}(s, t)\right),
\end{equation}
\begin{equation} \label{ineq:3.11.1}
|\partial_t \widehat{K}_1(t, s, \xi)| \lesssim \exp \left(-|\xi|^{\sigma} \widehat{\mathcal{B}}(s, t)\right) + \left(\frac{1}{b(s)} + \frac{1}{b(t)} \right) \frac{1}{b(t)} \exp \left(-C^{\prime}|\xi|^{\sigma} \mathcal{B}(s, t)\right),
\end{equation}
\begin{equation} \label{ineq:3.11.2}
|\partial_t \widehat{K}_1(t, s, \xi)| \lesssim \exp \left(-|\xi|^{\sigma} (t-s)\right) + \left(\frac{1}{b(s)} + \frac{1}{b(t)} \right) \frac{1}{b(t)} \exp \left(-C^{\prime}|\xi|^{\sigma} \mathcal{B}(s, t)\right).
\end{equation}
\end{enumerate}
\end{corollary}
\begin{corollary}
If $|\xi| \geq M$, then
\begin{equation} \label{ineq:3.12.1}
|\widehat{K}_0(t, 0, \xi)| \lesssim\left(\frac{1}{\lambda(t)}\right)^{1-2 \beta},
\end{equation}
\begin{equation} \label{ineq:3.13.1}
|\partial_t \widehat{K}_0(t, 0, \xi)| \lesssim |\xi|^\sigma \left(\frac{1}{\lambda(t)}\right)^{1-2 \beta} .
\end{equation}
If $|\xi| \leq M$, then
\begin{equation} \label{ineq:3.14.1}
|\widehat{K}_0(t, 0, \xi)| \lesssim \exp \left(-|\xi|^{\sigma} t \right) + \exp \left(-C^{\prime}|\xi|^{\sigma} \mathcal{B}(0, t)\right),
\end{equation}
\begin{equation} \label{ineq:3.15.1}
|\partial_t \widehat{K}_0(t, 0, \xi)| \lesssim \frac{|\xi|^{\sigma}}{b(t)} \exp \left(-|\xi|^{\sigma} t \right) + \frac{|\xi|^{\sigma}}{b(t)} \exp \left(-C^{\prime}|\xi|^{\sigma} \mathcal{B}(0, t)\right) .
\end{equation}
\end{corollary}

\subsection{Decay estimates of solutions of the linear Cauchy problem}
\label{section4.6}
We establish now the Matsumura-type decay estimates of the linear solutions. We consider two problems:
\begin{equation} \label{eq:3.16}
\begin{cases}v_{t t}+(-\Delta)^\sigma v+b(t) (-\Delta)^\delta v_t=0, & (t, x) \in[s, \infty) \times \mathbb{R}^n, \\ v(s, x)=0, \quad v_t(s, x)= g(s,x), & x \in \mathbb{R}^n,
\end{cases}
\end{equation}
and
\begin{equation} \label{eq:3.17}
\begin{cases}w_{t t}+(-\Delta)^\sigma w+b(t) (-\Delta)^\delta w_t=0, & (t, x) \in[0, \infty) \times \mathbb{R}^n, \\ w(0, x)=f(x), \quad w_t(0, x)=0, & x \in \mathbb{R}^n .\end{cases}
\end{equation}
We denoted the solutions to \eqref{eq:3.16} and \eqref{eq:3.17} by $K_1=K_1(t, s, x)$ and $K_0=K_0(t, 0, x)$ with data $g(x)=g(0,x)=u_1(x)=\delta_0$ and $f(x)=u_0(x)=\delta_0$, where $\delta_0$ is the Dirac distribution  in $x=0$. Hence
\begin{align*}
\|v(t, \cdot)\|_{L^2}  =\|\widehat{v}(t, \cdot)\|_{L^2} \leq\left\|\widehat{K}_1(t, 0, \xi) \widehat{u}_1(s,\xi) \right\|_{L^2}, 
~\|w(t, \cdot)\|_{L^2}  =\|\widehat{w}(t, \cdot)\|_{L^2} \leq\left\|\widehat{K}_0(t, 0, \xi) \widehat{u}_0(\xi) \right\|_{L^2} .
\end{align*}
\begin{proof}[Proof of Theorem 1.1]
We will present the detailed proof for  $\delta \in (0, \frac{\sigma}{2})$. The proofs for  $\delta=0$ and $\delta=\sigma/2$ can be proceeded similarly, and therefore will be omitted for sake of brevity in the article presentation. Our analysis below relies on the notion of decay characters of the initial data.\\

$\bullet$  {\it Estimating for $\|u(t, \cdot)\|_{\dot{H}^\alpha}$}. From \eqref{ineq:3.8} and the definition of $\lambda(t)$ in \eqref{lamda}, applying Parseval’s identity,  we see that
\begin{align} \label{ineq:3.18}
& \left\||\xi|^\alpha  \widehat{K}_1(t, 0, \xi) \widehat{u}_1(\xi) \right\|_{L^2 \left (|\xi| \geq \Omega(0, t)\right )}^2 
\lesssim \int_{|\xi| \geq \Omega(0, t)}|\xi|^{2\alpha- 2\sigma} \exp \left(-(1-2 \beta)|\xi|^{2 \delta} \widehat{\mathcal{B}}(0,t) \right) \left|\widehat{u_1}(\xi)\right|^2 d \xi \notag\\
\lesssim & \frac{1}{\left(\Omega(0, t) \right)^{2\sigma-2\alpha}} \exp \left(-(1-2 \beta) \left(\Omega(0, t) \right)^{2 \delta} \widehat{\mathcal{B}}(0,t) \right)  \| u_1 \|_{L^2}^2.
\end{align}
From \eqref{ineq:3.12} and the definition of $\lambda(t)$ in \eqref{lamda}, by Parseval’s identity, we obtain
\begin{align} \label{ineq:3.19}
&\left\||\xi|^\alpha  \widehat{K}_0(t, 0, \xi) \widehat{u}_0(\xi) \right\|_{L^2 \left (|\xi| \geq \Omega(0, t)\right )}^2 \notag\\
\lesssim & \int_{|\xi| \geq \Omega(0, t)}|\xi|^{2\alpha} \exp \left(-(1-2 \beta)|\xi|^{2 \delta} \widehat{\mathcal{B}}(0,t) \right) \left|\widehat{u_0}(\xi)\right|^2 d \xi 
\lesssim \exp \left(-(1-2 \beta) \left(\Omega(0, t) \right)^{2 \delta} \widehat{\mathcal{B}}(0,t) \right) \| u_0 \|_{H^\alpha}^2.
\end{align}
Let us turn to estimate $\left\||\xi|^\alpha  \widehat{K}_0(t, 0, \xi) \widehat{u}_1(\xi) \right\|_{L^2 \left (\Lambda(t) \leq |\xi| \leq \Omega(0, t)\right )}^2$ and $\left\||\xi|^\alpha  \widehat{K}_1(t, 0, \xi) \widehat{u}_0(\xi) \right\|_{L^2 \left (\Lambda(t) \leq |\xi| \leq \Omega(0, t)\right )}^2$.\\

\noindent{\it For $\alpha \in [0, 2 \delta)$}: Put $\widehat{\vartheta}=|\xi|^{\alpha-2 \delta} \widehat{u}_1$. Since $\frac{n}{2}+r^*\left(u_1\right)+\alpha-2 \delta>0$, we can apply Theorem \ref{decay_4} to obtain that $\vartheta:=\mathcal{F}^{-1}(\widehat{\vartheta}) \in L^2\left(\mathbb{R}^n\right)$ and $r^*(\vartheta)=r^*\left(u_1\right)+\alpha-2 \delta$. Hence,
from \eqref{ineq:3.10}, by Theorem \ref{decay_1}, we get
\begin{align*}
& \left\||\xi|^{\alpha}  \widehat{K}_1(t, 0, \xi) \widehat{u}_1(\xi) \right\|_{L^2 \left (\Lambda(t) \leq |\xi| \leq \Omega(0, t)\right )}^2 
\lesssim  \frac{1}{b^2(t)} \int_{\Lambda(t) \leq |\xi| \leq \Omega(0, t)}|\xi|^{2\alpha- 4 \delta} \exp \left(-2C^{\prime}|\xi|^{2 \sigma - 2 \delta} \mathcal{B}(0,t)\right) \left|\widehat{u_1}(\xi)\right|^2 d \xi \\
= & \frac{1}{b^2(t)} \int_{\Lambda(t) \leq |\xi| \leq \Omega(0, t)} \exp \left(-2C^{\prime}|\xi|^{2 \sigma - 2 \delta} \mathcal{B}(0,t)\right) \left|\widehat{\vartheta}(\xi)\right|^2 d \xi \\
\lesssim & \frac{1}{b^2(t)} \left(P_{r^*(\vartheta)}(\vartheta)\right)\left(C_3+\mathcal{B}(0,t) \right)^{-\frac{r^{*}(\vartheta)+ \frac{n}{2}}{ \sigma -  \delta}}+ \frac{1}{b^2(t)} \|\vartheta\|_{L^2}^2\left(C_3+\mathcal{B}(0,t) \right)^{-m}
\end{align*}
for any $m>\frac{r^{*}(\vartheta)+ \frac{n}{2}}{ \sigma -  \delta}$, for sufficiently large $C_3=$ const $>0$. By fixing $m>\frac{r^{*}(\vartheta)+ \frac{n}{2}}{ \sigma -  \delta}$ and choosing $C_3$ large enough, we always can achieve that
$$
\|\vartheta\|_{L^2}^2\left(C_3+\mathcal{B}(0,t)\right)^{-m} \lesssim\left(P_{r^*\left(u_1\right)}\left(u_1\right)+\left\|u_1\right\|_{L^2}^2\right)(1+\mathcal{B}(0,t))^{-\frac{r^{*}(\vartheta)+ \frac{n}{2}}{ \sigma -  \delta}}.
$$
Recalling the definition of the norm $\|\cdot\|_{P}$ we obtain
\begin{equation} \label{ineq:3.20.1}
\left\||\xi|^{\alpha}  \widehat{K}_1(t, 0, \xi) \widehat{u}_1(\xi) \right\|_{L^2 \left (\Lambda(t) \leq |\xi| \leq \Omega(0, t)\right )}^2 \lesssim \|u_1\|_{P}^2 \frac{1}{b^2(t)} \left(1+\mathcal{B}(0,t)\right)^{-\frac{r^{*}(u_1) + \frac{n}{2} + \alpha - 2 \delta}{\sigma - \delta}}.
\end{equation}
\vskip0.5cm
\noindent{\it For $\alpha \geq 2 \delta$:} From \eqref{ineq:3.10}, applying Young's inequality, Lemma \ref{lemmae^-c} and Theorem \ref{decay_1}, we get
\begin{align} \label{ineq:3.20}
& \left\||\xi|^{\alpha}  \widehat{K}_1(t, 0, \xi) \widehat{u}_1(\xi) \right\|_{L^2 \left (\Lambda(t) \leq |\xi| \leq \Omega(0, t)\right )}^2 
\lesssim  \frac{1}{b^2(t)} \int_{\Lambda(t) \leq |\xi| \leq \Omega(0, t)}|\xi|^{2\alpha- 4 \delta} \exp \left(-2C^{\prime}|\xi|^{2 \sigma - 2 \delta} \mathcal{B}(0,t)\right) \left|\widehat{u_1}(\xi)\right|^2 d \xi \notag\\
\lesssim & \frac{1}{b^2(t)} \left\|F^{-1}\left( |\xi|^{\alpha - 2 \delta} \exp \left(-\frac{C^{\prime}}{2}|\xi|^{2 \sigma - 2 \delta} \mathcal{B}(0,t)\right)\right)\right\|_{L^1}^2  \left\|\exp \left(- \frac{C^{\prime}}{2} |\xi|^{2 \sigma - 2 \delta} \mathcal{B}(0,t)\right) \widehat{u}_1(\xi)\right\|_{L^2}^2 \notag\\
\lesssim & \|u_1\|_{P}^2 \frac{1}{b^2(t)} \left(1+\mathcal{B}(0,t)\right)^{-\frac{r^{*}(u_1) + \frac{n}{2} + \alpha - 2 \delta}{\sigma - \delta}}.
\end{align}
From \eqref{ineq:3.14}, by Theorem \ref{decay_3}, we obtain
\begin{align} \label{ineq:3.21}
&\left\||\xi|^\alpha  \widehat{K}_0(t, 0, \xi) \widehat{u}_0(\xi) \right\|_{L^2 \left (\Lambda(t) \leq |\xi| \leq \Omega(0, t)\right )}^2 
\lesssim  \int_{\Lambda(t) \leq |\xi| \leq \Omega(0, t)}|\xi|^{2\alpha} \exp \left(-2C^{\prime}|\xi|^{2 \sigma - 2 \delta} \mathcal{B}(0,t)\right) \left|\widehat{u_0}(\xi)\right|^2 d \xi \notag\\
\lesssim &  \|u_0\|_{P, \alpha}^2 \left(1+\mathcal{B}(0,t)\right)^{-\frac{r^{*}(u_0) + \frac{n}{2} + \alpha}{\sigma - \delta}}.
\end{align}
Next, we focus on the estimates for small frequencies: $  |\xi| \leq \Lambda(t)$.\\

\noindent{\it For $\alpha \in [0, 2 \delta)$}: Put $\widehat{\vartheta}=|\xi|^{\alpha-2 \delta} \widehat{u}_1$. Since $\frac{n}{2}+r^*\left(u_1\right)+\alpha-2 \delta>0$, we can apply Theorem \ref{decay_4} to obtain that $\vartheta:=\mathcal{F}^{-1}(\widehat{\vartheta}) \in L^2\left(\mathbb{R}^n\right)$ and $r^*(\vartheta)=r^*\left(u_1\right)+\alpha-2 \delta$. Hence,
from \eqref{ineq:3.1000}, applying Theorem \ref{decay_1}, we get
\begin{align*} 
& \left\||\xi|^{\alpha}  \widehat{K}_1(t, 0, \xi) \widehat{u}_1(\xi) \right\|_{L^2 \left (|\xi| \leq \Lambda(t)\right )}^2 \\
\lesssim & \frac{1}{b^2(t)} \int_{|\xi| \leq \Lambda(t) }|\xi|^{2\alpha- 4 \delta} \exp \left(-2|\xi|^{ \sigma} t \right) \left|\widehat{u_1}(\xi)\right|^2 d \xi 
=  \frac{1}{b^2(t)} \int_{|\xi| \leq \Lambda(t)} \exp \left(-2|\xi|^{\sigma } t\right) \left|\widehat{\vartheta}(\xi)\right|^2 d \xi \\
\lesssim & \frac{1}{b^2(t)} \left(P_{r^*(\vartheta)}(\vartheta)\right)\left(C_3+t \right)^{-\frac{r^{*}(\vartheta)+ \frac{n}{2}}{ \sigma/2}}+ \frac{1}{b^2(t)} \|\vartheta\|_{L^2}^2\left(C_3+t \right)^{-m},
\end{align*}
for any $m>\frac{r^{*}(\vartheta)+ \frac{n}{2}}{ \sigma/2}$ and for sufficiently large $C_3=$ const $>0$. By fixing $m>\frac{r^{*}(\vartheta)+ \frac{n}{2}}{ \sigma /2}$ and choosing $C_3$ large enough, we see that
$
\|\vartheta\|_{L^2}^2\left(C_3+t\right)^{-m} \lesssim\left(P_{r^*\left(u_1\right)}\left(u_1\right)+\left\|u_1\right\|_{L^2}^2\right)(1+t)^{-\frac{r^{*}(\vartheta)+ \frac{n}{2}}{ \sigma/2}}.
$
Therefore, recalling the definition of the norm $\|\cdot\|_{P}$ we obtain
\begin{equation} \label{ineq:3.20.1000}
\left\||\xi|^{\alpha}  \widehat{K}_1(t, 0, \xi) \widehat{u}_1(\xi) \right\|_{L^2 \left ( |\xi| \leq \Lambda(t)\right )}^2 \lesssim \|u_1\|_{P}^2 \frac{1}{b^2(t)} \left(1+t\right)^{-\frac{r^{*}(u_1) + \frac{n}{2} + \alpha - 2 \delta}{\sigma/2}}.
\end{equation}
\noindent{\it For $\alpha \geq 2 \delta$:} 
From \eqref{ineq:3.1000}, by Young's inequality, Lemma \ref{lemmae^-c} and Theorem \ref{decay_1}, we get
\begin{align} \label{ineq:3.2000}
& \left\||\xi|^{\alpha}  \widehat{K}_1(t, 0, \xi) \widehat{u}_1(\xi) \right\|_{L^2 \left (|\xi| \leq \Lambda(t)\right )}^2 
\lesssim  \frac{1}{b^2(t)} \int_{|\xi| \leq \Lambda(t)}|\xi|^{2\alpha- 4 \delta} \exp \left(-2|\xi|^{ \sigma } t\right) \left|\widehat{u_1}(\xi)\right|^2 d \xi \notag\\
\lesssim & \frac{1}{b^2(t)} \left\|F^{-1}\left( |\xi|^{\alpha - 2 \delta} \exp \left(-\frac{1}{2}|\xi|^{ \sigma} t\right)\right) \right\|_{L^1}^2 \cdot \left\|\exp \left(- \frac{1}{2} |\xi|^{ \sigma} t\right) \widehat{u}_1(\xi)\right\|_{L^2}^2 \notag\\
\lesssim & \|u_1\|_{P}^2 \frac{1}{b^2(t)} \left(1+t\right)^{-\frac{r^{*}(u_1) + \frac{n}{2} + \alpha - 2 \delta}{\sigma/2}}.
\end{align}
Finally,  \eqref{ineq:3.1400} and Theorem \ref{decay_3}  imply that
\begin{align} \label{ineq:3.2100}
\left\||\xi|^\alpha  \widehat{K}_0(t, 0, \xi) \widehat{u}_0(\xi) \right\|_{L^2 \left (|\xi| \leq \Lambda(t)\right )}^2 
\lesssim  \int_{|\xi| \leq \Lambda(t)}|\xi|^{2\alpha} \exp \left(-2|\xi|^{ \sigma} t\right) \left|\widehat{u_0}(\xi)\right|^2 d \xi 
\lesssim  \|u_0\|_{P, \alpha}^2 \left(1+t\right)^{-\frac{r^{*}(u_0) + \frac{n}{2} + \alpha}{\sigma /2}}.
\end{align}
From \eqref{ineq:3.18}-\eqref{ineq:3.2100} and Parseval's identity, we have
\begin{align*}
\|u(t,\cdot)\|_{\dot{H}^\alpha} \lesssim & \|u_1\|_{P} \frac{1}{b(t)} \left(1+\mathcal{B}(0,t)\right)^{-\frac{r^{*}(u_1) + \frac{n}{2} + \alpha - 2 \delta}{2\sigma - 2\delta}} + \|u_0\|_{P, \alpha} \left(1+\mathcal{B}(0,t)\right)^{-\frac{r^{*}(u_0) + \frac{n}{2} + \alpha}{2\sigma - 2\delta}} \\
+& \|u_1\|_{P} \frac{1}{b(t)} \left(1+t\right)^{-\frac{r^{*}(u_1) + \frac{n}{2} + \alpha - 2 \delta}{2\sigma - 2\delta}} + \|u_0\|_{P, \alpha} \left(1+t\right)^{-\frac{r^{*}(u_0) + \frac{n}{2} + \alpha}{2\sigma - 2\delta}}.
\end{align*}
$\bullet$ {\it Estimating for $\|u_t(t, \cdot)\|_{L^2}$.}\\
From \eqref{ineq:3.9} and the definition of $\lambda(t)$ in \eqref{lamda}, applying Parseval’s identity we get
\begin{align} \label{ineq:3.24}
 \left\|\partial_t \widehat{K}_1(t, 0, \xi) \widehat{u}_1(\xi) \right\|_{L^2 \left (|\xi| \geq \Omega(0, t)\right )}^2
\lesssim  & \int_{|\xi| \geq \Omega(0, t)} \exp \left(-(1-2 \beta)|\xi|^{2 \delta} \widehat{\mathcal{B}}(0,t) \right) \left|\widehat{u_1}(\xi)\right|^2 d \xi \notag\\
\lesssim & \exp \left(-(1-2 \beta) \left(\Omega(0, t) \right)^{2 \delta} \widehat{\mathcal{B}}(0,t) \right) \| u_1 \|_{L^2}^2.
\end{align}
Meanwhile, from \eqref{ineq:3.13}  we obtain
\begin{align} \label{ineq:3.25}
& \left\|\partial_t \widehat{K}_0(t, 0, \xi) \widehat{u}_0(\xi) \right\|_{L^2 \left (|\xi| \geq \Omega(0, t)\right )}^2 
\lesssim  \int_{|\xi| \geq \Omega(0, t)} |\xi|^{2\sigma} \exp \left(-(1-2 \beta)|\xi|^{2 \delta} \widehat{\mathcal{B}}(0,t) \right) \left|\widehat{u_0}(\xi)\right|^2 d \xi \notag\\
\lesssim &  \exp \left(-(1-2 \beta) \left(\Omega(0, t) \right)^{2 \delta} \widehat{\mathcal{B}}(0,t) \right) \| u_0 \|_{H^\sigma}^2.
\end{align}
Let us estimate for   $\Lambda(t) \leq |\xi| \leq \Omega(0, t)$.  Young's inequality, \eqref{ineq:3.11}, Lemma \ref{lemmae^-c} and Theorem \ref{decay_1} yield
\begin{align} \label{ineq:3.26}
& \left\|\partial_t \widehat{K}_1(t, 0, \xi) \widehat{u}_1(\xi) \right\|_{L^2 \left (\Lambda(t) \leq |\xi| \leq \Omega(0, t)\right )}^2 \notag\\
\lesssim & \left(\frac{1}{b(0)}+\frac{1}{b(t)} \right)^2 \frac{1}{b^2(t)} \int_{\Lambda(t) \leq |\xi| \leq \Omega(0, t)}|\xi|^{4\sigma - 8 \delta} \exp \left(-2C^{\prime}|\xi|^{2 \sigma - 2 \delta} \mathcal{B}(0,t)\right) \left|\widehat{u_1}(\xi)\right|^2 d \xi \notag\\
\lesssim & \left(\frac{1}{b(0)}+\frac{1}{b(t)} \right)^2 \frac{1}{b^2(t)} \left\|F^{-1}\left(|\xi|^{2 \sigma - 4 \delta} \exp \left(-\frac{C^{\prime}}{2}|\xi|^{2 \sigma - 2 \delta} \mathcal{B}(0,t)\right)\right) \right\|_{L^1}^2 \notag\\ 
& \times \left\|\exp \left(-\frac{C^{\prime}}{2}|\xi|^{2 \sigma - 2 \delta} \mathcal{B}(0,t)\right) \widehat{u}_1(\xi)\right\|_{L^2}^2 
\lesssim \|u_1\|_{P}^2 \left(\frac{1}{b(0)}+\frac{1}{b(t)} \right)^2 \frac{1}{b^2(t)}  \left(1+\mathcal{B}(0,t)\right)^{-\frac{r^{*}(u_1) + \frac{n}{2} + 2 \sigma - 4 \delta}{\sigma - \delta}}.
\end{align}
From \eqref{ineq:3.15},  Young's inequality, Lemma \ref{lemmae^-c} and Theorem \ref{decay_3}, we obtain
\begin{align} \label{ineq:3.27}
& \left\| \partial_t \widehat{K}_0(t, 0, \xi) \widehat{u}_0(\xi) \right\|_{L^2 \left (\Lambda(t) \leq |\xi| \leq \Omega(0, t)\right )}^2 
\lesssim  \frac{1}{b^2(t)} \int_{\Lambda(t) \leq |\xi| \leq \Omega(0, t)} |\xi|^{4 \sigma - 4 \delta} \exp \left(-2C^{\prime}|\xi|^{2 \sigma - 2 \delta} \mathcal{B}(0,t)\right) \left|\widehat{u_0}(\xi)\right|^2 d \xi \notag\\
\lesssim & \frac{1}{b^2(t)} \left\| F^{-1}\left(|\xi|^{\sigma - 2 \delta} \exp \left(-\frac{C^{\prime}}{2}|\xi|^{2 \sigma - 2 \delta} \mathcal{B}(0,t)\right)\right) \right\|_{L^1}^2  \left\||\xi|^\sigma \exp \left(-\frac{C^{\prime}}{2}|\xi|^{2 \sigma - 2 \delta} \mathcal{B}(0,t)\right) \widehat{u}_0(\xi)\right\|_{L^2}^2 \notag\\
\lesssim & \|u_0\|_{P, \sigma}^2 \frac{1}{b^2(t)} \left(1+\mathcal{B}(0,t)\right)^{-\frac{r^{*}(u_0) + \frac{n}{2} + 2 \sigma - 2 \delta}{\sigma - \delta}}.
\end{align}
Let us proceed  with small frequencies $|\xi| \leq \Lambda(t)$. From \eqref{ineq:3.1100} and Theorem \ref{decay_1} we get
\begin{align} \label{ineq:3.2600} \left\|\partial_t \widehat{K}_1(t, 0, \xi) \widehat{u}_1(\xi) \right\|_{L^2 \left (|\xi| \leq \Lambda(t)\right )}^2 
\lesssim \int_{|\xi| \leq \Lambda(t)} \exp \left(-2|\xi|^{2 \delta} \widehat{\mathcal{B}}(0, t)\right) \left|\widehat{u_1}(\xi)\right|^2 d \xi 
\lesssim \|u_1\|_{P}^2 \left(1+\widehat{\mathcal{B}}(0,t)\right)^{-\frac{r^{*}(u_1) + \frac{n}{2}}{\delta}}.
\end{align}
Similarly, from \eqref{ineq:3.1101} and Theorem \ref{decay_1} we  also can obtain
\begin{align} \label{ineq:3.2601} \left\|\partial_t \widehat{K}_1(t, 0, \xi) \widehat{u}_1(\xi) \right\|_{L^2 \left (|\xi| \leq \Lambda(t)\right )}^2 \lesssim \|u_1\|_{P}^2 (1+t)^{-\frac{r^{*}(u_1) + \frac{n}{2}}{\sigma/2}}.
\end{align}
Meanwhile, from \eqref{ineq:3.1500}, Young's inequality, Lemma \ref{lemmae^-c} and Theorem \ref{decay_3}, we get
\begin{align} \label{ineq:3.2700}
& \left\| \partial_t \widehat{K}_0(t, 0, \xi) \widehat{u}_0(\xi) \right\|_{L^2 \left (\Lambda(t) \leq |\xi| \leq \Omega(0, t)\right )}^2 
\lesssim  \frac{1}{b^2(t)} \int_{|\xi| \leq \Lambda(t)} |\xi|^{4 \sigma - 4 \delta} \exp \left(-2|\xi|^{2 \sigma - 2 \delta} t\right) \left|\widehat{u_0}(\xi)\right|^2 d \xi \notag\\
\lesssim & \frac{1}{b^2(t)} \left\| F^{-1}\left(|\xi|^{\sigma - 2 \delta} \exp \left(-\frac{1}{2}|\xi|^{2 \sigma - 2 \delta} t\right)\right) \right\|_{L^1}^2  \left\||\xi|^\sigma \exp \left(-\frac{1}{2}|\xi|^{2 \sigma - 2 \delta} t\right) \widehat{u}_0(\xi)\right\|_{L^2}^2 \notag\\
\lesssim & \|u_0\|_{P, \sigma}^2 \frac{1}{b^2(t)} \left(1+t\right)^{-\frac{r^{*}(u_0) + \frac{n}{2} + 2 \sigma - 2 \delta}{\sigma/2}}.
\end{align}
From \eqref{ineq:3.24}, \eqref{ineq:3.25}, \eqref{ineq:3.26}, \eqref{ineq:3.27}, \eqref{ineq:3.2600} and \eqref{ineq:3.2700}, by applying Parseval's identity and using $\delta\in [0,\sigma/2],\\ \frac{n}{2} + r^{*}(u_1) > 2 \delta$, we obtain the required estimate:
\begin{align*}
&\|u_t(t,\cdot)\|_{L^2} \\ \lesssim & \|u_1\|_{P} \left(\frac{1}{b(0)}+\frac{1}{b(t)} \right) \frac{1}{b(t)}  \left(1+\mathcal{B}(0,t)\right)^{-\frac{r^{*}(u_1) + \frac{n}{2} + 2 \sigma - 4 \delta}{2\sigma - 2\delta}} + \|u_0\|_{P, \sigma} \frac{1}{b(t)} \left(1+\mathcal{B}(0,t)\right)^{-\frac{r^{*}(u_0) + \frac{n}{2} + 2 \sigma - 2 \delta}{2\sigma - 2\delta}} \\
& + \|u_1\|_{P} \left(1+\widehat{\mathcal{B}}(0,t)\right)^{-\frac{r^{*}(u_1) + \frac{n}{2} + 2 \sigma - 4 \delta}{2\sigma - 2\delta}} + \|u_0\|_{P, \sigma} \frac{1}{b(t)} \left(1+t\right)^{-\frac{r^{*}(u_0) + \frac{n}{2} + 2 \sigma - 2 \delta}{2\sigma - 2\delta}}.
\end{align*}
The proof of Theorem \ref{theorem:1.1} is now complete.
\end{proof}

\section{Global (in time) existence of solutions} \label{section_5}
If  $K_0(t, 0, x)$ and $K_1(t, 0, x)$ are the fundamental solutions to the corresponding linear equation of \eqref{eq:1.0}, then
$
u^{\operatorname{lin}}:=K_0(t, 0, x) *_{(x)} u_0(x)+K_1(t, 0, x) *_{(x)} u_1(x)
$
is the solution to linear Cauchy problem \eqref{eq:2.1}. For $T>0$, we define the following operator $N: u \in X(T) \mapsto N u=N u(t, x):=u^{\operatorname{lin}}(t, x)+u^{\text {nl }}(t, x)$,
where $X(T)$ is an evolution space which will be constructed,  and $u^{\text {nl }}(t, x)$ is expressed by the  integral operator:
$
u^{\mathrm{nl}}(t, x):=\displaystyle\int_0^t K_1(t, s, x) *_{(x)}|u(s, x)|^p d s
$. The global (in time) Sobolev solution to semi-linear Cauchy problem \eqref{eq:1.0} is considered as a fixed point of the operator $N$. For the existence of this fixed point, we will show that the solution mapping $N$ satisfies the following  estimates:
\begin{equation} \label{ineq:4.1}
\|N u\|_{X(T)} \lesssim\left\|\left(u_0, u_1\right)\right\|_{D^\sigma}+\|u\|_{X(T)}^p,
\end{equation}
\begin{equation} \label{ineq:4.2}
\|N u-N v\|_{X(T)} \lesssim\|u-v\|_{X(T)}\left(\|u\|_{X(T)}^{p-1}+\|v\|_{X(T)}^{p-1}\right),
\end{equation}
where the data space $D^\sigma$ is fixed in the statement of Theorem \ref{theorem:1.3}. Providing that $\left\|\left(u_0, u_1\right)\right\|_{D^\sigma}=\varepsilon$ is sufficiently small, then estimates \eqref{ineq:4.1} and \eqref{ineq:4.2} imply the existence of a unique local (in time) large data solution and a unique global (in time) small data solution in $X(T)$ by a standard argument.
\begin{proof}[Proof of Theorem 1.2]
We will consider only the case $\delta \in (0, \frac{\sigma}{2}]$, since the case with $\delta=0$ can be processed in a similar fashion. We study separately the problems with the $b'(t)\ge 0$ and $b'(t)< 0$. \\
\textbf{A. The case when} $\mathbf{b^{\prime}(t) \geq 0}$. 

In this case it is obvious that $1+\mathcal{B}(s, t) \leq 1+t-s \leq 1 + \widehat{\mathcal{B}}(s, t)$. 
\begin{proposition} \label{proposition:1.3}
Assume that $\delta \in\left(0, \frac{\sigma}{2}\right]$. Let $h \in L^\omega \cap L^2$, with $1 \leq \omega < 2$. We assume further that $\frac{1}{\omega}-\frac{1}{2}-\frac{2 \delta}{n} \geq 0$. Then for all $\alpha \in [0, \sigma]$ the following estimates are valid
\begin{align*}
&\left\||D|^\alpha\left(\mathcal{F}^{-1}\left(\widehat{K}_1(t-s, \xi)\right) * h\right)\right\|_{L^2} \lesssim \frac{1}{b(t)} (1+\mathcal{B}(s,t))^{-\frac{n}{2(\sigma-\delta)}\left(\frac{1}{\omega}-\frac{1}{2}\right)-\frac{\alpha}{2(\sigma-\delta)}+\frac{\delta}{\sigma-\delta}}\left(\|h\|_{L^\omega}+\|h\|_{L^2}\right), \\
&\left\|\partial_t\left(\mathcal{F}^{-1}\left(\widehat{K}_1(t-s, \xi)\right) * h\right)\right\|_{L^2} \lesssim \frac{1}{b(s)} \frac{1}{b(t)}(1+\mathcal{B}(s,t))^{-\frac{n}{2(\sigma-\delta)}\left(\frac{1}{\omega}-\frac{1}{2}\right)+\frac{\delta}{\sigma-\delta}-1}\left(\|h\|_{L^\omega}+\|h\|_{L^2}\right) .
\end{align*}
\end{proposition}
\begin{proof}[Proof of Proposition 5.1]
From Lemma \ref{lemmaB}, Lemma \ref{lemmaB^} and using $b^{\prime}(t) \geq 0$,  we have
\begin{align} \label{approx_1.1}
1+\widehat{\mathcal{B}}(s, t) & \approx (t+1)b(t)-(s+1)b(s)=b^2(t) \frac{t+1}{b(t)}- b^2(s) \frac{s+1}{b(s)} \notag\\
& \gtrsim b^2(t) \left( \frac{t+1}{b(t)}- \frac{s+1}{b(s)} \right) \gtrsim b^2(t) \left(1+\mathcal{B}(s, t) \right).
\end{align}
By applying the estimates in Corollary \ref{Corollary_4.1} and Corollary \ref{Corollary_4.3}, the statement in Proposition above is entirely demonstrated in the same manner as in Proposition 1 of \cite{Dao2023}.
\end{proof}
We define the  solutions space $X(T):=\mathcal{C}\left([0, T], H^\sigma\right) \cap \mathcal{C}^1 \left([0, T], L^2 \right)
$
with the norm
\begin{align*}
&\|u\|_{X(T)}:= \sup _{0 \leq t \leq T}\left[\left(1+\mathcal{B}(0,t)\right)^{\frac{\frac{n}{2} + \min \left\{ r^{*}(u_0), r^{*}(u_1) - 2 \delta \right\} }{2 \sigma - 2 \delta}} \|u(t, \cdot)\|_{L^2} \right.\\
&\left.
+(1+\mathcal{B}(0, t))^{\frac{\frac{n}{2} + \sigma + \min \left\{ r^{*}(u_0), r^{*}(u_1) - 2 \delta \right\} }{2 \sigma - 2 \delta}}\left\||D|^\sigma u(t, \cdot)\right\|_{L^2} +b(t) \left(1+\mathcal{B}(0,t)\right)^{\frac{\frac{n}{2} + 2 \sigma - 2 \delta + \min \left\{ r^{*}(u_0) , r^{*}(u_1)- 2 \delta \right\} }{2 \sigma - 2 \delta}} \|u_t(t, \cdot)\|_{L^2} \right].
\end{align*}
With $\omega:=\frac{n}{n+\min \left\{ r^{*}(u_0), r^{*}(u_1) - 2 \delta \right\} + 2 \delta}$,  it is easy to see that $1 \leq \omega < 2$. The application of the fractional Gagliardo-Nirenberg inequality with interpolation exponents $\theta_1(2 p):=\frac{n(p-1)+2\gamma p}{2 \sigma p}, \quad \theta_2(\omega p):=\frac{n(\omega p-2)+2 \omega \gamma p}{2 \omega \sigma p}$ from Lemma \ref{Gagliardo} and the definition of the evolution space leads to
\begin{align} \label{ineq:4.3}
\left\| \left| |D|^{\gamma} u(\tau, \cdot)\right|^p\right\|_{L^2}&=\left\||D|^{\gamma} u(\tau, \cdot)\right\|_{L^{2 p}}^p  \lesssim (1+\mathcal{B}(0, \tau))^{-\frac{n + \min \left\{ r^{*}(u_0), r^{*}(u_1) - 2 \delta \right\}+\gamma}{2 \sigma - 2 \delta} p+\frac{n}{4 \sigma - 4 \delta}}\|u\|_{X(T)}^p,\\
\left\| \left| |D|^{\gamma} u(\tau, \cdot)\right|^p\right\|_{L^\omega \cap L^2}&=\left\||D|^{\gamma} u(\tau, \cdot)\right\|_{L^{\omega p}}^p + \left\||D|^{\gamma} u(\tau, \cdot)\right\|_{L^{2 p}}^p \notag \\
&\lesssim(1+\mathcal{B}(0, \tau))^{-\frac{n + \min \left\{ r^{*}(u_0), r^{*}(u_1) - 2 \delta \right\}+\gamma}{2 \sigma - 2 \delta} p+\frac{n}{\omega ( 2 \sigma - 2 \delta)}}\|u\|_{X(T)}^p,\label{ineq:4.4}
\end{align}
provided that
$$
p \in\left[\frac{2}{\omega}, \infty\right) \text { if } n \leq 2 \sigma - 2 \gamma, \text { or } p \in\left[\frac{2}{\omega}, \frac{n}{n-2 \sigma+2 \gamma}\right] \text { if } 2 \sigma - 2 \gamma < n \leq \frac{4 \sigma - 4 \gamma}{2-\omega}.
$$
First, we will prove the inequality \eqref{ineq:4.1}. The estimates for solutions to \eqref{eq:2.1} from Theorem \ref{theorem:1.1} imply
$$
\begin{aligned}
\left\|u^{\operatorname{lin}}(t, \cdot)\right\|_{L^2} & \lesssim(1+\mathcal{B}(0, t))^{-\frac{\frac{n}{2} + \min \left\{ r^{*}(u_0), r^{*}(u_1) - 2 \delta \right\} }{2 \sigma - 2 \delta}}\left\|\left(u_0, u_1\right)\right\|_{D^\sigma}, \\
\left\||D|^\sigma u^{\operatorname{lin}}(t, \cdot)\right\|_{L^2} & \lesssim(1+\mathcal{B}(0, t))^{-\frac{\frac{n}{2} + \sigma + \min \left\{ r^{*}(u_0), r^{*}(u_1) - 2 \delta \right\} }{2 \sigma - 2 \delta}}\left\|\left(u_0, u_1\right)\right\|_{D^\sigma}, \\
\left\|u^{\operatorname{lin}}_t(t,\cdot)\right\|_{L^2} & \lesssim  \frac{1}{b(t)} \left(1+\mathcal{B}(0,t)\right)^{-\frac{\frac{n}{2} + 2 \sigma - 2 \delta + \min \left\{ r^{*}(u_0), r^{*}(u_1)- 2 \delta \right\} }{2 \sigma - 2 \delta}} \left\|\left(u_0, u_1\right)\right\|_{D^\sigma}.
\end{aligned}
$$
This means that the linear part fulfills
\begin{align*}
&(1+\mathcal{B}(0, t))^{\frac{\frac{n}{2} + \min \left\{ r^{*}(u_0), r^{*}(u_1) - 2 \delta \right\} }{2 \sigma - 2 \delta}}\left\|u^{\operatorname{lin}}(t, \cdot)\right\|_{L^2} 
+(1+\mathcal{B}(0, t))^{\frac{\frac{n}{2} + \sigma + \min \left\{ r^{*}(u_0), r^{*}(u_1) - 2 \delta \right\} }{2 \sigma - 2 \delta}}\left\||D|^\sigma u^{\operatorname{lin}}(t, \cdot)\right\|_{L^2} \\
&+  \left(1+\mathcal{B}(0,t)\right)^{\frac{\frac{n}{2} + 2 \sigma - 2 \delta + \min \left\{ r^{*}(u_0), r^{*}(u_1)- 2 \delta \right\} }{2 \sigma - 2 \delta}} \left\|u^{\operatorname{lin}}_t(t, \cdot)\right\|_{L^2} \lesssim \left\|\left(u_0, u_1\right)\right\|_{D^\sigma}.
\end{align*}
From the above we  see that $u^{\text {lin }} \in X(T)$.  It remains to prove that $
 \left\|u^{\text {nl }}\right\|_{X(T)} \lesssim\|u\|_{X(T)}^p$.
 
Since $b^{\prime}(t) \geq 0$ we may estimate the $L^2$ norm of $|D|^\sigma u^{\text {nl }}(t, \cdot)$  by applying $\left(L^2 \cap L^\omega\right)-L^2$ estimates in $[0, t / 2]$ and $L^2-L^2$ estimates in $[t / 2, t]$ from Proposition \ref{proposition:1.3} as follows:
\begin{align*}
\left\||D|^\sigma u^{\operatorname{nl}}(t, \cdot)\right\|_{L^2} \leq & \int_0^{\frac{t}{2}} b(t)^{-1}(1+\mathcal{B}(s, t))^{-\frac{n}{2 \sigma - 2 \delta}\left(\frac{1}{\omega}-\frac{1}{2}\right)-\frac{\sigma - 2 \delta}{2 \sigma - 2 \delta}}\left\| \left| |D|^{\gamma} u(s, \cdot)\right|^p\right\|_{L^{\omega} \cap L^2} d s \\
&+ \int_{\frac{t}{2}}^t b(t)^{-1}(1+\mathcal{B}(s, t))^{-\frac{\sigma - 2 \delta}{2 \sigma - 2 \delta}}\left\| \left| |D|^{\gamma} u(s, \cdot)\right|^p\right\|_{L^2} d s \\
\leq & \int_0^{\frac{t}{2}} b(s)^{-1}(1+\mathcal{B}(s, t))^{-\frac{n}{2 \sigma - 2 \delta}\left(\frac{1}{\omega}-\frac{1}{2}\right)-\frac{\sigma - 2 \delta}{2 \sigma - 2 \delta}}\left\| \left| |D|^{\gamma} u(s, \cdot)\right|^p\right\|_{L^{\omega} \cap L^2} d s \\
&+ \int_{\frac{t}{2}}^t b(s)^{-1}(1+\mathcal{B}(s, t))^{-\frac{\sigma - 2 \delta}{2 \sigma - 2 \delta}}\left\| \left| |D|^{\gamma} u(s, \cdot)\right|^p\right\|_{L^2} d s.
\end{align*}
On the one hand,  the first integral in the right hand part of the last inequality can be estimated as
\begin{align*}
& \int_0^{\frac{t}{2}} b(s)^{-1}(1+\mathcal{B}(s, t))^{-\frac{n}{2 \sigma - 2 \delta}\left(\frac{1}{\omega}-\frac{1}{2}\right)-\frac{\sigma - 2 \delta}{2 \sigma - 2 \delta}}\left\| \left| |D|^{\gamma} u(s, \cdot)\right|^p\right\|_{L^{\omega} \cap L^2} d s \\
\lesssim & \|u\|_{X(T)}^p(1+\mathcal{B}(0, t))^{-\frac{n}{2 \sigma - 2 \delta}\left(\frac{1}{\omega}-\frac{1}{2}\right)-\frac{\sigma - 2 \delta}{2 \sigma - 2 \delta}} 
 \int_0^{\frac{t}{2}} b(s)^{-1}(1+\mathcal{B}(0, s))^{-\frac{n + \min \left\{ r^{*}(u_0), r^{*}(u_1) - 2 \delta \right\}+\gamma}{2 \sigma - 2 \delta} p+\frac{n}{\omega ( 2 \sigma - 2 \delta)}} d s \\
\lesssim & \|u\|_{X(T)}^p(1+\mathcal{B}(0, t))^{-\frac{n}{2 \sigma - 2 \delta}\left(\frac{1}{\omega}-\frac{1}{2}\right)-\frac{\sigma - 2 \delta}{2 \sigma - 2 \delta}} = \|u\|_{X(T)}^p(1+\mathcal{B}(0, t))^{-\frac{\frac{n}{2} + \sigma + \min \left\{ r^{*}(u_0), r^{*}(u_1) - 2 \delta \right\} }{2 \sigma - 2 \delta}},
\end{align*}
by inequality \eqref{ineq:4.4}, Lemma \ref{lemmaB} and the definition of $\omega$. Since $p>p^*=\frac{n + 2\omega(\sigma-\delta)}{\omega n + \omega \min \left\{ r^{*}(u_0), r^{*}(u_1) - 2\delta \right\}+\omega \gamma}$, it follows  that $-\frac{n + \min \left\{ r^{*}(u_0), r^{*}(u_1) - 2 \delta \right\}+\gamma}{2 \sigma - 2 \delta} p+\frac{n}{\omega ( 2 \sigma - 2 \delta)} < -1$. On the other hand, for the second integral, using inequality \eqref{ineq:4.3}, Lemma \ref{lemmaB} and definition of $\omega$ we have
\begin{align*}
& \int_{\frac{t}{2}}^t b(s)^{-1}(1+\mathcal{B}(s, t))^{-\frac{\sigma-2\delta}{2 \sigma-2\delta}}\left\| \left| |D|^{\gamma} u(s, \cdot)\right|^p\right\|_{L^2} d s \\
\lesssim & \|u\|_{X(T)}^p(1+\mathcal{B}(0, t))^{-\frac{n + \min \left\{ r^{*}(u_0), r^{*}(u_1) - 2 \delta \right\}+\gamma}{2 \sigma - 2 \delta} p+\frac{n}{4 \sigma - 4 \delta}} \int_{\frac{t}{2}}^t b(s)^{-1}(1+\mathcal{B}(s, t))^{-\frac{\sigma-2\delta}{2 \sigma-2\delta}} d s \\ \lesssim & \|u\|_{X(T)}^p(1+\mathcal{B}(0, t))^{-\frac{n + \min \left\{ r^{*}(u_0), r^{*}(u_1) - 2 \delta \right\}+\gamma}{2 \sigma - 2 \delta} p+\frac{n}{4 \sigma - 4 \delta}}(1+\mathcal{B}(0, t))^{1-\frac{\sigma-2\delta}{2 \sigma-2\delta}} \\
\lesssim & \|u\|_{X(T)}^p(1+\mathcal{B}(0, t))^{-\frac{\frac{n}{2} + \sigma + \min \left\{ r^{*}(u_0), r^{*}(u_1) - 2 \delta \right\} }{2 \sigma - 2 \delta}},
\end{align*}
for $p>\frac{n + 2\omega(\sigma-\delta)}{\omega n + \omega \min \left\{ r^{*}(u_0), r^{*}(u_1) - 2\delta \right\}+\omega \gamma}$. Therefore,
$$
\left\| |D|^\sigma u^{\operatorname{nl}}(t, \cdot)\right\|_{L^2} \lesssim(1+\mathcal{B}(0, t))^{-\frac{\frac{n}{2} + \sigma + \min \left\{ r^{*}(u_0), r^{*}(u_1) - 2 \delta \right\} }{2 \sigma - 2 \delta}}\|u\|_{X(T)}^p.$$
In the same way we can derive
\begin{equation*}
\left\|u^{\operatorname{nl}}(t, \cdot)\right\|_{L^2} \lesssim (1+\mathcal{B}(0, t))^{-\frac{\frac{n}{2} + \min \left\{ r^{*}(u_0), r^{*}(u_1) - 2 \delta \right\} }{2 \sigma - 2 \delta}}\|u\|_{X(T)}^p.
\end{equation*}
In the same manner as in estimating $
\left\| |D|^\sigma u^{\operatorname{nl}}(t, \cdot)\right\|_{L^2}$  we can obtain
$$
\left\| u^{\operatorname{nl}}_t(t, \cdot)\right\|_{L^2} \lesssim \frac{1}{b(t)}  \left(1+\mathcal{B}(0, t) \right)^{-\frac{\frac{n}{2} + 2\sigma- 2\delta + \min \left\{ r^{*}(u_0), r^{*}(u_1) - 2 \delta \right\} }{2 \sigma - 2 \delta}}\|u\|_{X(T)}^p.$$
From the definition of the norm $X(T)$ the  inequality \eqref{ineq:4.1} is verified.

Let us prove  \eqref{ineq:4.2}. We see that
$$
\|N u-N v\|_{X(T)}=\left\|\int_0^t K_1(t, s, x) *_{(x)}\left( \left ||D|^{\gamma} u(s, x)\right|^p-\left| |D|^{\gamma} v(s, x)\right|^p\right) d s\right\|_{X(T)}.$$
Thanks to the estimates for the solutions from Proposition \ref{proposition:1.3} we can estimate
\begin{align} \label{ineq:4.5}
& \left\||D|^{\sigma} K_1(t, s, x) *_{(x)}\left( \left| |D|^{\gamma} u(s, x)\right|^p-\left| |D|^{\gamma} v(s, x)\right|^p\right)\right\|_{L^2} \notag\\
\lesssim & \begin{cases}b(s)^{-1}(1+\mathcal{B}(s, t))^{-\frac{n}{2 \sigma - 2 \delta}\left(\frac{1}{\omega}-\frac{1}{2}\right)-\frac{\sigma - 2 \delta}{2 \sigma - 2 \delta}} 
 \left\| \left| |D|^{\gamma} u(s, x)\right|^p-\left| |D|^{\gamma} v(s, x)\right|^p\right\|_{L^\omega \cap L^2} & \text { if } s \in[0, t / 2], \\
b(s)^{-1}(1+\mathcal{B}(s, t))^{-\frac{\sigma - 2 \delta}{2 \sigma - 2 \delta}} 
 \left\| \left| |D|^{\gamma} u(s, x)\right|^p-\left| |D|^{\gamma} v(s, x)\right|^p\right\|_{L^2} & \text { if } s \in[t / 2, t] .\end{cases}
\end{align}
Since
\begin{align*}
\left| \left| |D|^{\gamma} v(s, x)\right|^p-\left| |D|^{\gamma} v(s, x)\right|^p \right| 
\lesssim  \left| |D|^{\gamma} u(s, x)-|D|^{\gamma} v(s, x)\right| \left(\left| |D|^{\gamma} u(s, x)|^{p-1}+||D|^{\gamma} v(s, x)\right|^{p-1}\right),
\end{align*}
by Hölder's inequality we obtain
\begin{align*}
 \left\| \left| |D|^{\gamma} u(s, \cdot)\right|^p-\left| |D|^{\gamma} v(s, \cdot)\right|^p\right\|_{L^\omega} 
\lesssim \left\| |D|^{\gamma} u(s, \cdot)-|D|^{\gamma} v(s, \cdot)\right\|_{L^{\omega p}}\left(\left\| |D|^{\gamma} u(s, \cdot)\|_{L^{\omega p}}^{p-1}+\||D|^{\gamma} v(s, \cdot)\right\|_{L^{\omega p}}^{p-1}\right), \\
 \left\| \left| |D|^{\gamma} u(s, \cdot)\right|^p-\left| |D|^{\gamma} v(s, \cdot)\right|^p\right\|_{L^2} 
\lesssim  \left\| |D|^{\gamma} u(s, \cdot)-|D|^{\gamma} v(s, \cdot)\right\|_{L^{2 p}}\left( \left\| |D|^{\gamma} u(s, \cdot)\right\|_{L^{2 p}}^{p-1}+\left\| |D|^{\gamma} v(s, \cdot)\right\|_{L^{2 p}}^{p-1}\right).
\end{align*}
In a similar way to the proof of \eqref{ineq:4.1},  we apply again the fractional Gagliardo-Nirenberg inequality from Proposition \ref{Gagliardo} to the following terms $
\left \||D|^{\gamma} u(s, \cdot)-|D|^{\gamma} v(s, \cdot)\right\|_{L^h},  \left\||D|^{\gamma} u(s, \cdot)\right\|_{L^h} $ and $\left\||D|^{\gamma} v(s, \cdot)\right\|_{L^h} $,
with $h=\omega p$ and $h=2 p$. After plugging these estimates in \eqref{ineq:4.5},  by the same ideas as it was done in estimating the norms $\left\||D|^\sigma u^{\text {nl }}(t, \cdot)\right\|_{L^2}$ we may conclude inequality \eqref{ineq:4.2}. \\
\textbf{B.  The case when} $\mathbf{b^{\prime}(t) < 0}$.  In that case we see that $1+\mathcal{B}(s, t) \geq 1+t-s \geq 1 + \widehat{\mathcal{B}}(s, t)$.
\begin{proposition} \label{proposition_b(t)_increasing}
Let $\delta \in\left(0, \frac{\sigma}{2}\right]$, $h \in L^\omega \cap L^2$, with $1 \leq \omega < 2$. We assume further that $\frac{1}{\omega}-\frac{1}{2}-\frac{2\sigma+ 2\delta}{n} \geq 0$. Then for all $\alpha \in [0, \sigma]$ the following estimates are valid
\begin{align*}
&\left\||D|^\alpha\left(\mathcal{F}^{-1}\left(\widehat{K}_1(t-s, \xi)\right) * h\right)\right\|_{L^2}  \lesssim b(t) \left(1+\widehat{\mathcal{B}}(s,t)\right)^{-\frac{n}{2(\sigma-\delta)}\left(\frac{1}{\omega}-\frac{1}{2}\right)-\frac{\alpha}{2(\sigma-\delta)}+\frac{\delta}{\sigma-\delta}}\left(\|h\|_{L^\omega}+\|h\|_{L^2}\right), \\
&\left\|\partial_t\left(\mathcal{F}^{-1}\left(\widehat{K}_1(t-s, \xi)\right) * h\right)\right\|_{L^2} \lesssim b^2(t)\left(1+\widehat{\mathcal{B}}(s,t)\right)^{-\frac{n}{2(\sigma-\delta)}\left(\frac{1}{\omega}-\frac{1}{2}\right)+\frac{\delta}{\sigma-\delta}-1}\left(\|h\|_{L^\omega}+\|h\|_{L^2}\right) .
\end{align*}
\end{proposition}
\begin{proof}[Proof of Proposition \ref{proposition_b(t)_increasing}]
From Lemma \ref{lemmaB}, Lemma \ref{lemmaB^} and $b^{\prime}(t) <0$, we have:
\begin{equation*} 
1+\widehat{\mathcal{B}}(s, t) \approx (t+1)b(t)-(s+1)b(s) \lesssim b(t) (1+t-s),
\end{equation*}
\begin{align*}
1+\mathcal{B}(s, t) & \approx \frac{t+1}{b(t)}-\frac{s+1}{b(s)}=\frac{1}{b^2(t)} (t+1)b(t)- \frac{1}{b^2(s)} (s+1)b(s) \notag\\
& \gtrsim \frac{1}{b^2(t)} \left( (t+1)b(t)- (s+1)b(s) \right) \gtrsim \frac{1}{b^2(t)} \left(1+\widehat{\mathcal{B}}(s, t) \right).
\end{align*}
Therefore,
\begin{equation} \label{approx_1.2}
1+t-s \gtrsim \frac{1}{b(t)} \left(1+\widehat{\mathcal{B}}(s, t) \right),
\end{equation}
\begin{equation} \label{approx_1.3}
1+\mathcal{B}(s, t) \gtrsim \frac{1}{b^2(t)} \left(1+\widehat{\mathcal{B}}(s, t) \right).
\end{equation}
Now by applying estimates \eqref{ineq:3.1101} and \eqref{ineq:3.11.2} from   Corollary \ref{Corollary_4.1} and Corollary \ref{Corollary_4.3}, the statement of  Proposition \ref{proposition_b(t)_increasing}  is proved in the same manner as in Proposition 1 of \cite{Dao2023}.
\end{proof}
We define  the solution space $X(T)$ by $
X(T):=\mathcal{C}\left([0, T], H^\sigma\right) \cap \mathcal{C}^1 \left([0, T], L^2 \right)$,
with the  norm
\begin{align*}
\|u\|_{X(T)}:= 
\sup _{0 \leq t \leq T}&\left[\left(1+\widehat{\mathcal{B}}(0,t)\right)^{\frac{\frac{n}{2} + \min \left\{ r^{*}(u_0), r^{*}(u_1) - 2 \delta \right\} }{2 \sigma - 2 \delta}} \|u(t, \cdot)\|_{L^2} \right.\\
&\;\;\; \left.
+\left(1+\widehat{\mathcal{B}}(0,t)\right)^{\frac{\frac{n}{2} + \sigma + \min \left\{ r^{*}(u_0), r^{*}(u_1) - 2 \delta \right\} }{2 \sigma - 2 \delta}}\left\||D|^\sigma u(t, \cdot)\right\|_{L^2} \right.\\
&\;\;\; \left.
+ \frac{1}{b(t)} \left(1+\widehat{\mathcal{B}}(0,t)\right)^{\frac{\frac{n}{2} + 2 \sigma - 2 \delta + \min \left\{ r^{*}(u_0), r^{*}(u_1)- 2 \delta \right\} }{2 \sigma - 2 \delta}} \|u_t(t, \cdot)\|_{L^2} \right].
\end{align*}

The  fractional Gagliardo-Nirenberg inequality with  exponents $\theta_1(2 p):=\frac{n(p-1)+2\gamma p}{2 \sigma p}, \; \theta_2(\omega p):=\frac{n(\omega p-2)+2 \omega \gamma p}{2 \omega \sigma p}$ from Lemma \ref{Gagliardo} and the definition of $X(t)$ imply that
\begin{align} \label{ineq:4.3}
\left\| \left| |D|^{\gamma} u(\tau, \cdot)\right|^p\right\|_{L^2}=\left\||D|^{\gamma} u(\tau, \cdot)\right\|_{L^{2 p}}^p  \lesssim \left(1+\widehat{\mathcal{B}}(0,\tau)\right)^{-\frac{n + \min \left\{ r^{*}(u_0), r^{*}(u_1) - 2 \delta \right\}+\gamma}{2 \sigma - 2 \delta} p+\frac{n}{4 \sigma - 4 \delta}}\|u\|_{X(T)}^p,
\end{align}
\begin{align} \label{ineq:4.4}
\left\| \left| |D|^{\gamma} u(\tau, \cdot)\right|^p\right\|_{L^\omega \cap L^2}&=\left\||D|^{\gamma} u(\tau, \cdot)\right\|_{L^{\omega p}}^p + \left\||D|^{\gamma} u(\tau, \cdot)\right\|_{L^{2 p}}^p \notag \\
&\lesssim \left(1+\widehat{\mathcal{B}}(0,\tau)\right)^{-\frac{n + \min \left\{ r^{*}(u_0), r^{*}(u_1) - 2 \delta \right\}+\gamma}{2 \sigma - 2 \delta} p+\frac{n}{\omega ( 2 \sigma - 2 \delta)}}\|u\|_{X(T)}^p,
\end{align}
provided that $
p \in\left[\frac{2}{\omega}, \infty\right)$ if $ n \leq 2 \sigma - 2 \gamma$,  or $p \in\left[\frac{2}{\omega}, \frac{n}{n-2 \sigma+2 \gamma}\right]$ if $2 \sigma - 2 \gamma < n \leq \frac{4 \sigma - 4 \gamma}{2-\omega}$. 

Let us prove inequality \eqref{ineq:4.1}. Similarly as in the previous case we have $$
\begin{aligned}
\left\|u^{\operatorname{lin}}(t, \cdot)\right\|_{L^2} & \lesssim \left(1+\widehat{\mathcal{B}}(0,\tau)\right)^{-\frac{\frac{n}{2} + \min \left\{ r^{*}(u_0), r^{*}(u_1) - 2 \delta \right\} }{2 \sigma - 2 \delta}}\left\|\left(u_0, u_1\right)\right\|_{D^\sigma}, \\
\left\||D|^\sigma u^{\operatorname{lin}}(t, \cdot)\right\|_{L^2} & \lesssim \left(1+\widehat{\mathcal{B}}(0,\tau)\right)^{-\frac{\frac{n}{2} + \sigma + \min \left\{ r^{*}(u_0), r^{*}(u_1) - 2 \delta \right\} }{2 \sigma - 2 \delta}}\left\|\left(u_0, u_1\right)\right\|_{D^\sigma}, \\
\left\|u^{\operatorname{lin}}_t(t,\cdot)\right\|_{L^2} & \lesssim b(t)\left(1+\widehat{\mathcal{B}}(0,t)\right)^{-\frac{\frac{n}{2} + 2 \sigma - 2 \delta + \min \left\{ r^{*}(u_0), r^{*}(u_1)- 2 \delta \right\} }{2 \sigma - 2 \delta}} \left\|\left(u_0, u_1\right)\right\|_{D^\sigma}.
\end{aligned}
$$
This means that the linear part fulfills
\begin{align*}
& \left(1+\widehat{\mathcal{B}}(0,\tau)\right)^{\frac{\frac{n}{2} + \min \left\{ r^{*}(u_0), r^{*}(u_1) - 2 \delta \right\} }{2 \sigma - 2 \delta}}\left\|u^{\operatorname{lin}}(t, \cdot)\right\|_{L^2} + \left(1+\widehat{\mathcal{B}}(0,\tau)\right)^{\frac{\frac{n}{2} + \sigma + \min \left\{ r^{*}(u_0), r^{*}(u_1) - 2 \delta \right\} }{2 \sigma - 2 \delta}}\left\||D|^\sigma u^{\operatorname{lin}}(t, \cdot)\right\|_{L^2} \\
&+ \frac{1}{b(t)}\left(1+\widehat{\mathcal{B}}(0,t)\right)^{\frac{\frac{n}{2} + 2 \sigma - 2 \delta + \min \left\{ r^{*}(u_0), r^{*}(u_1)- 2 \delta \right\} }{2 \sigma - 2 \delta}} \left\|u^{\operatorname{lin}}_t(t, \cdot)\right\|_{L^2}~~ \lesssim \left\|\left(u_0, u_1\right)\right\|_{D^\sigma}.
\end{align*}
Thus $u^{\text {lin }} \in X(T)$. Now we show that $\left\|u^{\text {nl}}\right\|_{X(T)} \lesssim\|u\|_{X(T)}^p.$
We estimate  $L^2$ norm of $|D|^\sigma u^{\text {nl}}(t, \cdot)$  by applying $\left(L^2 \cap L^\omega\right)-L^2$ estimates in $[0, t / 2]$ and $L^2-L^2$ estimates in $[t / 2, t]$ from Proposition \ref{proposition:1.3}:  
\begin{align*}
\left\||D|^\sigma u^{\operatorname{nl}}(t, \cdot)\right\|_{L^2} \leq & \int_0^{\frac{t}{2}} b(t) \left(1+\widehat{\mathcal{B}}(s,t)\right)^{-\frac{n}{2 \sigma - 2 \delta}\left(\frac{1}{\omega}-\frac{1}{2}\right)-\frac{\sigma - 2 \delta}{2 \sigma - 2 \delta}}\left\| \left| |D|^{\gamma} u(s, \cdot)\right|^p\right\|_{L^{\omega} \cap L^2} d s \\
&+ \int_{\frac{t}{2}}^t b(t) \left(1+\widehat{\mathcal{B}}(s,t)\right)^{-\frac{\sigma - 2 \delta}{2 \sigma - 2 \delta}}\left\| \left| |D|^{\gamma} u(s, \cdot)\right|^p\right\|_{L^2} d s \\
\leq & \int_0^{\frac{t}{2}} b(s) \left(1+\widehat{\mathcal{B}}(s,t)\right)^{-\frac{n}{2 \sigma - 2 \delta}\left(\frac{1}{\omega}-\frac{1}{2}\right)-\frac{\sigma - 2 \delta}{2 \sigma - 2 \delta}}\left\| \left| |D|^{\gamma} u(s, \cdot)\right|^p\right\|_{L^{\omega} \cap L^2} d s \\
&+ \int_{\frac{t}{2}}^t b(s) \left(1+\widehat{\mathcal{B}}(s,t)\right)^{-\frac{\sigma - 2 \delta}{2 \sigma - 2 \delta}}\left\| \left| |D|^{\gamma} u(s, \cdot)\right|^p\right\|_{L^2} d s.
\end{align*}
The first integral above can be estimated as follows
\begin{align} \label{M_1(t)}
& \int_0^{\frac{t}{2}} b(s) \left(1+\widehat{\mathcal{B}}(s,t)\right)^{-\frac{n}{2 \sigma - 2 \delta}\left(\frac{1}{\omega}-\frac{1}{2}\right)-\frac{\sigma - 2 \delta}{2 \sigma - 2 \delta}}\left\| \left| |D|^{\gamma} u(s, \cdot)\right|^p\right\|_{L^{\omega} \cap L^2} d s \notag \\
\lesssim & \|u\|_{X(T)}^p \left(1+\widehat{\mathcal{B}}(0,t)\right)^{-\frac{n}{2 \sigma - 2 \delta}\left(\frac{1}{\omega}-\frac{1}{2}\right)-\frac{\sigma - 2 \delta}{2 \sigma - 2 \delta}}  \int_0^{\frac{t}{2}} b(s) \left(1+\widehat{\mathcal{B}}(0,s)\right)^{-\frac{n + \min \left\{ r^{*}(u_0), r^{*}(u_1) - 2 \delta \right\}+\gamma}{2 \sigma - 2 \delta} p+\frac{n}{\omega ( 2 \sigma - 2 \delta)}} d s \notag 
\\ 
\lesssim & \|u\|_{X(T)}^p \left(1+\widehat{\mathcal{B}}(0,t)\right)^{-\frac{n}{2 \sigma - 2 \delta}\left(\frac{1}{\omega}-\frac{1}{2}\right)-\frac{\sigma - 2 \delta}{2 \sigma - 2 \delta}} = \|u\|_{X(T)}^p \left(1+\widehat{\mathcal{B}}(0,t)\right)^{-\frac{\frac{n}{2} + \sigma + \min \left\{ r^{*}(u_0), r^{*}(u_1) - 2 \delta \right\} }{2 \sigma - 2 \delta}},
\end{align}
by inequality \eqref{ineq:4.4}, Lemma \ref{lemmaB} and the definition of $\omega$. Since $p>\frac{n + 2\omega(\sigma-\delta)}{\omega n + \omega \min \left\{ r^{*}(u_0), r^{*}(u_1) - 2\delta \right\}+\omega \gamma}$ it follows that $-\frac{n + \min \left\{ r^{*}(u_0), r^{*}(u_1) - 2 \delta \right\}+\gamma}{2 \sigma - 2 \delta} p+\frac{n}{\omega ( 2 \sigma - 2 \delta)} < -1$.  On the other hand, for the second integral using inequality \eqref{ineq:4.3}, Lemma \ref{lemmaB} and definition of $\omega$ we have
\begin{align*}
& \int_{\frac{t}{2}}^t b(s) \left(1+\widehat{\mathcal{B}}(s,t)\right)^{-\frac{\sigma - 2 \delta}{2 \sigma - 2 \delta}}\left\| \left| |D|^{\gamma} u(s, \cdot)\right|^p\right\|_{L^2} d s \\
\lesssim & \|u\|_{X(T)}^p \left(1+\widehat{\mathcal{B}}(0,t)\right)^{-\frac{n + \min \left\{ r^{*}(u_0), r^{*}(u_1) - 2 \delta \right\}+\gamma}{2 \sigma - 2 \delta} p+\frac{n}{4 \sigma - 4 \delta}} \int_{\frac{t}{2}}^t b(s) \left(1+\widehat{\mathcal{B}}(0,t)\right)^{-\frac{\sigma-2\delta}{2 \sigma-2\delta}} d s \\ \lesssim & \|u\|_{X(T)}^p \left(1+\widehat{\mathcal{B}}(0,t)\right)^{-\frac{n + \min \left\{ r^{*}(u_0), r^{*}(u_1) - 2 \delta \right\}+\gamma}{2 \sigma - 2 \delta} p+\frac{n}{4 \sigma - 4 \delta}} \left(1+\widehat{\mathcal{B}}(0,t)\right)^{1-\frac{\sigma-2\delta}{2 \sigma-2\delta}} \\
\lesssim & \|u\|_{X(T)}^p \left(1+\widehat{\mathcal{B}}(0,t)\right)^{-\frac{\frac{n}{2} + \sigma + \min \left\{ r^{*}(u_0), r^{*}(u_1) - 2 \delta \right\} }{2 \sigma - 2 \delta}}
\end{align*}
for $p > \frac{n + 2\omega(\sigma-\delta)}{\omega n + \omega \min \left\{ r^{*}(u_0), r^{*}(u_1) - 2\delta \right\}+\omega \gamma}$.
Therefore, 
$$
\left\| |D|^\sigma u^{\operatorname{non}}(t, \cdot)\right\|_{L^2} \lesssim  \left(1+\widehat{\mathcal{B}}(0,t)\right)^{-\frac{\frac{n}{2} + \sigma + \min \left\{ r^{*}(u_0), r^{*}(u_1) - 2 \delta \right\} }{2 \sigma - 2 \delta}}\|u\|_{X(T)}^p.$$
In the same way we can derive
$$
\left\|u^{\operatorname{nl}}(t, \cdot)\right\|_{L^2} \lesssim \left(1+\widehat{\mathcal{B}}(0,t)\right)^{-\frac{\frac{n}{2} + \min \left\{ r^{*}(u_0), r^{*}(u_1) - 2 \delta \right\} }{2 \sigma - 2 \delta}}\|u\|_{X(T)}^p,$$ and
$$
\left\| u^{\operatorname{nl}}_t(t, \cdot)\right\|_{L^2} \lesssim b(t) \left(1+\widehat{\mathcal{B}}(0,t)\right)^{-\frac{\frac{n}{2} + 2\sigma- 2\delta + \min \left\{ r^{*}(u_0), r^{*}(u_1) - 2 \delta \right\} }{2 \sigma - 2 \delta}}\|u\|_{X(T)}^p .$$The definition of the norm $X(T)$ immediately   implies inequality \eqref{ineq:4.1}.

The subsequent steps are carried out in a similar manner as in the case where $b^{\prime}(t) \geq 0$. The decay rates of the global (in time) solution and its derivatives stated in Theorem \ref{theorem:1.3} are deduced straightforwardly thanks to the special construction of the norm in $X(t)$. In this way,  the proof of Theorem \ref{theorem:1.3} is completed.
\end{proof}

\section{Blow-up results} \label{section_6}
\begin{proof}[Proof of Theorem 1.3]
From assumption (B2), it follows that $b^{\prime}(t) \ge 0$ or $b^{\prime}(t) <0$.
\subsection{The case when ${b^{\prime}(t)}\leq 0$}
\subsubsection{The parameter $\mathbf{\sigma}$ is integer and the parameter $\mathbf{\delta}$ is fractional from $(0, 1)$} \label{section_7.1.1}
First, we introduce the functions $\varphi=\varphi(|x|):=\langle x\rangle^{-n-2 \delta}$ and  $\eta=\eta(t)$ with the following properties:
\begin{itemize}
\item $\eta \in \mathcal{C}_0^{\infty}([0, \infty))$ and $\eta(t)= \begin{cases}1 & \text { if } 0 \leq t \leq \frac{1}{2}, \\ \text { decreasing } & \text { if } \frac{1}{2} \leq t \leq 1, \\ 0 & \text { if } t \geq 1,\end{cases}$
\item \begin{equation} \label{ineq:5.1.000}
\eta^{-\frac{p^{\prime}}{p}}(t)\left(\left|\eta^{\prime}(t)\right|^{p^{\prime}}+\left|\eta^{\prime \prime}(t)\right|^{p^{\prime}}\right) \leq C \quad \text{for any} \quad t \in \left[\frac{1}{2},1 \right],
\end{equation}
where $p^{\prime}$ is the conjugate of $p>1$ and $C$ is a suitable positive constant.
\end{itemize}
Let $R$ be a large parameter in $[0, \infty)$. We define the test function $
\phi_{R}(t, x)=\eta_R(t) \varphi_R(x)$,
where $
\eta_R(t):=\eta\left(R^{-2\sigma+2\delta} t\right)$ and $ \varphi_R(x):=\varphi\left(R^{-1} x\right)$.
We consider $$
I_R:=\displaystyle\int_0^{\infty} \displaystyle\int_{\mathbb{R}^n}|u(t, x)|^p \phi_R(t, x) d x d =\displaystyle\int_0^{R^{2\sigma-2\delta}} \int_{\mathbb{R}^n}|u(t, x)|^p \phi_R(t, x) d x d t,$$and
$
I_{R, t}:=\displaystyle\int_{\frac{R^{2\sigma-2\delta}}{2}}^{R^{2\sigma-2\delta}} \displaystyle\int_{\mathbb{R}^n}|u(t, x)|^p \phi_R(t, x) d x d t$. Suppose that $u=u(t, x)$ is a global (in time) Sobolev solution from $C\left([0, \infty), L^2\right)$ to \eqref{eq:1.0} with $\gamma=0$. Multiplying both sides of  \eqref{eq:1.0} by $\phi_R=\phi_R(t, x)$ and using  integration by parts  we obtain
\begin{align} \label{ineq:testfunction}
0 \leq I_R= & -\int_{\mathbb{R}^n} \left( u_1(x) + b(0) (-\Delta)^\delta u_0(x) \right)  \varphi_R(x) d x \notag\\
& +\int_{\frac{R^{2\sigma-2\delta}}{2}}^{R^{2\sigma-2\delta}} \int_{\mathbb{R}^n} u(t, x) \partial_t^2 \eta_R(t) \varphi_R(x) d x d t+\int_0^{\infty} \int_{\mathbb{R}^n} \eta_R(t) \varphi_R(x)(-\Delta)^\sigma u(t, x) d x d t \notag\\
& -\int_0^{\infty} \int_{\mathbb{R}^n} b(t) \partial_t\eta_R(t) \varphi_R(x) (-\Delta)^\delta u(t, x) d x d t -\int_0^{\infty} \int_{\mathbb{R}^n}  b^{\prime}(t) \eta_R(t) \varphi_R(x) (-\Delta)^\delta u(t, x) d x d t \notag\\
= & :  -\int_{\mathbb{R}^n} u_1(x) \varphi_R(x) d x- \int_{\mathbb{R}^n} b(0) (-\Delta)^\delta u_0(x)  \varphi_R(x) d x + J_1+J_2-J_3-J_4.
\end{align}
Applying H\"older's inequality with $\frac{1}{p}+\frac{1}{p^{\prime}}=1$ we may estimate $J_1$ as follows:
$$
\begin{aligned}
\left|J_1\right| & \leq \int_{\frac{R^{2\sigma-2\delta}}{2}}^{R^{2\sigma-2\delta}} \int_{\mathbb{R}^n}|u(t, x)|\left|\partial_t^2 \eta_R(t)\right| \varphi_R(x) d x d t \\
& \lesssim\left(\int_{\frac{R^{2\sigma-2\delta}}{2}}^{R^{2\sigma-2\delta}} \int_{\mathbb{R}^n}\left|u(t, x) \phi_R^{\frac{1}{p}}(t, x)\right|^p d x d t\right)^{\frac{1}{p}}\left(\int_{\frac{R^{2\sigma-2\delta}}{2}}^{R^{2\sigma-2\delta}} \int_{\mathbb{R}^n}\left|\phi_R^{-\frac{1}{p}}(t, x) \partial_t^2 \eta_R(t) \varphi_R(x)\right|^{p^{\prime}} d x d t\right)^{\frac{1}{p^{\prime}}} \\
& \lesssim I_{R, t}^{\frac{1}{p}}\left(\int_{\frac{R^{2\sigma-2\delta}}{2}}^{R^{2\sigma-2\delta}} \int_{\mathbb{R}^n} \eta_R^{-\frac{p^{\prime}}{p}}(t)\left|\partial_t^2 \eta_R(t)\right|^{p^{\prime}} \varphi_R(x) d x d t\right)^{\frac{1}{p^{\prime}}}.
\end{aligned}
$$
By the change of variables $\tilde{t}:=R^{-2\sigma+2\delta} t$ and $\tilde{x}:=R^{-1} x$, a straight-forward calculation gives
\begin{equation} \label{ineq:J_1}
\left|J_1\right| \lesssim I_{R, t}^{\frac{1}{p}} R^{-4\sigma+4\delta+\frac{n+2\sigma-2\delta}{p^{\prime}}}\left(\int_{\mathbb{R}^n}\langle\tilde{x}\rangle^{-n-2 \delta} d \tilde{x}\right)^{\frac{1}{p^{\prime}}} .
\end{equation}
Here we used $\partial_t^2 \eta_R(t)=R^{-4\sigma+4\delta} \eta^{\prime \prime}(\tilde{t})$ and assumption \eqref{ineq:5.1.000}. Now let us turn to estimate $J_2$ and $J_3$. First, by using $\varphi_R \in H^{2 \sigma}$ and $u \in C\left([0, \infty), L^2\right)$ we apply Lemma \ref{lemma:2.12} to conclude the following relations:
$$
\begin{aligned}
& \int_{\mathbb{R}^n} \varphi_R(x)(-\Delta)^\sigma u(t, x) d x=\int_{\mathbb{R}^n}|\xi|^{2 \sigma} \hat{\varphi}_R(\xi) \widehat{u}(t, \xi) d \xi=\int_{\mathbb{R}^n} u(t, x)(-\Delta)^\sigma \varphi_R(x) d x, \\
& \int_{\mathbb{R}^n} \varphi_R(x)(-\Delta)^\delta u(t, x) d x=\int_{\mathbb{R}^n}|\xi|^{2 \delta} \hat{\varphi}_R(\xi) \widehat{u}(t, \xi) d \xi=\int_{\mathbb{R}^n} u(t, x)(-\Delta)^\delta \varphi_R(x) d x.
\end{aligned}
$$
Hence, we obtain
$$
J_2=\int_0^{\infty} \int_{\mathbb{R}^n} \eta_R(t) \varphi_R(x)(-\Delta)^\sigma u(t, x) d x d t=\int_0^{\infty} \int_{\mathbb{R}^n} \eta_R(t) u(t, x)(-\Delta)^\sigma \varphi_R(x) d x d t,
$$
and
$$
J_3=\int_{\frac{R^{2\sigma-2\delta}}{2}}^{R^{2\sigma-2\delta}} b(t) \int_{\mathbb{R}^n} \partial_t \eta_R(t) \varphi_R(x)(-\Delta)^\delta u(t, x) d x d t=\int_{\frac{R^{2\sigma-2\delta}}{2}}^{R^{2\sigma-2\delta}} b(t) \int_{\mathbb{R}^n} \partial_t \eta_R(t) u(t, x)(-\Delta)^\delta \varphi_R(x) d x d t .
$$
Applying H\"older's inequality again as we estimated $J_1$ leads to
$$
\left|J_2\right| \leq I_R^{\frac{1}{p}}\left(\int_0^{R^{2\sigma-2\delta}} \int_{\mathbb{R}^n} \eta_R(t) \varphi_R^{-\frac{p^{\prime}}{p}}(x)\left|(-\Delta)^\sigma \varphi_R(x)\right|^{p^{\prime}} d x d t\right)^{\frac{1}{p^{\prime}}}.
$$

In order to control the above  integral, we rely on Lemmas \ref{lemma:2.10.000}, \ref{lemma:2.10.0} and \ref{lemma:2.11.0}. Namely, by the change of variables $\tilde{t}:=R^{-2\sigma+2\delta} t$ and $\tilde{x}:=R^{-1} x$ we arrive at
$$
\begin{aligned}
\left|J_2\right| & \lesssim I_R^{\frac{1}{p}} R^{-2 \sigma+\frac{n+2\sigma-2\delta}{p^{\prime}}}\left(\int_0^1 \int_{\mathbb{R}^n} \eta(\tilde{t}) \varphi^{-\frac{p^{\prime}}{p}}(\tilde{x})\left|(-\Delta)^\sigma(\varphi)(\tilde{x})\right|^{p^{\prime}} d \tilde{x} d \tilde{t}\right)^{\frac{1}{p^{\prime}}} \\
& \lesssim I_R^{\frac{1}{p}} R^{-2 \sigma+\frac{n+2\sigma-2\delta}{p^{\prime}}}\left(\int_{\mathbb{R}^n} \varphi^{-\frac{p^{\prime}}{p}}(\tilde{x})\left|(-\Delta)^\sigma(\varphi)(\tilde{x})\right|^{p^{\prime}} d \tilde{x}\right)^{\frac{1}{p^{\prime}}},
\end{aligned}
$$
where we note that $(-\Delta)^\sigma \varphi_R(x)=R^{-2 \sigma}(-\Delta)^\sigma \varphi(\tilde{x})$, since  $\sigma$ is an integer.  Lemma \ref{lemma:2.10.000} implies that
\begin{equation} \label{ineq:J_2}
\left|J_2\right| \lesssim I_R^{\frac{1}{p}} R^{-2 \sigma+\frac{n+2\sigma-2\delta}{p^{\prime}}}\left(\int_{\mathbb{R}^n}\langle\tilde{x}\rangle^{-n-2 \delta-2 \sigma p^{\prime}} d \tilde{x}\right)^{\frac{1}{p^{\prime}}} .
\end{equation}
Next, again by the change of variables $\tilde{t}:=R^{-2\sigma+2\delta} t$ and $\tilde{x}:=R^{-1} x$, using Lemma \ref{lemma:2.11.0} and condition $b(t) < b(0) \lesssim 1$ (since $b^{\prime}(t) < 0, \forall t \geq 0$), we estimate $J_3$ as follows:
$$
\begin{aligned}
\left|J_3\right| & \lesssim I_{R, t}^{\frac{1}{p}} R^{-2\sigma+\frac{n+2\sigma-2\delta}{p^{\prime}}}\left(\int_{\frac{1}{2}}^1 \int_{\mathbb{R}^n} \eta^{-\frac{p^{\prime}}{p}}(\tilde{t})\left|\eta^{\prime}(\tilde{t})\right|^{p^{\prime}} \varphi^{-\frac{p^{\prime}}{p}}(\tilde{x})\left|(-\Delta)^\delta(\varphi)(\tilde{x})\right|^{p^{\prime}} d \tilde{x} d \tilde{t}\right)^{\frac{1}{p^{\prime}}} \\
& \lesssim I_{R, t}^{\frac{1}{p}} R^{-2\sigma+\frac{n+2\sigma-2\delta}{p^{\prime}}}\left(\int_{\mathbb{R}^n} \varphi^{-\frac{p^{\prime}}{p}}(\tilde{x})\left|(-\Delta)^\delta(\varphi)(\tilde{x})\right|^{p^{\prime}} d \tilde{x}\right)^{\frac{1}{p^{\prime}}} .
\end{aligned}
$$
Here we used $\partial_t \eta_R(t)=R^{-2\sigma+2\delta} \eta^{\prime}(\tilde{t})$ and the assumption \eqref{ineq:5.1.000}. To deal with the last integral, we apply Lemma \ref{lemma:2.10.0} with $q=n+2 \delta$ and $\gamma=\delta$, that is, $m=0$ and $s=\delta$, to get
\begin{equation} \label{ineq:J_3}
\left|J_3\right| \lesssim I_{R, t}^{\frac{1}{p}} R^{-2 \sigma+\frac{n+2\sigma-2\delta}{p^{\prime}}}\left(\int_{\mathbb{R}^n}\langle\tilde{x}\rangle^{-n-2 \delta} d \tilde{x}\right)^{\frac{1}{p^{\prime}}} .
\end{equation}
To handle the term $J_4$, since $\varphi_R \in H^{2 \delta}$ and $u \in \mathcal{C}\left([0, \infty), L^2\right)$, by using Lemma \ref{lemma:2.12} we get the  identities
$$
\int_{\mathbb{R}^n} \varphi_R(x)(-\Delta)^\delta u(t, x) d x=\int_{\mathbb{R}^n}|\xi|^{2 \delta} \hat{\varphi}_R(\xi) \widehat{u}(t, \xi) d \xi=\int_{\mathbb{R}^n} u(t, x)(-\Delta)^\delta \varphi_R(x) d x .
$$
Therefore, using assumption $\mathbf{(B2)}$, we get
\begin{align*}
\left|J_4 \right|=& \left| \int_{\frac{R^{2\sigma-2\delta}}{2}}^{R^{2\sigma-2\delta}} \int_{\mathbb{R}^n} b^{\prime}(t) \eta_R(t) \varphi_R(x)(-\Delta)^\delta u(t, x) d x d t \right|\\
=& \left| \int_{\frac{R^{2\sigma-2\delta}}{2}}^{R^{2\sigma-2\delta}} \int_{\mathbb{R}^n} b^{\prime}(t) \eta_R(t) u(t, x) (-\Delta)^\delta \varphi_R(x) d x d t \right| 
\lesssim  \int_{\frac{R^{2\sigma-2\delta}}{2}}^{R^{2\sigma-2\delta}} \int_{\mathbb{R}^n} \frac{b(t)}{1+t} |u(t, x)| \eta_R (t) \left|(-\Delta)^\delta \varphi_R(x)\right| d(x, t).
\end{align*}
An application of H\"older's inequality as  in estimating $J_1$ gives
\begin{align*}
\left|J_4\right| \leq I_{R, t}^{\frac{1}{p}}\left(\int_{\frac{R^{2\sigma-2\delta}}{2}}^{R^{2\sigma-2\delta}} \int_{\mathbb{R}^n} \left(\frac{b(t)}{1+t} \right)^{p^\prime} \eta_R(t) \varphi_R^{-\frac{p^{\prime}}{p}}(x)\left|(-\Delta)^\delta \varphi_R(x)\right|^{p^{\prime}} d x d t\right)^{\frac{1}{p^{\prime}}}.
\end{align*}
To control the above integral we apply  Lemmas \ref{lemma:2.10.0} and \ref{lemma:2.11.0} as the key tools. In particular, by the change of variables $\tilde{x}:=R^{-1} x$ we get the following relation from Lemma \ref{lemma:2.11.0}:
$$
(-\Delta)^\delta \varphi_R(x)=R^{-2 \delta}(-\Delta)^\delta(\varphi)(\tilde{x}) .
$$
By the change of variables $\tilde{t}:=R^{-2\sigma+2\delta} t$ and inequality \eqref{ineq:5.1.000} we achieve that
\begin{align*}
\left|J_4\right| \lesssim & I_{R, t}^{\frac{1}{p}} R^{-2 \delta+\frac{n}{p^{\prime}}} \left(\int_{R^{2\sigma-2\delta} / 2}^{R^{2\sigma-2\delta}} \left(\frac{b(t)}{1+t} \right)^{p^\prime} d t\right)^{\frac{1}{p^{\prime}}} \left(\int_{\mathbb{R}^n} \varphi^{-\frac{p^{\prime}}{p}}(\tilde{x})\left|(-\Delta)^\delta(\varphi)(\tilde{x})\right|^{p^{\prime}} d \tilde{x}\right)^{\frac{1}{p^{\prime}}} \\
\lesssim & I_{R, t}^{\frac{1}{p}} R^{-2 \delta+\frac{n}{p^{\prime}}} \max _{t \in[R^{2\sigma-2\delta} / 2, R^{2\sigma-2\delta}]} \frac{b(t)}{1+t} \left(\int_{R^{2\sigma-2\delta} / 2}^{R^{2\sigma-2\delta}} 1 d t\right)^{\frac{1}{p^{\prime}}} \left(\int_{\mathbb{R}^n} \varphi^{-\frac{p^{\prime}}{p}}(\tilde{x})\left|(-\Delta)^\delta(\varphi)(\tilde{x})\right|^{p^{\prime}} d \tilde{x}\right)^{\frac{1}{p^{\prime}}} \\
\lesssim & I_{R, t}^{\frac{1}{p}} R^{-2 \delta+\frac{n}{p^{\prime}}} \max _{t \in[R^{2\sigma-2\delta} / 2, R^{2\sigma-2\delta}]} \frac{1}{t} \left(\int_{R^{2\sigma-2\delta} / 2}^{R^{2\sigma-2\delta}} 1 d t\right)^{\frac{1}{p^{\prime}}} \left(\int_{\mathbb{R}^n} \varphi^{-\frac{p^{\prime}}{p}}(\tilde{x})\left|(-\Delta)^\delta(\varphi)(\tilde{x})\right|^{p^{\prime}} d \tilde{x}\right)^{\frac{1}{p^{\prime}}} \\
\lesssim & I_{R, t}^{\frac{1}{p}} R^{-2\sigma+\frac{n+2\sigma-2\delta}{p^{\prime}}} \left(\int_{\mathbb{R}^n} \varphi^{-\frac{p^{\prime}}{p}}(\tilde{x})\left|(-\Delta)^\delta(\varphi)(\tilde{x})\right|^{p^{\prime}} d \tilde{x}\right)^{\frac{1}{p^{\prime}}},
\end{align*}
where we have used the relation $
\eta_R^{\prime}(t)=R^{-2\sigma+2\delta} \eta^{\prime}(\tilde{t})$, the condition
$b^{\prime}(t) \leq 0$ and the assumption \eqref{ineq:5.1.000}. 

Next, an application of Lemma \ref{lemma:2.10.0} leads to the  estimate
\begin{equation} \label{ineq:J_4}
\left|J_4\right| \lesssim I_{R, t}^{\frac{1}{p}} R^{-2\sigma+\frac{n+2\sigma-2\delta}{p^{\prime}}} \left(\int_{\mathbb{R}^n}\langle\tilde{x}\rangle^{-n-2 \delta} d \tilde{x}\right)^{\frac{1}{p^{\prime}}}.
\end{equation}
Thanks to  assumption \eqref{condition_u0u1}, there exists a sufficiently large constant $R_0>0$, such that for all $R>R_0$ it holds that
\begin{equation} \label{ineq:cd_b'<0}
\int_{\mathbb{R}^n} \left( u_1(x) + b(0) (-\Delta)^\delta u_0(x) \right)  \varphi_R(x) d x >0.
\end{equation}
Combining the estimates from \eqref{ineq:J_1} to \eqref{ineq:cd_b'<0}, for all $R>R_0$ we obtain that
\begin{equation} \label{ineq:500.1}
I_R^{1-\frac{1}{p}} \lesssim R^{-2 \sigma+\frac{n+2 \sigma-2\delta}{p^{\prime}}}.
\end{equation}
It is clear that assumption \eqref{p_critical} is equivalent to $-2 \sigma+\frac{n+2 \sigma-2\delta}{p^{\prime}} \leq 0$. Now let us consider the subcritical and critical cases.
\begin{enumerate}[a)]
\item If $
p<1+\frac{2 \sigma}{n-2\delta}$, i.e. $-2 \sigma+\frac{n+2 \sigma-2\delta}{p^{\prime}}<0$,
then passing $R \rightarrow \infty$ in \eqref{ineq:500.1} we obtain
$
\int_0^{\infty} \int_{\mathbb{R}^n}^{\infty} |u(t, x)|^p d x d t=0 .
$
This implies that $u \equiv 0$, which contradicts to  \eqref{condition_u0u1}. Therefore, there is no global (in time) Sobolev solution to \eqref{eq:1.0} with $\gamma=0$ in the subcritical case.

\item If
$
p=1+\frac{2 \sigma}{n-2\delta}$,  i.e. $-2 \sigma+\frac{n+2 \sigma-2\delta}{p^{\prime}}=0$,
then from \eqref{ineq:500.1} there exists a positive constant $C_0$ such that
$$
I_R=\int_{Q_{R}} |u(t, x)|^p \phi_R(t, x) d(x, t) \leq C_0
$$
for a sufficiently large $R$. Thus, it follows that
\begin{equation} \label{ineq:500.2}
\int_{\bar{Q}_{R}} |u(t, x)|^p \phi_R(t, x) d(x, t) \rightarrow 0  \;\text { as } \;R \rightarrow \infty,
\end{equation}
where we introduce the notation
$$
Q_{R}:=[0, R^{2\sigma-2\delta}] \times B_R \quad \text { with } \quad B_R:=\left\{x \in \mathbb{R}^n:|x| \leq R\right\} ,
$$
$$
\bar{Q}_{R}:=Q_{R} \backslash\left([0, R^{2\sigma-2\delta} / 2] \times B_{R / 2}\right) \quad \text { with } \quad B_{R / 2}:=\left\{x \in \mathbb{R}^n:|x| \leq R / 2\right\} .
$$
Since
$$
\partial_t^2 \phi_R(t, x)=\partial_t \phi_R(t, x)=(-\Delta)^\sigma \phi_R(t, x)=0 \text { in }\left(\mathbb{R}_{+} \times \mathbb{R}^n\right) \backslash \bar{Q}_{R},
$$
we repeat  the arguments from \eqref{ineq:testfunction} to \eqref{ineq:500.1} to conclude that
\begin{align} \label{ineq:500.3}
&I_R+\int_{\mathbb{R}^n} \left( u_1(x) + b(0) (-\Delta)^\delta u_0(x) \right)  \varphi_R(x) d x \notag \\
\lesssim & \left(\int_{\bar{Q}_{T, R}} |u(t, x)|^p \phi_R(t, x) d(x, t)\right)^{\frac{1}{p}} R^{-2 \sigma+\frac{n+2 \sigma-2\delta}{p^\prime}} =
\left(\int_{\bar{Q}_{T, R}} |u(t, x)|^p \phi_R(t, x) d(x, t)\right)^{\frac{1}{p}},
\end{align}
since $-2 \sigma+\frac{n+2 \sigma-2\delta}{p^{\prime}}=0$. By using \eqref{ineq:500.2}, we let $R \rightarrow \infty$ in \eqref{ineq:500.3} to derive that
$$
\int_0^{\infty} \int_{\mathbb{R}^n} |u(t, x)|^p d x d t+\int_{\mathbb{R}^n} u_1(x) + b(0) (-\Delta)^\delta u_0(x) d x=0 .
$$
This again contradicts to  \eqref{condition_u0u1}. Therefore,  there is no global (in time) Sobolev solution to \eqref{eq:1.0} with $\gamma=0$ in the critical case. 
\end{enumerate}
 Hence, our proof of Theorem \ref{theorem:1.5} in the case when $b^{\prime}(t) < 0$, the parameter $\sigma$ is integer and the parameter $\delta$ is fractional from $(0, 1)$ is now completed.
\subsubsection{\textit{The parameter $\sigma$ is integer and the parameter $\delta$ is fractional from $(1, \frac{\sigma}{2}]$}} \label{section_7.1.2}
We follow ideas from the proof of Section \ref{section_7.1.1}.  Let us denote $s_\delta:=\delta-\lfloor\delta\rfloor$. We introduce test functions $\eta=\eta(t)$ as in Section \ref{section_7.1.1}, and $\varphi=\varphi(x):=\langle x\rangle^{-n-2 s_\delta}$. Repeating the estimates for $J_1$ and $J_2$ as in the proof of Section \ref{section_7.1.1} we can conclude
\begin{equation} \label{ineq:J_1_100}
\left|J_1\right| \lesssim I_{R, t}^{\frac{1}{p}} R^{-4\sigma+4\delta+\frac{n+2\sigma-2\delta}{p^{\prime}}},
\end{equation}
\begin{equation} \label{ineq:J_2_100}
\left|J_2\right| \lesssim I_R^{\frac{1}{p}} R^{-2 \sigma+\frac{n+2\sigma-2\delta}{p^{\prime}}}.
\end{equation}
We turn to estimate $J_3$, where $\delta$ is any fractional number in $(1, \sigma)$. Applying Lemma \ref{lemma:2.12} and H\"older's inequality leads to
$
\left|J_3\right| \leq I_{R, t}^{\frac{1}{p}}\left(\displaystyle\int_{\frac{R^{2\sigma-2\delta}}{2}}^{R^{2\sigma-2\delta}} \int_{\mathbb{R}^n} \eta_R^{-\frac{p^{\prime}}{p}}(t)\left|\partial_t \eta_R(t)\right|^{p^{\prime}} \varphi_R^{-\frac{p^{\prime}}{p}}(x)\left|(-\Delta)^\delta \varphi_R(x)\right|^{p^{\prime}} d x d t\right)^{\frac{1}{p^{\prime}}}$.

Now we rewrite $\delta=m_\delta+s_\delta$, where $m_\delta:=\lfloor\delta\rfloor \geqslant 1$ is integer and $s_\delta$ is a fractional number in $(0,1)$. By Lemma \ref{lemma:2.10.00} we see that
$
(-\Delta)^\delta \varphi_R(x)=(-\Delta)^{s_\delta}\left((-\Delta)^{m_\delta} \varphi_R(x)\right)$. 
By the change of variables $\bar{x}:=R^{-1} x$ we have
$
(-\Delta)^{m_\delta} \varphi_R(x)=R^{-2 m_\delta}(-\Delta)^{m_\delta}(\varphi)(\bar{x}),
$
since $m_\delta$ is an integer. Using  formula (9.18) in the proof of Lemma 9.2.3 in  \cite{Dao2020},  we may rewrite
\begin{align*}
&(-\Delta)^{m_\delta} \varphi_R(x)=\notag\\
&=(-1)^{m_\delta} R^{-2 m_\delta} \prod_{j=0}^{m_\delta-1} (q+2 j)\left(\prod_{j=1}^{m_\delta}(-n+q+2 j)\langle\tilde{x}\rangle^{-q-2 m_\delta} -C_{m_\delta}^1 \prod_{j=2}^{m_\delta}(-n+q+2 j)\left(q+2 m_\delta\right)\langle\tilde{x}\rangle^{-q-2 m_\delta-2}\right. \notag \\
&  \left. +C_{m_\delta}^2 \prod_{j=3}^{m_\delta}(-n+q+2 j)\left(q+2 m_\delta\right)\left(q+2 m_\delta+2\right)\langle\tilde{x}\rangle^{-q-2 m_\delta-4} +\cdots+(-1)^{m_\delta} \prod_{j=0}^{m_\delta-1}\left(q+2 m_\delta+2 j\right)\langle\tilde{x}\rangle^{-q-4 m_\delta}\right),
\end{align*}
where $q:=n+2 s_\delta$. For simplicity, we introduce the following functions:
$\varphi_k(x):=\langle x\rangle^{-q-2 m_\delta-2 k}$ and $ \varphi_{k, R}(x):=\varphi_k\left(R^{-1} x\right)=\langle\tilde{x}\rangle^{-q-2 m_\delta-2 k}$,
with $k=0, \dots, m_\delta$. As a result, by Lemma \ref{lemma:2.11.0} we obtain
\begin{align*}
&(-\Delta)^\delta \varphi_R(x)=\\
&(-1)^{m_\delta} R^{-2 m_\delta} \prod_{j=0}^{m_\delta-1}(q+2 j)\left(\prod_{j=1}^{m_\delta}(-n+q+2 j)(-\Delta)^{s_\delta}\varphi_{0, R}(x)\right. -C_{m_\delta}^1 \prod_{j=2}^{m_\delta}(-n+q+2 j)\left(q+2 m_\delta\right)(-\Delta)^{s_\delta}\varphi_{1, R}(x)+ \\
& C_{m_\delta}^2 \prod_{j=3}^{m_\delta}(-n+q+2 j)\left(q+2 m_\delta\right)\left(q+2 m_\delta+2\right)(-\Delta)^{s_\delta}\varphi_{2, R}(x) \left.+\cdots+(-1)^{m_\delta} \prod_{j=0}^{m_\delta-1}\left(q+2 m_\delta+2 j\right)(-\Delta)^{s_\delta}\varphi_{m_\delta, R}(x)\right) \\
=&(-1)^{m_\delta} R^{-2 m_\delta-2 s_\delta} \prod_{j=0}^{m_\delta-1}(q+2 j)\left(\prod_{j=1}^{m_\delta}(-n+q+2 j)(-\Delta)^{s_\delta}\varphi_0(\tilde{x})\right. -C_{m_\delta}^1 \prod_{j=2}^{m_\delta}(-n+q+2 j)\left(q+2 m_\delta\right)(-\Delta)^{s_\delta}\varphi_1(\tilde{x}) \\
& +C_{m_\delta}^2 \prod_{j=3}^{m_\delta}(-n+q+2 j)\left(q+2 m_\delta\right)\left(q+2 m_\delta+2\right)(-\Delta)^{s_\delta}\varphi_2(\tilde{x})  \left.+\cdots+(-1)^{m_\delta} \prod_{j=0}^{m_\delta-1}\left(q+2 m_\delta+2 j\right)(-\Delta)^{s_\delta}\varphi_{m_\delta}(\tilde{x})\right) \\
=&R^{-2 \delta}(-\Delta)^\delta\varphi(\tilde{x}). &
\end{align*}
Thus, by the change of variables $\tilde{t}:=R^{-2\sigma+2\delta} t$ we obtain
$$
\begin{aligned}
\left|J_3\right| & \lesssim I_{R, t}^{\frac{1}{p}} R^{-2\sigma+\frac{n+2\sigma-2\delta}{p^{\prime}}}\left(\int_{\frac{1}{2}}^1 \int_{\mathbb{R}^n} \eta^{-\frac{p^{\prime}}{p}}(\tilde{t})\left|\eta^{\prime}(\tilde{t})\right|^{p^{\prime}} \varphi^{-\frac{p^{\prime}}{p}}(\tilde{x})\left|(-\Delta)^\delta(\varphi)(\tilde{x})\right|^{p^{\prime}} d \tilde{x} d \tilde{t}\right)^{\frac{1}{p^{\prime}}} \\
& \lesssim I_{R, t}^{\frac{1}{p}} R^{-2\sigma+\frac{n+2\sigma-2\delta}{p^{\prime}}}\left(\int_{\mathbb{R}^n} \varphi^{-\frac{p^{\prime}}{p}}(\tilde{x})\left|(-\Delta)^\delta(\varphi)(\tilde{x})\right|^{p^{\prime}} d \tilde{x}\right)^{\frac{1}{p^{\prime}}},
\end{aligned}
$$
since $\partial_t \eta_R(t)=R^{-2\sigma+2\delta} \eta^{\prime}(\tilde{t})$ and thanks to the assumption \eqref{ineq:5.1.000}. After applying Lemma \ref{lemma:2.10.0} with $q=n+2 s_\delta$ and $\gamma=\delta$, i.e. $m=m_\delta$ and $s=s_\delta$, we  conclude that
\begin{equation} \label{ineq:J_3_100}
\left|J_3\right| \lesssim I_{R, t}^{\frac{1}{p}} R^{-2\sigma+\frac{n+2\sigma-2\delta}{p^{\prime}}}\left(\int_{\mathbb{R}^n}\langle\tilde{x}\rangle^{-n-2 s_\delta} d \tilde{x}\right)^{\frac{1}{p^{\prime}}} \lesssim I_{R, t}^{\frac{1}{p}} R^{-2\sigma+\frac{n+2\sigma-2\delta}{p^{\prime}}} .
\end{equation}
Similarly, we can estimate $J_4$ as
\begin{equation} \label{ineq:J_4_100}
\left|J_4\right| \lesssim I_{R, t}^{\frac{1}{p}} R^{-2\sigma+\frac{n+2\sigma-2\delta}{p^{\prime}}} .
\end{equation}
Finally, combining \eqref{ineq:J_1_100} to \eqref{ineq:J_4_100} and repeating arguments as in Section \ref{section_7.1.1} we complete the proof of Theorem \ref{theorem:1.5} when $b^{\prime}(t) < 0$, the parameter $\sigma$ is integer and the parameter $\delta$ is fractional from $(1, \frac{\sigma}{2}]$.
\subsubsection{\textit{The parameter $\sigma$ is fractional from $(1, \infty)$ and the
parameter $\delta$ is integer}} \label{section_7.1.3}
The proof in that case is similar to the arguments  from Sections \ref{section_7.1.1} and \ref{section_7.1.2}. At first, we denote $s_\sigma:=\sigma-\lfloor\sigma\rfloor$. We introduce test functions $\eta=\eta(t)$ as in Section \ref{section_7.1.1} and $\varphi=\varphi(x):=\langle x\rangle^{-n-2 s_\sigma}$. Then, repeating the proof of Sections \ref{section_7.1.1} and \ref{section_7.1.2} we may conclude the desired statement.
\subsubsection{\textit{The  parameter $\sigma$ is fractional from $(1, \infty)$ and the
parameter $\delta$ is fractional from $(0, 1)$}} \label{section_7.1.4}
We apply the arguments  analogous to  the proofs of Sections \ref{section_7.1.1} and \ref{section_7.1.3}. First, we denote $s_\sigma:=\sigma-\lfloor\sigma\rfloor$. Next, we put $s^*:=\min \left\{s_\sigma, \delta\right\}$. It is obvious that $s^*$ is fractional from $(0,1)$. Let us introduce test functions $\eta=\eta(t)$ as in Section \ref{section_7.1.1} and $\varphi=\varphi(x):=\langle x\rangle^{-n-2 s^*}$. Then, repeating the proof of Sections \ref{section_7.1.1} and \ref{section_7.1.3} we may conclude the required statement.
\subsubsection{\textit{The  parameter $\sigma$ is fractional from $(1, \infty)$ and the
parameter $\delta$ is fractional from $(1, \frac{\sigma}{2})$}}
We proceed to prove in that case in similar manner to arguments used in Sections \ref{section_7.1.1} and \ref{section_7.1.4}.  We denote $s_\sigma:=\sigma-\lfloor\sigma\rfloor$ and $s_\delta:=\delta-\lfloor\delta\rfloor$. Then, we put $s^*:=\min \left\{s_\sigma, s_\delta\right\}$. It is obvious that $s^*$ is fractional from $(0,1)$. Let us introduce test functions $\eta=\eta(t)$ as in Section \ref{section_7.1.1} and $\varphi=\varphi(x):=\langle x\rangle^{-n-2 s^*}$. Repeating the proof of Sections \ref{section_7.1.1} and \ref{section_7.1.4} we may conclude what we wanted to prove.
\subsubsection{\textit{Both parameters $\sigma$ and $\delta$ are integers}}
The main arguments for proving in that case are analogous to proofs of Sections \ref{section_7.1.1} and \ref{section_7.1.3}.  We introduce the test functions $\eta=\eta(t)$ and $\varphi=\varphi(x)$ with the following properties 
\begin{itemize}
\item $\eta \in \mathcal{C}_0^{\infty}([0, \infty))$ and $\eta(t)= \begin{cases}1 & \text { if } 0 \leq t \leq \frac{1}{2}, \\ \text { decreasing } & \text { if } \frac{1}{2} \leq t \leq 1, \\ 0 & \text { if } t \geq 1,\end{cases}$
\item $\varphi \in \mathcal{C}_0^{\infty}\left(\mathbb{R}^n\right)$ and $\varphi(x)= \begin{cases}1 & \text { if }|x| \leq \frac{1}{2}, \\ \text { decreasing } & \text { if } \frac{1}{2} \leq|x| \leq 1, \\ 0 & \text { if }|x| \geq 1,\end{cases}$
\item $
\eta^{-\frac{p^{\prime}}{p}}(t)\left(\left|\eta^{\prime}(t)\right|^{p^{\prime}}+\left|\eta^{\prime \prime}(t)\right|^{p^{\prime}}\right) \leq C, \; \text{for any} \;t \in \left[\frac{1}{2},1 \right]$; 
$\varphi^{-\frac{p^{\prime}}{p}}(x)\left|\Delta^\sigma \varphi(x)\right|^{p^{\prime}} \leq C,\; \text{for any} \; x \in \mathbb{R}^n$ such that $ |x| \in \left[\frac{1}{2},1 \right]$, 
where $p^{\prime}$ is the conjugate of $p$ and $C$ is a suitable positive constant. In addition, we suppose that $\varphi=\varphi(|x|)$ is a radial function satisfying $\varphi(|x|) \leq \varphi(|y|)$ for any $|x| \geq|y|$. 
\end{itemize}
Repeating the proof of Sections \ref{section_7.1.1} and \ref{section_7.1.3} we may conclude the required assessment.
\subsection{The case when $b^{\prime}(t) \ge 0$}
\subsubsection{\textit{The parameter $\sigma$ is integer and the parameter $\delta$ is fractional from $(0, 1)$}} \label{section_7.2.1}
First, we introduce the functions $\varphi=\varphi(|x|):=\langle x\rangle^{-n-2 \delta}$ and  $\eta=\eta(t)$ with  properties:
\begin{itemize}
\item $\eta \in \mathcal{C}_0^{\infty}([0, \infty))$ and $\eta(t)= \begin{cases}1 & \text { if } 0 \leq t \leq \frac{1}{2}, \\ \text { decreasing } & \text { if } \frac{1}{2} \leq t \leq 1, \\ 0 & \text { if } t \geq 1,\end{cases}$
\item \begin{equation} \label{ineq:5.1}
\eta^{-\frac{p^{\prime}}{p}}(t)\left(\left|\eta^{\prime}(t)\right|^{p^{\prime}}+\left|\eta^{\prime \prime}(t)\right|^{p^{\prime}}\right) \leq C \quad \text{for any} \quad t \in \left[\frac{1}{2},1 \right],
\end{equation}
where $p^{\prime}$ is the conjugate of $p>1$ and $C$ is a suitable positive constant.
\end{itemize}
Let $R$ be a large parameter in $[0, \infty)$. We define the  test function
$
\phi_{R}(t, x)=\eta_R(t) \varphi_R(x)$,
where
$
\eta_R(t):=\eta\left(R^{-2\sigma+2\delta} t\right)$ and $ \varphi_R(x):=\varphi\left(R^{-1} x\right)$.
Assume that $g(t)$ is the solution to \eqref{eq:5.0}. After multiplying the first equation in \eqref{eq:1.0} by $g(t)$ and performing a direct calculation, one achieves
\begin{align} \label{eq:5.4}
& \left[g(t) u(t, x)\right]_{tt}+(-\Delta)^\sigma \left[g(t) u(t, x)\right]-2\left[g^{\prime}(t)u(t, x)\right]_t+g^{\prime \prime}(t)u(t, x) \notag \\
& +(-\Delta)^\delta \left[\left(g^{\prime}(t)+1\right) u(t, x)\right]_t - (-\Delta)^\delta \left[ g^{\prime \prime}(t)u(t, x)\right]=g(t)|u(t, x)|^p.
\end{align}
Now we define 
$
I_R:=\displaystyle\int_0^{\infty} \displaystyle\int_{\mathbb{R}^n} g(t)|u(t, x)|^p \phi_{R}(t, x) d x d t=\displaystyle\int_{Q_{R}} g(t)|u(t, x)|^p \phi_{R}(t, x) d(x, t)$, 
where
$
Q_{R}:=[0, R^{2\sigma-2\delta}] \times B_R$ with $B_R:=\left\{x \in \mathbb{R}^n:|x| \leq R\right\}$.
Suppose that $u=u(t, x)$ is a global (in time) Sobolev solution from $\mathcal{C}\left([0, \infty), L^2\right)$ to \eqref{eq:1.0}. Multiplying \eqref{eq:5.4} by $\phi_{R}=\phi_{R}(t, x)$, by integration by parts we obtain
\begin{align} \label{eq:5.5}
&I_R+\int_{B_R}\left(-g^{\prime}(0) u_0(x) + g(0) u_1(x) + (-\Delta)^\delta \left[ g^{\prime}(0) u_0(x) \right] + (-\Delta)^\delta u_0(x)\right) \varphi_R(x) dx \notag \\
= & \int_{Q_{R}} g(t) u(t, x) \eta_R^{\prime \prime}(t) \varphi_R(x) d(x, t)+\int_{Q_{R}} g(t) \eta_R(t) \varphi_R(x) (-\Delta)^\sigma u(t, x) d(x, t) \notag \\
&+ 2 \int_{Q_{R}} g^{\prime}(t) u(t, x) \eta_R^{\prime}(t) \varphi_R(x) d(x, t) + \int_{Q_{R}} g^{\prime \prime}(t) u(t, x) \eta_R(t) \varphi_R(x) d(x, t) \notag \\
&- \int_{Q_{R}} \left(g^{\prime}(t)+1\right) \eta_R^{\prime}(t) \varphi_R(x) (-\Delta)^\delta u(t, x) d(x, t) - \int_{Q_{R}} g^{\prime \prime}(t) \eta_R(t) \varphi_R(x) (-\Delta)^\delta u(t, x) d(x, t) \notag \\
=&: I_{1, R}+I_{2, R}+I_{3, R}+I_{4, R}+I_{5, R}+I_{6, R}.
\end{align}
Using H\"older's inequality with $\frac{1}{p}+\frac{1}{p^{\prime}}=1$ we can proceed further as follows:
\begin{align*}
\left|I_{1, R}\right| \leq & \int_{Q_{R}} g(t)|u(t, x)|\left|\eta_R^{\prime \prime}(t)\right| \varphi_R(x) d(x, t) \\
\lesssim & \left(\int_{Q_{R}}\left|g(t)^{\frac{1}{p}} u(t, x) \phi_{R}^{\frac{1}{p}}(t, x)\right|^p d(x, t)\right)^{\frac{1}{p}}  \left(\int_{Q_{R}}\left|g(t)^{\frac{1}{p^{\prime}}} \phi_{R}^{-\frac{1}{p}}(t, x) \eta_R^{\prime \prime}(t) \varphi_R(x)\right|^{p^{\prime}} d(x, t)\right)^{\frac{1}{p^{\prime}}} \\
\lesssim & I_R^{\frac{1}{p}}\left(\int_{Q_{R}} g(t) \eta_R^{-\frac{p^{\prime}}{p}}(t)\left|\eta_R^{\prime \prime}(t)\right|^{p^{\prime}} \varphi_R(x) d(x, t)\right)^{\frac{1}{p^{\prime}}} .
\end{align*}
By the change of variables $\tilde{t}:=R^{-2\sigma+2\delta} t$ and $\tilde{x}:=R^{-1} x$, we derive that
$
\left|I_{1, R}\right| \lesssim I_R^{\frac{1}{p}} R^{-4\sigma+4\delta+{\frac{n}{p^{\prime}}}}\left(\displaystyle\int_{ R^{2\sigma-2\delta} / 2}^{R^{2\sigma-2\delta}} g(t) d t\right)^{\frac{1}{p^{\prime}}},$
where we have used the relation
$
\eta_R^{\prime \prime}(t)=R^{-4\sigma+4\delta} \eta^{\prime \prime}(\bar{t}).
$
and assumption \eqref{ineq:5.1}. To estimate the above integral,  we use Lemma \ref{lemma:b(t)g(t)} as an important tool. The behavior $g(t) \sim b(t)^{-1}$ from the assertion \textbf{i)} of Lemma \ref{lemma:b(t)g(t)} leads to
\begin{align} \label{ineq:5.6}
\left|I_{1, R}\right| \lesssim  I_R^{\frac{1}{p}} R^{-4\sigma+4\delta+{\frac{n}{p^{\prime}}}} \left(\int_{R^{2\sigma-2\delta} / 2}^{R^{2\sigma-2\delta}} \frac{1}{b(t)} d t\right)^{\frac{1}{p^{\prime}}}
\lesssim  I_R^{\frac{1}{p}} R^{-4\sigma+4\delta+{\frac{n}{p^{\prime}}}} \left(\int_{R^{2\sigma-2\delta} / 2}^{R^{2\sigma-2\delta}} 1 d t\right)^{\frac{1}{p^{\prime}}} \lesssim I_R^{\frac{1}{p}} R^{-4\sigma+4\delta+\frac{n+2 \sigma-2\delta}{p^{\prime}}},
\end{align}
due to the relation $b(t) \geq b(0) > 0$ ( since $b^{\prime}(t) \ge 0, \quad \forall t \geq 0$). By the integration by parts we obtain
\begin{align*}
I_{2, R}=\int_{Q_{R}} g(t) \eta_R(t) \varphi_R(x) (-\Delta)^\sigma u(t, x) d(x, t)=\int_{Q_{R}} g(t) u(t, x) \eta_R(t)(-\Delta)^\sigma \varphi_R(x) d(x, t).
\end{align*}
Applying H\"older's inequality again as in estimating $I_{1, R}$ we get
$$
\left|I_{2, R}\right| \leq I_R^{\frac{1}{p}}\left(\int_{Q_{R}} g(t) \eta_R(t) \varphi_R^{-\frac{p^{\prime}}{p}}(x)\left|(-\Delta)^\sigma \varphi_R(x)\right|^{p^{\prime}} d(x, t)\right)^{\frac{1}{p^{\prime}}}.
$$
In order to control the above two integrals, we rely on  Lemmas \ref{lemma:2.10.000}, \ref{lemma:2.10.0} and \ref{lemma:2.11.0}. Namely, by the change of variables $\tilde{t}:=R^{-2\sigma+2\delta} t$ and $\tilde{x}:=R^{-1} x$, we see that
$$
\begin{aligned}
\left|I_{2, R}\right| & \lesssim I_R^{\frac{1}{p}} R^{-2 \sigma+\frac{n}{p^{\prime}}} \left(\int_{ R^{2\sigma-2\delta} / 2}^{R^{2\sigma-2\delta}} g(t) d t\right)^{\frac{1}{p^{\prime}}} \left(\int_{\mathbb{R}^n} \varphi^{-\frac{p^{\prime}}{p}}(\tilde{x})\left|(-\Delta)^\sigma(\varphi)(\tilde{x})\right|^{p^{\prime}} d \tilde{x}\right)^{\frac{1}{p^{\prime}}} \\
& \lesssim I_R^{\frac{1}{p}} R^{-2 \sigma+\frac{n+2\sigma-2\delta}{p^{\prime}}}\left(\int_{\mathbb{R}^n} \varphi^{-\frac{p^{\prime}}{p}}(\tilde{x})\left|(-\Delta)^\sigma(\varphi)(\tilde{x})\right|^{p^{\prime}} d \tilde{x}\right)^{\frac{1}{p^{\prime}}},
\end{aligned}
$$
noting that
$
(-\Delta)^\sigma \varphi_R(x)=R^{-2 \sigma}(-\Delta)^\sigma \varphi(\tilde{x}),
$
since $\sigma$ is an integer number. Here, we have estimated $\left(\int_{ R^{2\sigma-2\delta} / 2}^{R^{2\sigma-2\delta}} g(t) d t\right)^{\frac{1}{p^{\prime}}}$ similarly to  $I_{1, R}$.  Lemma \ref{lemma:2.10.000} implies the following 
\begin{equation} \label{ineq:5.7}
\left|I_{2, R}\right| \lesssim I_R^{\frac{1}{p}} R^{-2 \sigma+\frac{n+2\sigma-2\delta}{p^{\prime}}}\left(\int_{\mathbb{R}^n}\langle\tilde{x}\rangle^{-n-2 \delta-2 \sigma p^{\prime}} d \tilde{x}\right)^{\frac{1}{p^{\prime}}}.
\end{equation}
Next, in order to  estimate $I_{3, R}$, first we see  from Lemma \ref{lemma:b(t)g(t)} that the term $\left|g^{\prime}(t)\right|$ is bounded. Combining this and the assertion \textbf{i)} of Lemma \ref{lemma:b(t)g(t)}, we can deduce
\begin{equation} \label{lessim_g'}
\left|g^{\prime}(t)\right| \sim b(t) g(t) .
\end{equation}
Hence, it follows immediately that
$
\left|I_{3, R}\right| \lesssim \int_{Q_{R}} b(t) g(t)|u(t, x)|\left|\eta_R^{\prime}(t)\right| \varphi_R(x) d(x, t) .
$
In an analogous procedure as in estimating $I_{1, R}$, we may conclude that
\begin{align*}
\left|I_{3, R}\right| \lesssim & I_R^{\frac{1}{p}} R^{-2\sigma+2\delta+\frac{n}{p^{\prime}}}\left(\int_{R^{2\sigma-2\delta} / 2}^{R^{2\sigma-2\delta}} g(t) b(t)^{p^{\prime}} d t\right)^{\frac{1}{p^{\prime}}} 
\lesssim  I_R^{\frac{1}{p}} R^{-2\sigma+2\delta+\frac{n}{p^{\prime}}} \max _{t \in[R^{2\sigma-2\delta} / 2, R^{2\sigma-2\delta}]} b(t)\left(\int_{R^{2\sigma-2\delta} / 2}^{R^{2\sigma-2\delta}} g(t) d t\right)^{\frac{1}{p^{\prime}}},
\end{align*}
where we have used the relation
$
\eta_R^{\prime}(t)=R^{-2\sigma+2\delta} \eta^{\prime}(\tilde{t}),
$
and assumption \eqref{ineq:5.1}. 

For any $t \in[R^{2\sigma-2\delta} / 2, R^{2\sigma-2\delta}]$, one rewrites
$
g(t)=g(0)+\int_0^t g^{\prime}(s) d s \sim g(0)+C t \sim t,
$
thanks to the assertion \textbf{ii)} of Lemma \ref{lemma:b(t)g(t)}, whenever $R$ is chosen to be sufficiently large. This means that $b(t) \sim t^{-1} \sim R^{-2\sigma+2\delta}$ for any $t \in[R^{2\sigma-2\delta} / 2, R^{2\sigma-2\delta}]$ with a sufficiently large number $R$. From this observation, we obtain
\begin{equation} \label{ineq:5.8}
\left|I_{3, R}\right| \lesssim I_R^{\frac{1}{p}} R^{-4\sigma+4\delta+\frac{n}{p^{\prime}}}\left(\int_{R^{2\sigma-2\delta} / 2}^{R^{2\sigma-2\delta}} \frac{1}{b(t)} d t\right)^{\frac{1}{p^{\prime}}} \lesssim I_R^{\frac{1}{p}} R^{-4\sigma+4\delta+\frac{n+2 \sigma - 2\delta}{p^{\prime}}} .
\end{equation}
To handle the term $I_{4, R}$, we use \eqref{eq:5.0}, \eqref{lessim_g'},  $\mathbf{(B3)}$ and condition $b(t)\geq b(0)>0$ to deduce that
\begin{equation} \label{lessim_g''(t)}
\left|g^{\prime \prime}(t)\right| = \left|b^{\prime}(t)g(t)+b(t)g^{\prime}(t)\right| \sim \left| \left(b^{\prime}(t)+b^2(t) \right)\right|g(t) \lesssim b^2(t) g(t).
\end{equation}
Thus, it follows immediately that
$
\left|I_{4, R}\right| \lesssim \int_{Q_{R}} b^2(t) g(t)|u(t, x)| \eta_R (t) \varphi_R(x) d(x, t).
$
In an analogous procedure as in the estimation of $I_{1, R}$ and $I_{3, R}$, we may conclude that
\begin{align} \label{ineq:5.9}
\left|I_{4, R}\right| \lesssim & I_R^{\frac{1}{p}} R^{\frac{n}{p^{\prime}}}\left(\int_{R^{2\sigma-2\delta} / 2}^{R^{2\sigma-2\delta}} g(t) b(t)^{2p^{\prime}} d t\right)^{\frac{1}{p^{\prime}}} 
\lesssim I_R^{\frac{1}{p}} R^{\frac{n}{p^{\prime}}} \max _{t \in[R^{2\sigma-2\delta} / 2, R^{2\sigma-2\delta}]} b^2(t)\left(\int_{R^{2\sigma-2\delta} / 2}^{R^{2\sigma-2\delta}} \frac{1}{b(t)} d t\right)^{\frac{1}{p^{\prime}}} \notag \\
\lesssim & I_R^{\frac{1}{p}} R^{-4\sigma+4\delta+\frac{n+2\sigma-2\delta}{p^{\prime}}},
\end{align}
where we have used the relation
$b(t) \sim t^{-1} \sim R^{-2\sigma+2\delta}$ for any $t \in[R^{2\sigma-2\delta} / 2, R^{2\sigma-2\delta}]$, with a sufficiently large number $R$ and $b(t) \geq b(0) > 0$.

Let us focus our attention on estimating $I_{5, R}$. Since $\varphi_R \in H^{2 \delta}$ and $u \in \mathcal{C}\left([0, \infty), L^2\right)$, we apply Lemma \ref{lemma:2.12} to get:
$\displaystyle\int_{\mathbb{R}^n} \varphi_R(x)(-\Delta)^\delta u(t, x) d x=\int_{\mathbb{R}^n}|\xi|^{2 \delta} \hat{\varphi}_R(\xi) \widehat{u}(t, \xi) d \xi=\int_{\mathbb{R}^n} u(t, x)(-\Delta)^\delta \varphi_R(x) d x$.
Therefore, 
\begin{align*}
I_{5, R}= \int_0^{\infty} \int_{\mathbb{R}^n} \left(g^\prime(t)+1\right) \eta_R^{\prime}(t) \varphi_R(x)(-\Delta)^\delta u(t, x) d x d t 
= \int_0^{\infty} \int_{\mathbb{R}^n} b(t) g(t) \eta_R^{\prime}(t) u(t, x)(-\Delta)^\delta \varphi_R(x) d x d t,
\end{align*}
where we have used equation \eqref{eq:5.0}. Another application of H\"older's inequality  as in estimating $I_{1, R}$ gives
\begin{align*}
\left|I_{5, R}\right| \leq I_R^{\frac{1}{p}}\left(\int_0^{R^{2\sigma-2\delta}} \int_{\mathbb{R}^n} b(t)^{p^\prime} g(t) \eta_R^{-\frac{p^{\prime}}{p}}(t) \left|\eta_R^{\prime}(t) \right|^{p^\prime} \varphi_R^{-\frac{p^{\prime}}{p}}(x)\left|(-\Delta)^\delta \varphi_R(x)\right|^{p^{\prime}} d x d t\right)^{\frac{1}{p^{\prime}}}.
\end{align*}
To control the above integral we  apply  Lemmas \ref{lemma:2.10.0} and \ref{lemma:2.11.0} as  key tools. In particular, by  the change of variables $\tilde{x}:=R^{-1} x$ we get   from Lemma \ref{lemma:2.11.0} that
$
(-\Delta)^\delta \varphi_R(x)=R^{-2 \delta}(-\Delta)^\delta(\varphi)(\tilde{x}) .
$
By the change of variables $\tilde{t}:=R^{-2\sigma+2\delta} t$ and inequality \eqref{ineq:5.1} we achieve that
\begin{align*}
\left|I_{5, R}\right| \lesssim & I_R^{\frac{1}{p}} R^{-2 \sigma+\frac{n}{p^{\prime}}} \left(\int_{R^{2\sigma-2\delta} / 2}^{R^{2\sigma-2\delta}} g(t) b(t)^{p^{\prime}} d t\right)^{\frac{1}{p^{\prime}}} \left(\int_{\mathbb{R}^n} \varphi^{-\frac{p^{\prime}}{p}}(\tilde{x})\left|(-\Delta)^\delta(\varphi)(\tilde{x})\right|^{p^{\prime}} d \tilde{x}\right)^{\frac{1}{p^{\prime}}} \\
\lesssim & I_R^{\frac{1}{p}} R^{-4 \sigma+2\delta+\frac{n+2\sigma-2\delta}{p^{\prime}}} \left(\int_{\mathbb{R}^n} \varphi^{-\frac{p^{\prime}}{p}}(\tilde{x})\left|(-\Delta)^\delta(\varphi)(\tilde{x})\right|^{p^{\prime}} d \tilde{x}\right)^{\frac{1}{p^{\prime}}},
\end{align*}
where we have estimated $\left(\int_{R^{2\sigma-2\delta} / 2}^{R^{2\sigma-2\delta}} g(t) b(t)^{p^{\prime}} d t\right)^{\frac{1}{p^{\prime}}}$  similarly as for $I_{3, R}$. Now Lemma \ref{lemma:2.10.0} leads to the following estimate
\begin{equation} \label{ineq:5.10}
\left|I_{5, R}\right| \lesssim I_R^{\frac{1}{\tilde{p}}} R^{-4 \sigma+2\delta+\frac{n+2\sigma-2\delta}{p^{\prime}}} \left(\int_{\mathbb{R}^n}\langle\tilde{x}\rangle^{-n-2 \delta} d \tilde{x}\right)^{\frac{1}{p^{\prime}}}.
\end{equation}
To estimate $I_{6, R}$, since $\varphi_R \in H^{2 \delta}$ and $u \in \mathcal{C}\left([0, \infty), L^2\right)$, we apply Lemma \ref{lemma:2.12} to derive that 

$\displaystyle\int_{\mathbb{R}^n} \varphi_R(x)(-\Delta)^\delta u(t, x) d x=\displaystyle\int_{\mathbb{R}^n}|\xi|^{2 \delta} \hat{\varphi}_R(\xi) \widehat{u}(t, \xi) d \xi=\displaystyle\int_{\mathbb{R}^n} u(t, x)(-\Delta)^\delta \varphi_R(x) d x$. 
Therefore, by \eqref{lessim_g''(t)}:
\begin{align*}
\left|I_{6, R}\right|=& \left| \int_0^{\infty} \int_{\mathbb{R}^n} g^{\prime \prime}(t) \eta_R(t) \varphi_R(x)(-\Delta)^\delta u(t, x) d x d t \right|
= \left| \int_0^{\infty} \int_{\mathbb{R}^n} g^{\prime \prime}(t) \eta_R(t) u(t, x) (-\Delta)^\delta \varphi_R(x) d x d t \right| \\
\lesssim & \int_{Q_{R}} b^2(t) g(t)|u(t, x)| \eta_R (t) \left|(-\Delta)^\delta \varphi_R(x)\right| d(x, t).
\end{align*}
Another application of H\"older's inequality  as in estimating $I_{1, R}$ gives
\begin{align*}
\left|I_{6, R}\right| \leq I_R^{\frac{1}{p}}\left(\int_0^{R^{2\sigma-2\delta}} \int_{\mathbb{R}^n} b(t)^{2 p^\prime} g(t) \eta_R(t) \varphi_R^{-\frac{p^{\prime}}{p}}(x)\left|(-\Delta)^\delta \varphi_R(x)\right|^{p^{\prime}} d x d t\right)^{\frac{1}{p^{\prime}}}.
\end{align*}
To control the above integral we apply  Lemmas \ref{lemma:2.10.0} and \ref{lemma:2.11.0}. In particular,  by the change of variables $\tilde{x}:=R^{-1} x$ we get  from Lemma \ref{lemma:2.11.0} that
$
(-\Delta)^\delta \varphi_R(x)=R^{-2 \delta}(-\Delta)^\delta(\varphi)(\tilde{x}) .
$
By the change of variables $\tilde{t}:=R^{-2\sigma+2\delta} t$ and inequality \eqref{ineq:5.1} we see that
\begin{align*}
\left|I_{6, R}\right| \lesssim & I_R^{\frac{1}{p}} R^{-2 \delta+\frac{n}{p^{\prime}}} \left(\int_{R^{2\sigma-2\delta} / 2}^{R^{2\sigma-2\delta}} g(t) b(t)^{2 p^{\prime}} d t\right)^{\frac{1}{p^{\prime}}} \left(\int_{\mathbb{R}^n} \varphi^{-\frac{p^{\prime}}{p}}(\tilde{x})\left|(-\Delta)^\delta(\varphi)(\tilde{x})\right|^{p^{\prime}} d \tilde{x}\right)^{\frac{1}{p^{\prime}}} \\
\lesssim & I_R^{\frac{1}{p}} R^{-4 \sigma+2\delta+\frac{n+2\sigma-2\delta}{p^{\prime}}} \left(\int_{\mathbb{R}^n} \varphi^{-\frac{p^{\prime}}{p}}(\tilde{x})\left|(-\Delta)^\delta(\varphi)(\tilde{x})\right|^{p^{\prime}} d \tilde{x}\right)^{\frac{1}{p^{\prime}}},
\end{align*}
where the term $\left(\int_{R^{2\sigma-2\delta} / 2}^{R^{2\sigma-2\delta}} g(t) b(t)^{2 p^{\prime}} d t\right)^{\frac{1}{p^{\prime}}}$ is estimated similarly as $I_{4, R}$. Now,  Lemma \ref{lemma:2.10.0} implies that \begin{equation} \label{ineq:5.11}
\left|I_{6, R}\right| \lesssim I_R^{\frac{1}{p}} R^{-4 \sigma+2\delta+\frac{n+2\sigma-2\delta}{p^{\prime}}} \left(\int_{\mathbb{R}^n}\langle\tilde{x}\rangle^{-n-2 \delta} d \tilde{x}\right)^{\frac{1}{p^{\prime}}}.
\end{equation}
Plugging \eqref{ineq:5.6}, \eqref{ineq:5.7}, \eqref{ineq:5.8}, \eqref{ineq:5.9}, \eqref{ineq:5.10} and \eqref{ineq:5.11} into \eqref{eq:5.5}, we see that
\begin{align} \label{ineq:5.12}
&I_R + \int_{B_R}\left(-g^{\prime}(0) u_0(x) + g(0) u_1(x) + (-\Delta)^\delta \left[ g^{\prime}(0) u_0(x) \right] + (-\Delta)^\delta u_0(x)\right) \varphi_R(x) dx \notag \\
\lesssim & I_R^{\frac{1}{p}} \left(R^{-2 \sigma+\frac{n+2 \sigma-2\delta}{p^{\prime}}} + R^{-4 \sigma+4\delta+\frac{n+2\sigma-2\delta}{p^{\prime}}} + R^{-4 \sigma+2\delta+\frac{n+2\sigma-2\delta}{p^{\prime}}} \right) \lesssim I_R^{\frac{1}{p}} R^{-2 \sigma+\frac{n+2 \sigma-2\delta}{p^{\prime}}}.
\end{align}
Due to assumption \eqref{condition_u0u1}, there exists a sufficiently large constant $R_0>0$ such that 
\begin{align*}
\int_{B_R}\left(-g^{\prime}(0) u_0(x) + g(0) u_1(x) + (-\Delta)^\delta \left[ g^{\prime}(0) u_0(x) \right] + (-\Delta)^\delta u_0(x)\right) \varphi_R(x) dx >0,
\end{align*}
that is,
\begin{equation} \label{ineq:5.13}
\int_{B_R}\left(-g^{\prime}(0) u_0(x) + g(0) u_1(x) + (-\Delta)^\delta \left[ g^{\prime}(0) u_0(x) \right] + (-\Delta)^\delta u_0(x)\right) \varphi_R(x) dx >0,     
\end{equation} 
for any $R>R_0$. From \eqref{ineq:5.12} and \eqref{ineq:5.13}, one gets
\begin{equation} \label{ineq:5.14}
I_R^{1-\frac{1}{p}} \lesssim R^{-2 \sigma+\frac{n+2 \sigma-2\delta}{p^{\prime}}}.
\end{equation}
It is clear that \eqref{p_critical} is equivalent to $-2 \sigma+\frac{n+2 \sigma-2\delta}{p^{\prime}} \leq 0$. 

Now, by arguments analogous to the case with $b'(t)<0$,  both subcritical case with $p<1+\frac{2 \sigma}{n-2\delta}$, ( i.e. $-2 \sigma+\frac{n+2 \sigma-2\delta}{p^{\prime}}<0,$) and critical case with $p=1+\frac{2 \sigma}{n-2\delta}$,  (i.e. $-2 \sigma+\frac{n+2 \sigma-2\delta}{p^{\prime}}=0$)  lead to a contradiction to assumption \eqref{condition_u0u1}. 

Therefore, there is no global (in time) Sobolev solution to \eqref{eq:1.0} with $\gamma=0$ in these subcritical and critical cases.
Hence, our proof of Theorem \ref{theorem:1.5} in the case when $b^{\prime}(t) >0$, the parameter $\sigma$ is integer and the parameter $\delta$ is fractional from $(0, 1)$ is completed.
\subsubsection{\textit{ $\sigma$ is integer and $\delta$ is fractional from $(1, \frac{\sigma}{2}]$}}\label{section_7.2.2}

The proof in that case is similar to  the proofs of Sections \ref{section_7.2.1} and \ref{section_7.1.2}. We denote $s_\delta:=\delta-\lfloor\delta\rfloor$ and introduce test functions $\eta=\eta(t)$ as in Section \ref{section_7.2.1} and $\varphi=\varphi(x):=\langle x\rangle^{-n-2 s_\delta}$. Repeating the proof of Sections \ref{section_7.2.1} and \ref{section_7.1.2} we may conclude  the desired statement.
\subsubsection{\textit{ $\sigma$ is fractional from $(1, \infty)$ and  $\delta$ is integer}} \label{section_7.2.3}
We apply the arguments analogous to those used in the proofs of Sections \ref{section_7.2.1} and \ref{section_7.2.2}. We denote $s_\sigma:=\sigma-\lfloor\sigma\rfloor$ and introduce test functions $\eta=\eta(t)$ as in Section \ref{section_7.2.1} and $\varphi=\varphi(x):=\langle x\rangle^{-n-2 s_\sigma}$. Repeating the proof of Sections \ref{section_7.2.1} and \ref{section_7.2.2} we may arrive at  the desired statement.
\subsubsection{\textit{$\sigma$ is fractional from $(1, \infty)$ and  $\delta$ is fractional from $(0, 1)$}} \label{section_7.2.4}
We proceed the proof in that case in a similar manner as in proofs of Sections \ref{section_7.2.1} and \ref{section_7.2.3}. We denote $s_\sigma:=\sigma-\lfloor\sigma\rfloor$ and put $s^*:=\min \left\{s_\sigma, \delta\right\}$. It is obvious that $s^*$ is fractional from $(0,1)$. Let us introduce test functions $\eta=\eta(t)$ as in Section \ref{section_7.2.1} and $\varphi=\varphi(x):=\langle x\rangle^{-n-2 s^*}$. Repeating  the proof of Sections \ref{section_7.2.1} and \ref{section_7.2.3}  we get the desired statement.

\subsubsection{\textit{ $\sigma$ is fractional from $(1, \infty)$ and 
 $\delta$ is fractional from $(1, \frac{\sigma}{2})$}}
The main ideas of proving in that case are analogous to those from the proofs of Sections \ref{section_7.2.1} and \ref{section_7.2.4}. We denote $s_\sigma:=\sigma-\lfloor\sigma\rfloor$ and $s_\delta:=\delta-\lfloor\delta\rfloor$ and put $s^*:=\min \left\{s_\sigma, s_\delta\right\}$. It is obvious that $s^*$ is fractional from $(0,1)$. Let us introduce test functions $\eta=\eta(t)$ as in Section \ref{section_7.2.1} and $\varphi=\varphi(x):=\langle x\rangle^{-n-2 s^*}$. Repeating the proof of Sections \ref{section_7.2.1} and \ref{section_7.2.4} we may conclude the required statement.
\subsubsection{\textit{The case when both parameters $\sigma$ and $\delta$ are integers}}
Our proof in that case is analogous to  the proofs of Sections \ref{section_7.2.1} and \ref{section_7.2.3}. We introduce the test functions $\eta=\eta(t)$ and $\varphi=\varphi(x)$ with following properties (see, for example, \cite{DAmbrosio2003, Mitidieri2001}):
\begin{itemize}
\item $\eta \in \mathcal{C}_0^{\infty}([0, \infty))$ and $\eta(t)= \begin{cases}1 & \text { if } 0 \leq t \leq \frac{1}{2}, \\ \text { decreasing } & \text { if } \frac{1}{2} \leq t \leq 1, \\ 0 & \text { if } t \geq 1,\end{cases}$
\item $\varphi \in \mathcal{C}_0^{\infty}\left(\mathbb{R}^n\right)$ and $\varphi(x)= \begin{cases}1 & \text { if }|x| \leq \frac{1}{2}, \\ \text { decreasing } & \text { if } \frac{1}{2} \leq|x| \leq 1, \\ 0 & \text { if }|x| \geq 1,\end{cases}$
\item $
\eta^{-\frac{p^{\prime}}{p}}(t)\left(\left|\eta^{\prime}(t)\right|^{p^{\prime}}+\left|\eta^{\prime \prime}(t)\right|^{p^{\prime}}\right) \leq C,\; \forall t \in \left[\frac{1}{2},1 \right]$, and
$\varphi^{-\frac{p^{\prime}}{p}}(x)\left|\Delta^\sigma \varphi(x)\right|^{p^{\prime}} \leq C,\; \forall x \in \mathbb{R}^n \; \text{such that} \; |x| \in \left[\frac{1}{2},1 \right],
$
where $p^{\prime}$ is the conjugate of $p$ and $C$ is a suitable positive constant. In addition, we suppose that $\varphi=\varphi(|x|)$ is a radial function satisfying $\varphi(|x|) \leq \varphi(|y|)$ for any $|x| \geq|y|$. 
\end{itemize}
Repeating the arguments used in Sections \ref{section_7.2.1} and \ref{section_7.2.3} we  arrive at the desired conclusion. The proof for Theorem \ref{theorem:1.5} is now complete.
\end{proof}

\begin{appendices}
\renewcommand{\thesection}{\Alph{section}}
\section{Decay indicator and decay character}\label{sec2}
The notions of decay character have been introduced by C. Bjorland, M.E. Schonbek, C. J. Niche in \cite{Bjorland2009}, \cite{Ferreira2017}, \cite{Niche2015} to measure the algebraic order of the Fourier transform $\hat{u}(\xi)$ of the function $u \in L^2$ near the point $\xi=0$.

Let $\Lambda=(-\Delta)^{1 / 2}$. We also use $|D|^s$ to denote $\Lambda^s$. First, we recall the notion of decay indicator.
\begin{definition}[Decay indicator]\label{definition_2.1}
For $u_0 \in L^2\left(\mathbb{R}^n\right), s>0$ and $r \in(-n / 2+s, \infty)$, the decay indicator $P_r^s\left(u_0\right)$ corresponding to $\Lambda^s u_0$ is defined by
$$
P_r^s\left(u_0\right)=\lim _{\rho \rightarrow 0} \rho^{-2 r-n} \int_{B(\rho)}|\xi|^{2 s}\left|\widehat{u_0}(\xi)\right|^2 d \xi,
$$
where $B(\rho)$ is the ball centered at origin with radius $\rho$.
\end{definition} 
\begin{definition}[Decay character] \label{definition_2.2} The decay character of $\Lambda^s u_0$, denoted by $r_s^*=r_s^*\left(u_0\right)$ is the unique $r \in(-n / 2+s, \infty)$ such that $0<P_r^s\left(u_0\right)<\infty$, if this number $P_r^s\left(u_0\right)$ exists. If such $P_r^s\left(u_0\right)$ does not exist, we set $r_s^*=-n / 2+s$ when $P_r^s\left(u_0\right)=+\infty$ for all $r \in(-n / 2+s, \infty)$, or $r_s^*=+\infty$ when $P_r^s\left(u_0\right)=0$ for all $r \in(-n / 2+s, \infty)$.
For $s=0$, we denote $P_r^0\left(u_0\right)=P_r\left(u_0\right)$ and $r_0^*=r^*$.
\end{definition} 

\begin{theorem} \label{decay_1}
Let $v_0 \in L^2\left(\mathbb{R}^n\right)$ with the decay character $r^*\left(v_0\right)=r^*$, where $v_0=v(0,x)$, $v=v(t,x)$. Let $v(t)$ be a solution to $v_t=-a(t)(-\Delta)^{\alpha} v$ with the initial datum $v_0$, where $a(t)$ is positive and continuous with $a(t) \notin L^1([0, \infty)) $ and $\alpha$ is a positive constant. If $-\frac{n}{2}<r^*<\infty$, then 
$$
\left(1+\int\limits_{0}^{t} a(\tau) d \tau \right)^{-\frac{2 r^{*} +n}{2 \alpha}} \lesssim \|v(t)\|_{L^2}^2 \lesssim \left(1+\int\limits_{0}^{t} a(\tau) d \tau \right)^{-\frac{2 r^{*} +n}{2 \alpha}}.
$$
\end{theorem}

\begin{proof}
The proof mainly repeats the arguments used in the proof of Theorem 2.10  by Niche and Schonbek in \cite{Niche2015} for the case of coefficient $a(t)\equiv 1$, thus can be omitted here.
\end{proof}
Next, we  introduce the following notions of decay indicators and decay character with negative order.

\begin{definition}[Decay indicators and decay characters with negative orders \cite{ADL2024}] For $u \in L^2\left(\mathbb{R}^n\right), s>0$ and $r \in(-n / 2, \infty)$, the decay indicator $P_r^{-s}(u)$ corresponding to $\Lambda^{-s} u$ is defined by
$$
P_r^{-s}(u)=\lim _{\rho \rightarrow 0} \rho^{-2 r-n} \int_{B(\rho)}|\xi|^{-2 s}|\widehat{u}(\xi)|^2 d \xi,
$$
where $B(\rho)$ is the ball centered at origin with radius $\rho$.
If $v:=\mathcal{F}^{-1}\left(|\xi|^{-s} \widehat{u}\right) \in L^2(\mathbb{R}^n)$ for $s>0$, then we denote $r_{-s}^*(u):=r^*(v)$.
\end{definition}
\begin{theorem} \label{decay_3} Let $v_0 \in H^s\left(\mathbb{R}^n\right), s>0$ with the decay character $r_s^*=r_s^*\left(v_0\right)$. Let $v(t)$ be a solution to $v_t=-a(t)(-\Delta)^{\alpha} v$ with the initial datum $v_0$, where $a(t)$ is positive and continuous with $a(t)\notin L^1([0, \infty))$, $\alpha$ is a positive constant. If $-n / 2<r^*<\infty$, then
$$
\left(1+\int\limits_{s}^{t} a(\tau) d \tau \right)^{-\frac{2 r^{*} +n+2s}{2 \alpha}} \lesssim \|v(t)\|_{\dot{H}^s}^2 \lesssim \left(1+\int\limits_{s}^{t} a(\tau) d \tau \right)^{-\frac{2 r^{*} +n+2s}{2 \alpha}}.
$$
\end{theorem}
\begin{theorem}[\cite{ADL2024}] \label{decay_4}
Let $u \in L^2\left(\mathbb{R}^n\right)$ such that $r^*(u)-s>\frac{n}{2}$ for some $s \geq 0$. Then $v:=\mathcal{F}^{-1}\left(|\xi|^{-s} \widehat{u}\right) \in L^2\left(\mathbb{R}^n\right)$ and $r_{-s}^*(u)=r^*(v)=r^*(u)-s$.
\end{theorem}

\section{Some auxiliary inequalities}
\begin{lemma}[Gronwall's inequality] \label{Gronwall}
 Let $f$ and $h$ be continuous and nonnegative functions defined on $I=[a, b]$ and let $g$ be a continuous, positive and nondecreasing function defined on $I$. Then, the inequality
$$
f(t) \leq g(t)+\int_a^t h(r) f(r) d r, \quad t \in J,
$$
implies that
$$
f(t) \leq g(t) \exp \left(\int_a^t h(r) d r\right), \quad t \in J.
$$
\end{lemma}
\begin{lemma} \label{lemmae^-c}
The estimate
\begin{align*}
\left\|\mathcal{F}^{-1}\left(|\xi|^a e^{-c|\xi|^{2 \kappa} d(t)}\right)(t, \cdot)\right\|_{L^p} \lesssim \left(1+d(t)\right)^{-\frac{n}{2 \kappa}\left(1-\frac{1}{p}\right)-\frac{a}{2 \kappa}}
\end{align*}
holds for any $\kappa \in(0, \infty), \quad p \in[1, \infty], \quad t \geq 0, \quad d(t) \geq 0 \quad \forall t \geq 0, \quad d(t) \rightarrow \infty$ as $t \rightarrow \infty$ and for all dimensions $n \geq 1$. The numbers a and $c$ are,  supposed to be non-negative and positive, respectively.
\end{lemma}
\begin{proof}
Lemma \ref{lemmae^-c} is entirely proved in the same manner as in  \cite[Theorem 24.2.4]{Ebert2018}.
\end{proof}
\begin{lemma}[Fractional Gagliardo-Nirenberg inequality \cite{Hajaiej2011}] \label{Gagliardo}
 Let $q, q_0, q_1 \in(1, \infty)$ and $\kappa \in[0, r)$ with $r>0$. Then, the following inequality holds for all $f \in L^{q_0}\left(\mathbb{R}^n\right) \cap \dot{H}_{q_1}^r\left(\mathbb{R}^n\right)$:
$$
\|f\|_{\dot{H}_q^{\kappa}} \lesssim\|f\|_{L^{q_0}}^{1-\theta}\|f\|_{\dot{H}_{q_1}^r}^\theta,
$$
where $\theta=\theta_{\kappa, r}\left(q, q_0, q_1, n\right):=\left(\frac{1}{q_0}-\frac{1}{q}+\frac{\kappa}{n}\right) /\left(\frac{1}{q_0}-\frac{1}{q_1}+\frac{r}{n}\right)$, and $\theta \in[\kappa / r, 1]$.
\end{lemma}

\end{appendices}

\vskip 0.5cm
	
\noindent{\bf Acknowledgements.} This research is funded by Vietnam National Foundation for Science and Technology Development (NAFOSTED) under grant number 101.02--2023.29. The authors also would like to thank Hanoi National University of Education for providing a fruitful working environment.

\end{document}